\documentclass[onefignum,onetabnum]{siamart251216}
\usepackage{booktabs}   % Permite usar \toprule, \midrule, \bottomrule
\usepackage{multirow}   % Permite agrupar filas en la Opción 2
\usepackage{subcaption} % Permite usar subtable en la Opción 1
\usepackage{amsmath,amssymb,amsfonts,latexsym,stmaryrd}
\usepackage{graphicx,float} 
\usepackage{mathrsfs} 
\usepackage{color}
\usepackage{setspace} % Packages and macros go here
\usepackage{lipsum}
\usepackage{graphicx,float}
\usepackage{color}
\usepackage{epstopdf}
\usepackage{multirow}
\usepackage{url}
\usepackage{diagbox}
\def\ds{\displaystyle}
\def\O{\Omega}

\def\t{\theta}
\def\p{\psi}

\renewcommand\sp{\mathop{\mathrm{Sp}}\nolimits}
\newcommand{\set}[1]{\lbrace #1 \rbrace}

\newcommand{\norm}[1]{\lVert#1\rVert}

\usepackage{booktabs}
\usepackage{float}

\newcommand\bu{\boldsymbol{u}}
\newcommand\bv{\boldsymbol{v}}
\newcommand\bw{\boldsymbol{w}}
\newcommand\bx{\boldsymbol{x}}

\newcommand\bn{\boldsymbol{n}}

\newcommand\mE{\boldsymbol{\mathcal{E}}}

\def\hdel{\widehat{\delta}}%\def\VK{V^{\E}}

\newcommand\bF{\boldsymbol{f}}

\newcommand\bT{\boldsymbol{T}}

\newcommand\bxi{\boldsymbol{\xi}}

\def\CT{{\mathcal T}}

\newcommand\bcI{\boldsymbol{\mathcal{I}}}

\newcommand{\dd}{\texttt{d}}

\newcommand\bsig{\boldsymbol{\sigma}}
\newcommand\btau{\boldsymbol{\tau}}

\newcommand\bphi{\boldsymbol{\varphi}}
\newcommand\brho{\boldsymbol{\rho}}

\newcommand\R{\mathbb{R}}

\renewcommand\H{\mathrm{H}}

\renewcommand\L{\mathrm{L}}
\renewcommand\O{\Omega}
\newcommand\DO{\partial\O}
\newcommand\bdiv{\mathop{\mathbf{div}}\nolimits}

\renewcommand\div{\mathop{\mathrm{div}}\nolimits}

\newcommand\brot{\mathop{\mathbf{rot}}\nolimits}

\newcommand\tr{\mathop{\mathrm{tr}}\nolimits}

\renewcommand\sp{\mathop{\mathrm{sp}}\nolimits}

\renewcommand\t{\mathtt{t}}

\newcommand\LO{\L^2(\O)}

\newcommand\HsO{\H^s(\O)}

\renewcommand\t{\mathtt{t}}

\newcommand{\vertiii}[1]{{\left\vert\kern-0.25ex\left\vert\kern-0.25ex\left\vert #1 
    \right\vert\kern-0.25ex\right\vert\kern-0.25ex\right\vert}}

\newsiamremark{remark}{Remark}
\newsiamremark{hypothesis}{Hypothesis}
\newsiamremark{assumption}{Assumption}
\crefname{hypothesis}{Hypothesis}{Hypotheses}
\newsiamthm{problem}{Problem}
\newsiamthm{claim}{Claim}

\headers{Mixed VEM for the elasticity eigenvalue problem}{Felipe Lepe and  Gonzalo Rivera}

\title{A locking-free mixed virtual element discretization for the two dimensional elasticity eigenvalue problem\thanks{Submitted to the editors DATE.
\funding{The first author was partially supported by
	DIUBB through project 2120173 GI/C Universidad del B\'io-B\'io (Chile).
	The second author was supported by  ANID-Chile through FONDECYT project 1231619 (Chile).  
}}}

\author{Felipe Lepe\thanks{GIMNAP-Departamento de Matem\'atica, Universidad del B\'io - B\'io, Casilla 5-C, Concepci\'on, Chile. \email{flepe@ubiobio.cl}.}
\and Gonzalo Rivera\thanks{Departamento de Ciencias Exactas,
	Universidad de Los Lagos, Casilla 933, Osorno, Chile. \email{gonzalo.rivera@ulagos.cl}.}
}

\usepackage{amsopn}

\ifpdf
\hypersetup{
  pdftitle={Mixed virtual VEM for the elasticity eigenvalue problem},
  pdfauthor={Felipe Lepe and Gonzalo Rivera }
}
\fi
\def\CT{{\mathcal T}}

\begin{document}

\maketitle

% REQUIRED
\begin{abstract}
In this paper, we propose and analyze a mixed virtual element method for the approximation of the eigenvalues and eigenfunctions of the two-dimensional elasticity eigenvalue problem. Under standard assumptions on polygonal meshes, we prove the convergence of the discrete solution operator to its continuous counterpart as the mesh size tends to zero. Relying on the spectral theory of compact operators, we establish the spectral correctness of the method and derive error estimates for both eigenvalues and eigenfunctions.  A series of numerical experiments is presented to support the theoretical analysis. The results confirm the predicted convergence rates and show that the method is locking-free and able to approximate the spectrum accurately, independently of the shape of the polygonal elements in the mesh.
\end{abstract}

% REQUIRED incpor

% REQUIRED

\section{Introduction}\label{sec:intro}
The development of numerical methods to approximate the elasticity equations, and particularly the elasticity eigenvalue problem, has been widely studied in recent years. Different primal and mixed formulations, together with their corresponding discretizations, have emerged to address this problem. A non-exhaustive list of contributions and the references therein is available in \cite{MR3715319,MR2824163,MR3453481,MR2831058,MR4542511,MR3036997,MR4619882}. Here, the load and eigenvalue problems have been analyzed using different techniques with the aim of approximating the solutions while taking into account the inherent difficulties of the linear elasticity equations, particularly those related to the locking phenomenon that the Poisson ratio may produce when it is close to $1/2$. This issue is important to consider, since if the formulations and numerical methods are not properly designed, instabilities may arise in the approximation of the solutions and, in particular, in the spectrum of the eigenvalue problem. The bibliographical discussion reveals that a suitable alternative to handle the linear elasticity eigenvalue problem is provided by mixed formulations, since these methods are precisely capable of avoiding locking. 
The aim of this paper is to present a mixed virtual element method (VEM) to approximate the spectrum of the elasticity operator. The application of VEM to eigenvalue problems is in constant development, as observed in the following non-exhaustive list \cite{MR4742045,MR4884682,MR4136230,MR4488799,MR4951056,MR4597196,MR4497827}, where the VEM has been used to approximate eigenvalues and eigenfunctions of partial differential equations in different contexts. In particular, for the elasticity eigenvalue problem, we highlight \cite{MR4658607,MR4050542} as contributions employing the VEM. Despite  the fact that there is an important amount of contributions where the VEM is analyzed for eigenvalue problems, the literature regarding mixed formulations and their VEM approximations remains scarce where \cite{MR4253143,MR4229296,MR4136230,MR4608374,MR4597196,MR4092604} are our main references on this topic where the research is in constant development. Particularly, mixed formulations based on tensorial fields and their respective tensorial version of VEM has only an application for the Stokes eigenvalue problem \cite{MR4229296}. This paper reveals that a pseudostress tensor is a suitable alternative to approximate the spectrum of the eigenvalue problem since with only this unknown, it is possible to recover other quantities as the velocity and pressure via a simple  post-process. The task now is to analyze a similar approach for the elasticity eigenvalue problem.

In this paper we study a mixed VEM for the elasticity eigenproblem considering the formulation provided in \cite{MR3453481}. This mixed formulation, originally proposed for the load problem, was analyzed for the eigenvalue problem in \cite{MR4570534} under a finite element scheme, where a priori error estimates were obtained. This formulation shows that it is possible to compute the spectrum of the elasticity eigenproblem while avoiding the locking that arises in the primal formulation. Hence, we expect that for the VEM the results can be as accurate as those obtained with the finite element scheme, while allowing the use of general polygonal meshes to discretize different domains. Here the main difference with the FEM setting is that under the VEM approach, some terms are computable via the degrees of freedom whereas other terms are not and hence, polynomial projections are needed. For the mixed load problem, \cite{MR3860570} describes rigorously this analysis and some of these results are necessary for the spectral problem. On one hand, the regularity of the load problems implies the compactness of the solution operator, which in our case, will be compact and as a consequence, the well known theory for compact operators of \cite{MR1115235} is applied in order to conclude convergence of the method. However, for the error estimates of eigenvalues and eigenfunctions, the regularity of the load problem is not enough, because the eigenfunctions are of different nature and hence, these error estimates must be concluded with the corresponding regularity. Other important (and not minor) issue to address is the influence of the Lam\'e constants on the spectrum. As is stated in \cite{MR3962898}, the Lam\'e constant $\lambda$ makes the spectrum of the elasticity problem change, implying to consider families of spectrums depending on $\lambda$. In fact, when this Lamé constant blows up (namely, when the Poisson ratio tends to 1/2), one is forced to analyze the limit problem, which coincides with the Stokes eigenvalue problem.  It is for this reason that a mixed formulation is a suitable way to analyze the nearly incompressible and perfectly incompressible regimes for elasticity. Now our aim is to extend the numerical analysis of the mixed finite element method  proposed in \cite{MR4570534},  to a suitable  VEM that corresponds to an extension of the classic VEM spaces that approximate vectors to tensor fields.

Let us precise that our analysis is focused in ordinary materials, which means that the Poisson ratio is positive. It is well known that a negative Poisson ratio is admissible in linear elasticity when auxetic materials are considered (see \cite{Zhou2025} for instance). This assumption is  applicable with no problems in load problems, but in eigenvalue problems this is not completely clear, in the sense that the spectrum of nearly incompressible materials has a well established limit on the perfectly incompressible case which coincides with the Stokes spectrum (see \cite{MR4570534,MR3962898} for further details). This holds for ordinary materials (i.e., positive Poisson ratio with  $1/2$ as upper bound) but, for the best of our knowledge, not for other limits such as auxetic materials. 

%Now we are interested in investigating if the VEM is capable to capture the spectrum of elasticity with polygonal meshes and what effects cause the stabilization of the VEM with respect to the spurious eigenmodes and the convergence orders. 
%Nevertheless, these results are not sufficient for the study of eigenvalue problems since, the eigenfunctions not necessarily have the same regularity of the solutions of the load problem.   }
\subsection{Outline}
The paper is organized as follows: In Section \ref{sec:model_problem} we present the model problem and summarize important results from the continuous problem related to the solution operators, regularity of the load and spectral solutions, and the corresponding spectral characterization. Section \ref{sec:spec_app} contains the discrete elements  of the paper, where we begin by introducing the VEM and the ingredients for its analysis, namely virtual spaces, degrees of freedom, discrete bilinear forms, and approximation properties. In Section \ref{sec:conv} begins the analysis of the VEM method, where we prove convergence and error estimates. Additionally we prove that the method is spurious free. Finally in Section \ref{sec:numerics} we report a series of numerical tests to assess the performance of the method when the spectrum is approximated. These tests are  performed in two dimensions for different Poisson ratios and domains.

\subsection{Notations and preliminaries}
%Throughout this work, we will use notations that will allow a smoother reading of the content. Let us set these notations. We denote $\mathbb{R}^{2 \times 2}$ as  the space of square real matrices of order $d$, where $\mathbb{I}:=\left(\delta_{i j}\right) \in \mathbb{R}^{2 \times 2}$ denotes the identity matrix. Given $\boldsymbol{\mathbb{A}}:=\left(A_{i j}\right), \boldsymbol{\mathbb{B}}:=\left(B_{i j}\right) \in \mathbb{R}^{2 \times 2}$, we define 
%$$
% \boldsymbol{\mathbb{A}}: \boldsymbol{\mathbb{B}}:=\sum_{i, j=1}^d A_{i j} \overline{B}_{i j},
%$$
%as the tensorial product between $\boldsymbol{\mathbb{A}}$ and $\boldsymbol{\mathbb{B}}$. The entry $\overline{B}_{i j}$ represent the complex conjugate of $B_{i j}$. Similarly, given two vectors $\bs:=(s_i), \br:=(r_i) \in  \mathbb{C}^d$, we define the products
%$$
%\bs\cdot\br:= \sum_{i=1}^d s_i \overline{r}_i, \qquad \bs \otimes \br : = \bs \overline{\br}^\texttt{t}  = \sum_{i=1}^d\sum_{j=1}^d s_i \overline{r}_j,
%%\begin{pmatrix}
%%	s_1 \overline{r}_1 & s_1 \overline{r}_2 & s_1 \overline{r}_3\\
%%	s_2 \overline{r}_1 & s_2 \overline{r}_2 & s_2 \overline{r}_3\\
%%	s_3 \overline{r}_1 & s_3 \overline{r}_2 & s_3 \overline{r}_3
%%\end{pmatrix},
%$$
%as the dot and dyadic product in $\mathbb{C}$, respectively, where $(\cdot)^\texttt{t}$ denotes the transpose operator.

Given
any Hilbert space $X$, let $X^2$ and $\boldsymbol{\mathcal{X}}$ denote, respectively,
the space of vectors and tensors  with
entries in $X$. In particular, $\mathbf{I}$ is the identity matrix of
$\R^{2\times 2}$, and $\mathbf{0}$ denotes a generic null vector or tensor. 
Given $\btau:=(\tau_{ij})$ and $\bsig:=(\sigma_{ij})\in\R^{2\times 2}$, 
we define, as usual, the transpose tensor $\btau^{\t}:=(\tau_{ji})$, 
the trace $\tr\btau:=\sum_{i=1}^2\tau_{ii}$ and the
tensor inner product $\btau:\bsig:=\sum_{i,j=1}^2\tau_{ij}\sigma_{ij}$. 

Let $\O$ be a polygonal Lipschitz bounded domain of $\R^2$ with
boundary $\DO$. For $s\geq 0$, $\norm{\cdot}_{s,\O}$ stands indistinctly
for the norm of the Hilbertian Sobolev spaces $\HsO$, $\HsO^2$ or
$\boldsymbol{\mathcal{H}}^s(\O)$ for scalar, vectorial and tensorial fields, respectively, with the convention $\H^0(\O):=\LO$, $\H^0(\O)^{2}=\L^2(\O)^2$ and $\boldsymbol{\mathcal{H}}^0(\O):=\boldsymbol{\mathcal{L}}^2(\O)$. We also define the Hilbert space 
$\boldsymbol{\mathcal{H}}(\bdiv;\O):=\set{\btau\in\boldsymbol{\mathcal{L}}^2(\O):\ \bdiv\btau\in\L^2(\O)^2}$, whose norm
is given by $\norm{\btau}^2_{\bdiv,\O}
:=\norm{\btau}_{0,\O}^2+\norm{\bdiv\btau}^2_{0,\O}$.

\section{The model problem}
\label{sec:model_problem}
Let us consider an open and bounded  domain $\O\subset\mathbb{R}^2$, with Lipschitz boundary $\partial\O$. The study of the paper concerns approximating the physical spectrum of elastic structures, formalized as:
%We are interested in the elasticity eigenvalue problem: 
Find $\kappa\in\mathbb{R}^+$ and the pair $(\bsig,\bu)$ such that
\begin{equation}\label{def:elast_system}
\left\{
\begin{array}{rcll}
\bsig & = & 2\mu\boldsymbol{\varepsilon}(\bu)+\lambda\tr(\boldsymbol{\varepsilon}(\bu))\mathbf{I}&  \text{ in } \quad \Omega, \\
\bdiv\bsig & = & -\kappa\bu & \text{ in } \quad \Omega, \\
\bu & = & \mathbf{0} & \text{ on } \quad \partial\Omega,
\end{array}
\right.
\end{equation}
where $\mathbf{I}\in\mathbb{R}^{2\times 2}$ is the identity matrix,  $\bu$ represents the displacement of the elastic structure, $\boldsymbol{\varepsilon}(\bu)$ represents the tensor of small deformations  given by $\boldsymbol{\varepsilon}(\bu):=\frac{1}{2}(\nabla\bu+(\nabla\bu)^{\texttt{t}})$, where $\texttt{t}$ is the transpose operator, $\bsig$ is the symmetric Cauchy  tensor, and $\lambda$ and $\mu$ are  the positive Lam\'e constants defined by
\begin{equation*}
\lambda:=\frac{E\nu}{(1+\nu)(1-2\nu)}\quad\text{and}\quad\mu:=\frac{E}{2(1+\nu)},
\end{equation*}
where $E$ is the Young's modulus and $\nu$ is the Poisson ratio. Let us remark that on our paper, we are considering ordinary materials, implying that $\nu\in [0,1/2]$.

We describe the model problem of interest. From the first equation of
\eqref{def:elast_system} we have the following identity
\begin{equation*}
\bdiv\bsig=2\mu\bdiv\boldsymbol{\varepsilon}(\bu)+\lambda\nabla\div\bu=\mu\Delta\bu+(\lambda+\mu)\nabla\div\bu.
\end{equation*}
%This allows us  to rewrite \eqref{def:elast_system} as follows
%\begin{equation*}\label{def:elast_system_reduced}
%\left\{
%\begin{array}{rcll}
%\mu\Delta\bu+(\lambda+\mu)\nabla\div\bu & = & -\kappa\bu &  \text{ in } \quad \Omega, \\
%\bu & = & \mathbf{0} & \text{ on } \quad \partial\Omega.
%\end{array}
%\right.
%\end{equation*}
Now we introduce the so called pseudostress tensor (see \cite{MR3453481} for instance), defined by 
\begin{equation*}
\label{eq:pseudo}
\boldsymbol{\rho}:=\mu\nabla\bu+(\lambda+\mu)\div\bu\mathbf{I}=\mu\nabla\bu+(\lambda+\mu)\tr(\nabla\bu)\mathbf{I}.
\end{equation*}
Hence, with this new unknown at hand, and following the same steps of \cite{MR4570534}, we obtain the following system 
%Observe that $\bdiv\bsig=\bdiv\boldsymbol{\rho}$. Hence, we have the following formulation where the pseudostress and the displacement are the main unknowns: Find $\kappa\in\mathbb{R}$ and $(\boldsymbol{\rho}, \bu)$ such that
%\begin{equation}\label{def:elast_system_rho_1}
%\left\{
%\begin{array}{rcll}
%\boldsymbol{\rho} & = &\mu\nabla\bu+(\lambda+\mu)\tr(\nabla\bu)\mathbb{I}&  \text{ in } \quad \Omega, \\
%\bdiv\boldsymbol{\rho} & = & -\kappa\bu & \text{ in } \quad \Omega, \\
%\bu & = & \mathbf{0} & \text{ on } \quad \partial\Omega.
%\end{array}
%\right.
%\end{equation}
%Moreover, the following identity holds (see \cite[Section 2]{MR3453481} for details)
%\begin{equation*}
%\displaystyle\frac{1}{\mu}\left\{\boldsymbol{\rho}-\frac{\lambda+\mu}{2\lambda+3\mu}\tr(\boldsymbol{\rho})\mathbb{I} \right\}=\nabla\bu,
%\end{equation*}
%which, replacing in \eqref{def:elast_system_rho_1}, leads to the following  eigenvalue problem
\begin{equation*}\label{def:elast_system_rho}
\left\{
\begin{array}{rcll}
\displaystyle\frac{1}{\mu}\left\{\boldsymbol{\rho}-\frac{\lambda+\mu}{2\lambda+3\mu}\tr(\boldsymbol{\rho})\mathbf{I} \right\}& = &\nabla\bu&  \text{ in } \quad \Omega, \\
\bdiv\boldsymbol{\rho} & = & -\kappa\bu & \text{ in } \quad \Omega, \\
\bu & = & \mathbf{0} & \text{ on } \quad \partial\Omega,
\end{array}
\right.
\end{equation*}
whose variational formulation is: Find $\kappa\in\mathbb{R}^+$ and $(\boldsymbol{0},\boldsymbol{0})\neq (\boldsymbol{\rho},\bu)\in\boldsymbol{\mathcal{H}}\times \mathbf{Q}$ such that

\begin{equation}\label{def:spectral_1}
\left\{
\begin{array}{rcll}
a(\boldsymbol{\rho},\btau)+b(\btau,\bu) & = &0&  \forall\btau\in\boldsymbol{\mathcal{H}}, \\
b(\boldsymbol{\rho},\bv)& = & -\kappa c(\bu,\bv) &  \forall\bv\in \mathbf{Q},
\end{array}
\right.
\end{equation}
where $\boldsymbol{\mathcal{H}}:=\boldsymbol{\mathcal{H}}(\bdiv;\O)$ and  $\mathbf{Q}:=\L^2(\O)^2$. 

Let us introduce the bilinear form  $a:\boldsymbol{\mathcal{H}}\times\boldsymbol{\mathcal{H}}\rightarrow\mathbb{R}$,  defined by
\begin{equation*}
\displaystyle a(\bxi,\btau):=\frac{1}{\mu}\int_{\Omega}\bxi:\btau-\frac{\lambda+\mu}{\mu(2\lambda+3\mu)}\int_{\O}\tr(\bxi)\tr(\btau)\quad\forall\bxi,\btau\in\boldsymbol{\mathcal{H}},
\end{equation*}
together with the  forms $b:\boldsymbol{\mathcal{H}}\times\mathbf{Q}\rightarrow\mathbb{R}$ and $c:\mathbf{Q}\times\mathbf{Q}\rightarrow\mathbb{R}$ which are  defined, respectively,  by
\begin{equation*}
b(\btau,\bv):=\int_{\O}\bv\cdot\bdiv\btau\quad\forall\btau\in\boldsymbol{\mathcal{H}},\,\,\forall\bv\in \mathbf{Q},\quad\text{and}\quad c(\bw,\bv):=\int_{\O}\bw\cdot\bv\quad\forall\bw,\bv\in\mathbf{Q}.
\end{equation*}

It is not difficult to check that  these bilinear forms are continuous. Now, let  $\btau\in\boldsymbol{\mathcal{H}}$. The deviator of $\btau$, defined by $\btau^{\texttt{d}}:=\btau-\frac{1}{2}\tr(\btau)\mathbf{I}$, allows us to rewrite $a(\cdot,\cdot)$ as follows
\begin{equation*}
\label{eq:identity_a}
\displaystyle a(\bxi,\btau):=\frac{1}{\mu}\int_{\Omega}\bxi^{\texttt{d}}:\btau^{\texttt{d}}+\frac{1}{4\lambda+6\mu}\int_{\O}\tr(\bxi)\tr(\btau)\quad\forall\bxi,\btau\in\boldsymbol{\mathcal{H}},
\end{equation*}
where the presence of the coefficient $\lambda$ is now on the denominator, leading to estimates that will be independent of this parameter.

For the analysis  of the mixed formulation  it is convenient to decompose the space $\boldsymbol{\mathcal{H}}$ as follows $\boldsymbol{\mathcal{H}}:=\boldsymbol{\mathcal{H}}_0\oplus \R \mathbf{I}$, where
\begin{equation*}
\boldsymbol{\mathcal{H}}_0:=\left\{\btau\in\boldsymbol{\mathcal{H}}\,:\,\int_{\O}\tr(\btau)=0\right\}.
\end{equation*}

Note that for any $\bxi\in\boldsymbol{\mathcal{H}}$ there exists a unique $\bxi_{0}\in \boldsymbol{\mathcal{H}}_0$ and a constant $d:=\dfrac{1}{2|\O|}\displaystyle\int_{\O}\tr(\bxi)\in\R$ such that $\bxi=\bxi_{0}+d \, \mathbf{I}$. Hence, problem \eqref{def:spectral_1} is now stated as follows (see \cite{MR4570534}): Find $\kappa\in\mathbb{R^+}$ and $(\boldsymbol{0},\boldsymbol{0})\neq (\boldsymbol{\rho},\bu)\in\boldsymbol{\mathcal{H}}_0\times \mathbf{Q}$, such that
\begin{equation}\label{def:spectral_2}
\left\{
\begin{array}{rcll}
a(\boldsymbol{\rho},\btau)+b(\btau,\bu) & = &0&  \forall\btau\in\boldsymbol{\mathcal{H}}_0, \\
b(\boldsymbol{\rho},\bv)& = & -\kappa c(\bu,\bv) &  \forall\bv\in \mathbf{Q}.
\end{array}
\right.
\end{equation}

Now we introduce the following solution operators associated to the displacement and pseudostress, respectively
%Let us  invoke the  following technical result (see \cite[Proposition 9.1.1]{MR3097958})
\begin{equation*}
\bT:\mathbf{Q}\rightarrow\mathbf{Q},\quad
           \boldsymbol{f}\mapsto \bT\boldsymbol{f}:=\widehat{\bu}, \quad \boldsymbol{\mathcal{S}}:\mathbf{Q}\rightarrow\boldsymbol{\mathcal{H}}_{0},\quad
           \boldsymbol{f}\mapsto \boldsymbol{\mathcal{S}}\boldsymbol{f}:=\widehat{\boldsymbol{\rho}},
\end{equation*}
where the pair $(\widehat{\boldsymbol{\rho}}, \widehat{\bu})$ is the solution of the following source problem
 \begin{equation}\label{def:sourcel_1}
\left\{
\begin{array}{rcll}
a(\widehat{\boldsymbol{\rho}},\btau)+b(\btau,\widehat{\bu}) & = &0&  \forall\btau\in\boldsymbol{\mathcal{H}}_{0}, \\
b(\widehat{\boldsymbol{\rho}},\bv)& = & - c(\boldsymbol{f},\bv)&  \forall\bv\in \mathbf{Q}.
\end{array}
\right.
\end{equation}

Let us recall the  following technical result that allows to control  the $\L^2$-norm of the deviator of a tensor $\btau\in\boldsymbol{\mathcal{H}}_0$ (see \cite[Proposition 9.1.1]{MR3097958})
\begin{equation}
\label{eq:des_deviator}
\|\btau\|_{0,\O}^2\leq C \left(\|\btau^{\dd}\|_{0,\O}^2+\|\bdiv\btau\|_{0,\O}^2\right)\quad\forall\btau\in\boldsymbol{\mathcal{H}}_0.
\end{equation}
Hence, if  $\boldsymbol{\mathcal{V}}:=\{\btau\in\boldsymbol{\mathcal{H}}_0\,:\,\,\bdiv\btau=\boldsymbol{0}\}$  is the kernel of $b(\cdot,\cdot)$, with \eqref{eq:des_deviator} at hand,  it is easy to check the existence of a constant  $\alpha=\dfrac{1}{\mu}>0$ independent of $\lambda$ but depending on $\mu$, such that 
\begin{equation*}
a(\btau,\btau)\geq \alpha\|\btau\|_{\bdiv,\O}^2\quad\forall\btau\in\boldsymbol{\mathcal{V}}.
\end{equation*}

On the other hand, the  following inf-sup condition for bilinear $b(\cdot,\cdot)$ holds  (see \cite[Theorem 2.1]{MR3453481})
\begin{equation*}
\label{eq:in_sup_cont}
\displaystyle\sup_{\boldsymbol{0}\neq\btau\in \boldsymbol{\mathcal{H}}_{0}}\frac{b(\btau,\bv)}{\|\btau\|_{\bdiv,\O}}\geq\beta\|\bv\|_{0,\O}\quad\forall\bv\in \mathbf{Q},
\end{equation*}
where the constant $\beta>0$ is independent of the Lam\'e constant $\lambda$. Hence, applying the Babuška–Brezzi theory (see \cite{MR3097958} for a comprehensive review), the source problem \eqref{def:sourcel_1} is well posed, implying the existence of a constant $C>0$, independent of $\lambda$, such that the following continuous dependence on the data   for the solution holds
\begin{equation*}
\label{eq_cotasupfuente}
\|\widehat{\boldsymbol{\rho}}\|_{\bdiv,\O}+\|\widehat{\bu}\|_{0,\O}\leq C \|\boldsymbol{f}\|_{0,\O}.
\end{equation*}

  The well posedness of problem \eqref{def:sourcel_1} implies that  operators $\bT$ and $\boldsymbol{\mathcal{S}}$ are  well defined. Furthermore, it is straightforward to verify that $\bT$ is self-adjoint with respect to the $\L^2(\O)$ inner product  and that $(\kappa, (\boldsymbol{\rho},\bu))\in\mathbb{R}^+\times \boldsymbol{\mathcal{H}}_{0}\times \mathbf{Q}$ solves \eqref{def:spectral_1} if and only if $(\zeta,\bu)$ is an eigenpair of $\bT$ with $\zeta=1/\kappa$ and $\kappa\neq 0$.
%i.e.,
%\begin{equation*}
%\ds \bT\bu=\zeta\bu\quad\text{with}\,\,\zeta:=\frac{1}{\kappa} \quad\text{and}\,\,\zeta\neq 0.
%\end{equation*}

%On the other hand, if $\kappa\bT_{\lambda}\bu =\bu$, $\bu\neq  \mathbf{0}$, there exists $\boldsymbol{\rho} = \bS(\kappa\bu)\in \mathbb{H}_{0}$ such that $(\boldsymbol{\rho},\bu,\kappa) $is an eigenpair of \eqref{def:spectral_H0}. Thus, 

\subsection{Regularity of solutions and spectral characterization}

Now we present an additional regularity for the  solutions of the source problem \eqref{def:sourcel_1} and the  eigenfunctions of $\bT$. These regularity properties are derived from the classic regularity results for linear elasticity provided by \cite{MR840970}, together with a standard bootstrap argument.  The following lemma contains these results. 
\begin{lemma}
\label{lmm:add_eigen}
The following statements hold:
 \begin{itemize}
\item For all $\boldsymbol{f}\in \mathbf{Q}$, if $(\widehat{\brho},\widehat{u})\in\boldsymbol{\mathcal{H}}_0\times\mathbf{Q}$ solves problem \eqref{def:sourcel_1}, there exists $\widehat{s}\in (0,1]$ and $\widehat{C}>0$ that depend on $\O$ and $\lambda$  such that  $\bu\in\H^{1+s}(\Omega)^2$, $\brho\in \boldsymbol{\mathcal{H}}^s(\Omega)$ and $\bdiv\widehat{\brho}\in \H^{1+s}(\O)^2$ and satisfies the following estimate
$$\|\widehat{\brho}\|_{s,\O}+\|\bdiv\widehat{\brho}\|_{1+s,\O}+\|\widehat{\bu}\|_{1+s,\O}\leq \widehat{C}\|\boldsymbol{f}\|_{0,\O}\qquad \forall s\in (0,\widehat{s}).$$
\item Let $\bu$ be an eigenfunction of $\bT$ associated to an eigenvalue $\kappa$. Then, 
for all $r>0$, we have that $\bu\in\H^{1+r}(\Omega)^2$. Also, there exists a constant $\widehat{C}>0$ which in principle depends on $\lambda$ and on the eigenvalue $\kappa$ such that
\begin{equation*}
\|\brho\|_{r,\O}+\|\bdiv\brho\|_{1+r,\O}+\|\bu\|_{1+r,\O}\leq \widehat{C}\|\bu\|_{0,\O}.
\end{equation*}
\end{itemize}
\end{lemma}

%\begin{remark} \label{daniel0}
%Observe that Lemma \ref{lmm:add_eigen}, in conjunction with the first equation of \eqref{def:elast_system_rho_1}, implies immediately that $\boldsymbol{\rho}\in  \mathbb{H}^{s}(\O)$. On the other hand, for the divergence term, it is enough to consider the second equation in \eqref{def:elast_system_rho} to deduce that $\bdiv\boldsymbol{\rho}\in \H^{1+s}(\O)^{2}$.
%\end{remark} 

\begin{remark}
We need to precise an important point related to the regularity exponents and involved constants. In \cite[Section 2]{MR3962898} the authors discussed about the dependency  $\lambda$ on the regularity exponents $s$ and $r$ of the previous lemma, claiming that this dependency is not completely evident, according to the numerical tests that they report. This phenomenon was also observed in our numerical experiments (cf. Section \ref{sec:numerics}), where for the perfectly incompressible elasticity eigenvalue problem ($\lambda=\infty$), our method is capable to attain the expected convergence orders, showing that the Lam\'e constant $\lambda$ does not affect the results. For the best of the author's knowledge, a rigorous  mathematical proof for this fact is not available on the literature, Hence, and motivated by this fact, we introduce the following assumption that will hold along our paper.  
\begin{assumption}
Constants $s$, $r$, and constant $\widehat{C}$ involved in Lemma \ref{lmm:add_eigen} are independent of $\lambda$.
\end{assumption}
\end{remark}
\begin{remark}
The perfectly incompressible for problem \eqref{def:spectral_2} has been also studied in \cite[subsection 2.1]{MR4570534} where has been proved that for formulation \eqref{def:spectral_2} its limits corresponds to the Stokes eigenvalue problem. Hence, the spectrum of \eqref{def:spectral_2} converges to the spectrum of Stokes as $\lambda\rightarrow+\infty$. See \cite[Lemma 2.2 and Theorem 2.2]{MR4570534}  for further details. This fact is important since it provides a locking-free numerical method.
\end{remark}

The regularity of Lemma \ref{lmm:add_eigen} allows us to conclude that  $\bT$ is a compact operator, thanks to  the compact inclusion $\H^{1+s}(\O)^2\hookrightarrow \mathbf{Q}$. Now, since $\bT$ is compact and selfadjoint,
the following spectral characterization of $\bT$ holds.
\begin{theorem}[Spectral characterization of $\bT$]
\label{thrm:spec_char_T}
The spectrum of $\bT$ satisfies $\sp(\bT)=\{0\}\cup\{\zeta_k\}_{k\in\mathbb{N}}$, where $\{\zeta_k\}_{k\in\mathbb{N}}$
is a sequence of real positive eigenvalues which converges to zero, repeated according their respective multiplicities. 
\end{theorem}

To end this section, we define the following continuous bilinear form
\begin{equation*}
A((\boldsymbol{\rho},\bu),(\btau,\bv)):=a(\boldsymbol{\rho},\btau)+b(\btau,\bu)+b(\boldsymbol{\rho},\bv)\quad\forall\boldsymbol{\rho},\btau\in\boldsymbol{\mathcal{H}},\,\,\forall\bv,\bu\in\mathbf{Q},
\end{equation*}
With this form $A(\cdot,\cdot)$ at hand,  problems \eqref{def:spectral_1} and \eqref{def:sourcel_1} can be rewritten in a more simple way in order to prove some results.

\section{The virtual element approximation}
\label{sec:spec_app}

In this section, we present the formulation and the theoretical analysis of a virtual element method for approximating the solutions of problem \eqref{def:spectral_1}.  To this end, we introduce a set of assumptions and definitions that are required to conduct the analysis within the virtual element framework
%In this section, we propose and analyze a virtual element method
%to approximate the solutions of problem \eqref{eq:fv_2}. 
%To do this task, we need to introduce some assumptions and definitions
%to operate in the virtual element setting.

\subsection{Assumptions on the mesh and virtual spaces}
Consider the family of meshes  $\{\mathcal{T}_h(\O)\}_{h>0}$ that partition the polygonal domain  $\O$ into
polygonal elements $E$, generated via the procedure outlined below.

The  mesh $\mathcal{T}_h$ satisfies \emph{regularity properties} defined by positive constants 
$ c, \chi$, ensuring:
\begin{enumerate}
\item Any edge $e$ on the boundary $\partial E$ has a length $h_e \ge c \: h_E$,
with $h_E$ as the diameter of element $E$; 
\item Every polygonal element $E$ is star-shaped relative to a ball of radius $\chi h_E$.
\end{enumerate}

%pseudo-stress, velocity and the scalar function $\var$, 
 \subsubsection{Virtual space to approximate the pseudostress}

This subsection summarizes the construction of the required Virtual Element Method (VEM) spaces, with further technical details available in \cite[Subsection 3.2]{MR3629152}.

The construction begins by considering a geometric object $\mathcal{O}$ of dimension $d\in\{1, 2\}$, characterized by its barycenter  $x_{\mathcal{O}}$ and diameter  $h_{\mathcal{O}}$, A set of normalized monomials is defined on  $\mathcal{O}$, where the cardinality of the set is  $k+1$ for $d=1$ and $(k+1)(k+2)/2$ for $d=2$). This set is formally expressed as: %We   summarize the construction of the VEM spaces that we require. For further details we refer to  \cite[Subsection 3.2]{MR3629152}.  Given a  geometric object $\mathcal{O}$ of dimension $d\in\{1, 2\}$, as an edge or an element, with barycenter $x_{\mathcal{O}}$ and diameter  $h_{\mathcal{O}}$, we consider the following set of normalized monomials on $\mathcal{O}$ (of dimension $k+1$ for $d=1$ and $(k+1)(k+2)/2$ for $d=2$)
\begin{equation*}
\label{eq:monomio}
\mathcal{M}_{k}(\mathcal{O}):=\left\{ q \;\ \Big| \;\ q:=\left(\dfrac{\mathbf{x}-\mathbf{x}_{\mathcal{O}}}{h_{\mathcal{O}}}\right)^{\boldsymbol{\alpha}} \text{ for } \alpha\in \mathbb{N}^{d} \text{ with } |\boldsymbol{\alpha}|\leq k\right\},
\end{equation*}
In this context $\boldsymbol{\alpha}:=(\alpha_{1},...,\alpha_{d})$ denotes a multi-index , where, $\boldsymbol{\alpha}:=\alpha_{1}+...+\alpha_{d}$ and $\mathbf{x}^{\boldsymbol{\alpha}}:=x_{1}^{\alpha_{1}}...x_{d}^{\alpha_{d}}$. It is easy to
check that $\mathcal{M}_{k}(\mathcal{O})$ is a  basis of $\textsc{P}_k(\mathcal{O})$ (see \cite{MR2997471} for more details). Here, $\textsc{P}_k(\mathcal{O})$ denotes the space of polynomials of degree at most $k$ on $\mathcal{O}$.
 Throughout this work, we adopt the following notation for polynomial spaces:  $\textsc{P}_k(\mathcal{O})$ for scalars,  $\textsc{P}_k(\mathcal{O})^2$  for vectors, and $\textbf{\textsc{P}}_{k}(\mathcal{O})$ for tensors.

For each integer $k\geq 0$ and for each $E\in\mathcal{T}_h$, we introduce the following local virtual element space of order $k$:
\begin{multline*}
\label{eq:global_space}
\boldsymbol{\mathcal{H}}_{h}^{E}:=\{\btau\in\boldsymbol{\mathcal{H}}(\bdiv;E)\cap\boldsymbol{\mathcal{H}}(\brot;E) : \btau\bn\in\textsc{P}_{k}(e)^2\quad\forall e\subset\partial E,\\
\quad\bdiv\btau\in \textsc{P}_{k}(E)^2, \quad\textbf{\text{rot}}\,\btau\in \textsc{P}_{k-1}(E)^2 \},
\end{multline*}
and let us define the \begin{align*}
\boldsymbol{\mathcal{M}}_k(\mathcal{O}):=\left\{(q,0)^{t}: q\in\mathcal{M}_{k}(\mathcal{O})\right\}\cup \left\{(0,q)^{t}:q\in\mathcal{M}_{k}(\mathcal{O}) \right\}, \\
\end{align*}
On the other hand, for every element $E$ we have that:
$$\textsc{P}_{k}(E)^2=\mathcal{H}_k(E)\oplus\mathcal{H}_k^{\bot}(E),$$
where $\mathcal{H}_k(E)=\nabla \textsc{P}_{k+1}(E)$ and $\mathcal{H}_k^{\bot}(E)=$ the $\L^2(E)^2$ orthogonal of $\mathcal{H}_k(E)$ in $\textsc{P}_{k}(E)^2$. Thus, we can define 
 \begin{equation*}
\ds \boldsymbol{\mathcal{H}}_k^{\bot}:=\left\{\begin{pmatrix}\mathbf{q}\\\boldsymbol{0}\end{pmatrix} \,:\mathbf{q}\in\mathcal{H}_k^{\bot}(E)\right\}\cup\left\{\begin{pmatrix}\boldsymbol{0}\\\mathbf{q}\end{pmatrix} \,:\mathbf{q}\in\mathcal{H}_k^{\bot}(E)\right\}.
\end{equation*}

It can be easily verified that the space  $\boldsymbol{\mathcal{H}}_{h}^{E}$  is unisolvent respect to the following degrees of freedom
\begin{align}
\ds\int_e\btau\boldsymbol{n}\cdot\boldsymbol{q}\qquad&\forall\boldsymbol{q}\in\boldsymbol{\mathcal{M}}_k(e)\quad\forall\,\text{edge}\,\,e\in\mathcal{T}_h,\label{eq:dof_normal}\\
\int_E\btau:\nabla\boldsymbol{q}\qquad&\forall\boldsymbol{q}\in\boldsymbol{\mathcal{M}}_{k}(E)\backslash\{(1,0)^{t},(0,1)^{t}\}\quad\forall E\in\mathcal{T}_h,\label{eq:dof_grad}\\
\int_E\btau :\boldsymbol{\rho}\qquad&\forall\boldsymbol{\rho}\in\boldsymbol{\mathcal{H}}_{k}^{\bot}(E)\quad\forall E\in\mathcal{T}_h.\label{eq:dof_rot}
\end{align}

Now, for every decomposition $ \mathcal{T}_h$ of $\O$ into simple polygons  $E$, we define  the global virtual element space
$$\boldsymbol{\mathcal{H}}_h:=\{\btau_h\in \boldsymbol{\mathcal{H}}: \btau_h|_E\in \boldsymbol{\mathcal{H}}_h^E\,\, \text{for all}\,\, E\in \mathcal{T}_h\},$$
and the discrete counterpart of  $\boldsymbol{\mathcal{H}}_0$ is defined by 
$$\boldsymbol{\mathcal{H}}_{0,h}:=\left\{\btau_h\in \boldsymbol{\mathcal{H}}_h: \int_\O\tr(\btau_h)=0\right\}.$$

Let us observe that if  $\bx\in\textsc{P}_{k}(E)^2$, the term $\displaystyle\int_\O\tr(\btau)$ is computable from the degrees of freedom as follows
\begin{multline*}\int_\O\tr(\btau)=\sum_{E\in\CT_h}\int_{E}\tr(\btau)=\sum_{E\in\CT_h}\int_{E}\btau: \mathbf{I}=\sum_{E\in\CT_h}\int_{E}
\btau: \nabla \bx\\
=\sum_{E\in\CT_h}\left(-\int_{E}\bdiv\btau\cdot \bx+\int_{\partial E}\btau\bn\cdot \bx\right).
\end{multline*}
\subsection{Discrete bilinear forms}
Let us begin by writing the bilinear forms elementwise as follows
\begin{equation*}
a(\boldsymbol{\rho},\btau)=\sum_{E\in\CT_h}a^E(\boldsymbol{\rho},\btau)=\sum_{E\in\CT_h}\frac{1}{\mu}\int_E\boldsymbol{\rho}^{\texttt{d}}:\btau^{\texttt{d}}+\frac{1}{4\lambda+6\mu}\int_E\tr(\boldsymbol{\rho})\tr(\btau),
\end{equation*}
\begin{equation*}
b(\btau,\bv)=\sum_{E\in\CT_h}b^E(\btau,\bv)=\sum_{E\in\CT_h}\int_E\bv\cdot\bdiv\btau.
\end{equation*}
and
\begin{equation*}
c(\bw,\bv)=\sum_{E\in\CT_h}c^E(\bw,\bv)=\sum_{E\in\CT_h}\int_E\bw\cdot\bv.
\end{equation*}
We denote by $\mathcal{P}_{k}^h$ the $\L^2$-orthogonal projection onto $\mathbf{Q}_h$ as follows
$$\mathcal{P}_{k}^h:\L^2(\O)^{2}\rightarrow \mathbf{Q}_{h}:=\{\boldsymbol{q}\in \L^2(\O)^{2}\;\ \boldsymbol{q}|_{E}\in\textsc{P}_{k}(E)^2\quad \forall E\in \mathcal{T}_h\},$$ 
which for $\bv\in\L^2(\O)^{2}$ satisfies
\begin{equation*}
\ds\int_E\mathcal{P}_{k}^h(\bv)\cdot\boldsymbol{q}=\int_E\bv\cdot\boldsymbol{q}\quad\forall E\in\mathcal{T}_h, \quad\forall\boldsymbol{q}\in\textsc{P}_{k}(E)^2.
\end{equation*}
Observe that $\mathcal{P}_{k}^h(\bv)|_E=\mathcal{P}_{k}^h(\bv|_E)$. %Moreover, $\mathcal{P}_{k}^h(\tau)$ is explicitly computable for every $\tau\in\mathbf{W}_h^E$ using only its degree of freedom  \eqref{eq:dof_normal0}--\eqref{eq:dof_rot0}. 
%
%On the other hand, for   $\boldsymbol{q}\in \textbf{\textsc{P}}_{k}(E)$ we know that  there exist unique $\boldsymbol{q}^{\bot}\in(\nabla \textsc{P}_{k+1}(E)^{\bot|_{\textbf{\textsc{P}}_k(E)}}\cap \textbf{\textsc{P}}_{k}(E)) $ and  $\widetilde{q}\in  \textsc{P}_{k+1}(E)$, which is unique up to a constant, such that $\boldsymbol{q}=\boldsymbol{q}^{\bot}+\nabla\widetilde{q}$, (see \cite{{MR3629152}} for more details). Then 
%\begin{equation*}
%\ds\int_E\tau_{h}\cdot\boldsymbol{q}=\int_{E}\tau_{h}\cdot\boldsymbol{q}^{\bot}+ \int_{E}\tau_{h}\cdot\nabla\widetilde{q}=\int_{E}\tau_{h}\cdot\boldsymbol{q}^{\bot}-\int_{E}\widetilde{q}\div\tau+\int_{\partial E}\tau\cdot\n\widetilde{q}.
%\end{equation*}
Also, for $m\in\{0,1,\ldots,k+1\}$. For this operator, the following a priori error estimate holds (see \cite{MR3614887} for further details),
\begin{equation*}
\|\bv-\mathcal{P}_k^h\bv\|_{0,E}\leq C h_E^m|\bv|_{m,E},\qquad\forall\bv\in\H^m(E)^{2}, \,\forall E\in\mathcal{T}_h.
\end{equation*}

In the spirit of \cite[Subsections 4.1 and 4.4]{MR3629152}, we assign to each polygonal element $E \in \mathcal{T}_h$ the $\boldsymbol{\mathcal{L}}^2(E)$-orthogonal projector $\Pi_h^E : \boldsymbol{\mathcal{L}}^2(E) \rightarrow \textbf{\textsc{P}}_k(E)$, whose key properties are recalled below from \cite[Section 4]{MR3614887}.%Now, inspired by the analysis presented in \cite[Subection 4.1 and 4.4]{MR3629152}, for each $E\in\mathcal{T}_h$ we define  $\Pi_h ^E:\FL{\boldsymbol{\mathcal{L}}}^2(E)\rightarrow\GR{\textbf{\textsc{P}}_{k}(E)}$ be the $\FL{\boldsymbol{\mathcal{L}}}^2(E)$-orthogonal projector, which satisfies the following properties (see \cite[Section 4]{MR3614887}):%\FL{that is to say $\Pi_h ^E$ stands for the operator $\mathcal{P}_{k}$ acting along each row of a tensor in $\mathbb{L}^2(E)$.
%$$\int_{E}\Pi_h ^E(\btau):\boldsymbol{\rho}=\int_{E}\btau:\boldsymbol{\rho},$$
 %such that satisfies the following properties}
\begin{itemize}
\item[(A.1)] The  stability estimate 
\begin{equation*}
\|\Pi_h^E(\btau)\|_{0,E}\leq\|\btau\|_{0,E},\qquad\forall\btau\in\boldsymbol{\mathcal{H}}(\bdiv;E),
\end{equation*}
\item[(A.2)] $\ds\int_E\big(\Pi_h^E\btau \big)^{\texttt{d}}:\big(\Pi_h^E\boldsymbol{\rho} \big)^{\texttt{d}}=\int_E\big(\Pi_h^E\btau \big)^{\texttt{d}}:\boldsymbol{\rho}^{\texttt{d}}$, for all $\btau,\boldsymbol{\rho}\in\boldsymbol{\mathcal{H}}(\bdiv;E)$, and
\item[(A.3)] given an integer $0\leq m\leq k+1$, there exists  $C>0$, independent of $E$, such that 
\begin{equation*}
\| \btau-\Pi_h^E\btau\|_{0,E}\leq C h_E^m|\btau|_{m,E},\qquad\forall\btau\in\boldsymbol{\mathcal{H}}^m(E).
\end{equation*}
%\GR{or at least for all $\btau\in\bcW_{\nabla\curl}^{m}(E)$ where
%\begin{equation*}
%\bcW_{\nabla\curl}^{m}(E):=\{ \boldsymbol{\xi}\in\mathbb{H}^m(\O)\,:\, \boldsymbol{\xi}^{\tD}=\nabla\curl\boldsymbol{w}\,\,\text{for some}\,\,\boldsymbol{w}\in\H^{r+2}(E)^{2} \}
%\end{equation*}}
\end{itemize}

Let us remark that (A.1) may change if we assume (A.2). Indeed, if $\btau\in\boldsymbol{\mathcal{H}}(\bdiv;E)$, elementary algebraic manipulations reveal
\begin{multline*}
\|\Pi_h^E\btau\|_{0,E}^2=\int_E(\Pi_h^E\btau)^{\texttt{d}}:(\Pi_h^E\btau)^{\texttt{d}}+\frac{1}{4}\|\tr(\Pi_h^E\btau)\|_{0,E}^2\\
=\int_E(\Pi_h^E\btau)^{\texttt{d}}:\btau^{\texttt{d}}+\frac{1}{4}\|\tr(\Pi_h^E\btau)\|_{0,E}^2\\
\leq\int_E\left(\Pi_h^E\btau-\frac{1}{2}\tr(\Pi_h^E\btau)\mathbf{I}\right):\btau^{\texttt{d}}+\frac{1}{2}\|\Pi_h^E\btau\|_{0,E}^2\\
 \leq \|\Pi_h^E\btau\|_{0,E}\|\btau^{\texttt{d}}\|_{0,E}+\frac{\sqrt{2}}{2}\|\Pi_h^E\btau\|_{0,E}\|\btau^{\texttt{d}}\|_{0,E}+\frac{1}{2}\|\Pi_h^E\btau\|_{0,E}^2,
\end{multline*}
implying that $\displaystyle\frac{1}{2}\|\Pi_h^E\btau\|\leq \left(\frac{2+\sqrt{2}}{2} \right)\|\btau^{\texttt{d}}\|_{0,E}$ where, using again the definition of $\btau^{\texttt{d}}$ we obtain
\begin{equation*}
\|\Pi_h^E\btau\|_{0,E}\leq\left(3+\frac{3\sqrt{2}}{2} \right)\|\btau\|_{0,E}\quad\forall\btau\in\boldsymbol{\mathcal{H}}(\bdiv;E).
\end{equation*}

%From , (A.1) and (A.3) are straightforward, meanwhile (A.2) follows from the fact that if $\boldsymbol{\rho}\in\mathbb{P}_k(E)$ it holds that $\boldsymbol{\rho}^{\tD}\in\mathbb{P}_k(E)$ and, for all $\boldsymbol{\rho}, \btau\in\mathbb{P}_k(E$), we have
%\begin{align*}
%\int_E\left(\Pi_k^E\btau\right)^{\tD}:\left(\Pi_k^E\boldsymbol{\rho}\right)^{\tD}:=\int_E\Pi_k^E\boldsymbol{\rho}:\left(\Pi_k^E\btau\right)^{\tD}=\int_E \boldsymbol{\rho}:\left(\Pi_k^E\btau\right)^{\tD}=\int_E\left(\Pi_k^E\btau\right)^{\tD}:\boldsymbol{\rho}^{\tD}.
%\end{align*}
%\FL{Let us note that A.2) not necessarily implies A.1). Indeed, if $\btau\in\boldsymbol{\mathcal{H}}(\bdiv;E)$, simple calculations reveal
%\begin{multline*}
%\|\Pi_h\|
%\end{multline*}
%
% }

In addition, let $S^E(\cdot,\cdot)$ represent any bilinear form endowed with symmetric positive definiteness and verifying%On the other hand, let $S^E(\cdot,\cdot)$ be any symmetric positive definite bilinear form that satisfies
\begin{equation*}
\label{eq:s_stable}
\dfrac{c_0}\mu{}\int_E\btau_h:\btau_h\leq S^E(\btau_h,\btau_h)\leq \dfrac{c_1}{\mu}\int_E\btau_h:\btau_h,\qquad\forall\btau_h\in\boldsymbol{\mathcal{H}}_h^E.
\end{equation*}
Both $c_0$ and $c_1$ are positive constants governed by the mesh regularity properties. Let us check the first inequality. Let $\btau_h\in\boldsymbol{\mathcal{H}}_h^E$. Then
\begin{multline*}
\int_E\btau_h:\btau_h=\frac{\mu}{\mu}\int_E\left(\btau_h^{\texttt{d}}+\frac{1}{2}\tr(\btau_h)\mathbf{I}\right):\left(\btau_h^{\texttt{d}}+\frac{1}{2}\tr(\btau_h)\mathbf{I}\right)\\
=\mu\left[\frac{1}{\mu}\int_E\btau_h^{\texttt{d}}:\btau_h^{\texttt{d}}+\frac{1}{4\mu}\int_E \tr(\btau_h)\tr(\btau_h)\right]
\leq\mu\left[\frac{1}{\mu}\int_E\btau_h^{\texttt{d}}:\btau_h^{\texttt{d}}+\frac{1}{2\mu}\int_E \btau_h:\btau_h\right]\\
\mu\left[\frac{1}{\mu}\int_E\sum_{i=1}^{N_E}\alpha_i\boldsymbol{\varphi}_i:\sum_{i=1}^{N_E}\alpha_i\boldsymbol{\varphi}_i \right]
=\mu\left[\left(\frac{1}{\mu}+\frac{1}{2\mu}\right)\sum_{i=1}^{N_E}\int_E (m_{i,E}(\btau_h)\boldsymbol{\varphi}_i)^2\right]\\
=\mu\sum_{i=1}^{N_E}m_{i,E}(\btau_h)^2\frac{3}{2\mu}\|\boldsymbol{\varphi}_i\|_{0,E}^2\leq \mu\sum_{i=1}^{N_E}m_{i,E}(\btau_h)^2\frac{C}{\mu},
\end{multline*}
where we have used the existence of a constant $C>0$ such that $\|\boldsymbol{\varphi}_i\|_{0,E}\leq C$ (see \cite[Lemma 15]{MR3660301}). Then, since $S^E(\btau_h,\btau_h):=\displaystyle\sum_{i=1}^{N_E}m_{i,E}(\btau_h)^2$ and defining 
$c_0:=(C/\mu)^{-1}$, obtain the desired estimate.
\begin{remark}
We need to mention that the form $S^E(\btau_h,\btau_h):=\displaystyle\sum_{i=1}^{N_E}m_{i,E}(\btau_h)^2$ is the natural way to  be chosen. However, since we are focusing on the eigenvalue problem, it is well known that stabilization terms on the numerical methods may introduce spurious eigenvalues if are not correctly chosen (see for instance \cite{MR4177014,MR4253143,zbMATH06443639}). For this reason, we will scale $S^E(\cdot,\cdot)$ with a parameter $\gamma_E$ that we will change in order to observe its influence on the computation of the spectrum (cf. Section \ref{sec:numerics}).
\end{remark}

 Building upon this, the following local bilinear form is defined on each element%where $c_0$ and $c_1$ are positive constants depending on the mesh assumptions. Then, for each element we define the bilinear form
\begin{multline*}
\displaystyle a_h ^E(\bsig_h,\btau_h):=\dfrac{1}{\mu}\int_E\left(\Pi_h^E\bsig_h\right)^{\texttt{d}}:\left(\Pi_h^E\btau_h\right)^{\texttt{d}}+\frac{1}{4\lambda+6\mu}\int_E\tr\left(\Pi_h^E\bsig_h \right)\tr\left(\Pi_h^E\btau_h \right)\\
+S^E\left(\bsig_h-\Pi_h^E
\bsig_h,\btau_h-\Pi_h^E
\btau_h\right),
\end{multline*}
for  $\bsig_h,\btau_h\in\boldsymbol{\mathcal{H}}_h^E$ and, in a natural way, bilinear form $a^E_h(\cdot,\cdot)$ is written as the sum of the local contributions as follows
\begin{align*}
%\displaystyle a_h(\bsig,\btau)=\sum_{E\in\mathcal{T}_h}a_h^E(\bsig,\btau)&:=\sum_{E\in\mathcal{T}_h}\int_{E}\bdiv\bsig\cdot\bdiv\btau+b_h(\bsig,\btau),\\
\displaystyle a_h(\bsig_h,\btau_h):=\sum_{E\in\mathcal{T}_h}a_h^E(\bsig_h,\btau_h),\qquad \bsig_h,\btau_h\in\boldsymbol{\mathcal{H}}_h^E.
\end{align*}

Consistency and stability of the bilinear form $a_h^E(\cdot,\cdot)$ are established in the result below (see \cite[Lemma 4.6]{MR3614887})%The following result states that bilinear form $a_h^E(\cdot,\cdot)$ is consistent and stable \FL{(see \cite[Lemma 4.6]{MR3614887})}.
\begin{lemma}
\label{lmm:stab}
For each $E\in\mathcal{T}_h$ there holds
\begin{equation*}
a_h^E(\boldsymbol{\rho}_h,\btau_h)= a^E(\boldsymbol{\rho}_h,\btau_h)\quad \forall\boldsymbol{\rho}_h\in\textbf{\textsc{P}}_{k}(E),\quad\forall\btau_h\in\boldsymbol{\mathcal{H}}_h^E.
\end{equation*}
We also note the existence of positive constants $\alpha_1, \alpha_2$, uniform in both $h$ and $E$, such that
\begin{equation*}
\alpha_1 a^E(\btau_h,\btau_h)\leq a_h^E(\btau_h,\btau_h)\leq \alpha_2 \left(\|\btau_h\|_{0,E}^{2}+\|\btau_h-\Pi_h^E\btau_h\|_{0,E}^{2}\right)\quad\forall\btau_h\in\boldsymbol{\mathcal{H}}_h^E.
\end{equation*}
\end{lemma}

Now, for bilinear form $b(\cdot,\cdot)$ we have
\begin{equation*}
b(\btau_h,\bv_h)=\sum_{E\in\CT_h}b^E(\btau_h,\bv_h)=\sum_{E\in\CT_h}\int_E\bv_h\cdot\bdiv\btau_h,
\end{equation*}
which is explicitly computable with the degrees of freedom of $\boldsymbol{\mathcal{H}}_{0,h}$ and $\mathbf{Q}_h$, whereas for $c(\cdot,\cdot)$ we have
and
\begin{equation*}
c(\bw_h,\bv_h)=\sum_{E\in\CT_h}c^E(\bw_h,\bv_h)=\sum_{E\in\CT_h}\int_E\bw_h\cdot\bv_h,
\end{equation*}
which is  computable with the degrees of freedom of $\mathbf{Q}_h$
%\GR{Finally, we define the bilinear form $c_h(\cdot,\cdot)$  as:
%$$c(\boldsymbol{\rho}_h,\btau_h)\sum_{E\in CT_h}c^E(\boldsymbol{\rho}_h,\btau_h)=\sum_{E\in CT_h}\int_E\Pi_h^E\boldsymbol{\rho}_h:\Pi_h^E\btau_h,\quad\boldsymbol{\rho}_h,\btau_h\in\boldsymbol{\mathcal{H}}_h^E.$$}

With these bilinear forms at hand, we introduce the VEM discretization of \eqref{def:spectral_1}: find $\kappa_h\in\mathbb{R}^+$ and $(\boldsymbol{0},\boldsymbol{0})\neq(\boldsymbol{\rho}_h,\bu_h)\in\boldsymbol{\mathcal{H}}_{0,h}\times\mathbf{Q}_h$ such that 
\begin{equation}\label{def:spectral_disc}
\left\{
\begin{array}{rcll}
a_h(\boldsymbol{\rho}_h,\btau_h)+b(\btau_h,\bu_h) & = &0&  \forall\btau_h\in\boldsymbol{\mathcal{H}}_{0,h}, \\
b(\boldsymbol{\rho}_h,\bv_h)& = & -\kappa_h c(\bu_h,\bv_h) &  \forall\bv_h\in \mathbf{Q}_h.
\end{array}
\right.
\end{equation}
Now we introduce the discrete solution operators as follows
\begin{equation*}
\bT_{h}:\mathbf{Q}\rightarrow\mathbf{Q}_h,\quad
           \boldsymbol{f}\mapsto \bT_{h}\boldsymbol{f}:=\widehat{\bu}_h, \quad \boldsymbol{\mathcal{S}}_{h}:\mathbf{Q}\rightarrow\boldsymbol{\mathcal{H}}_{0,h},\quad
           \boldsymbol{f}\mapsto \boldsymbol{\mathcal{S}}_{h}\boldsymbol{f}:=\widehat{\boldsymbol{\rho}}_h
\end{equation*}
where the pair $(\widehat{\boldsymbol{\rho}}_h, \widehat{\bu}_h)\in\boldsymbol{\mathcal{H}}_{0,h}\times\mathbf{Q}_h$ is the solution of the following source problem
 \begin{equation}\label{def:sourcel_1_disc}
\left\{
\begin{array}{rcll}
a_h(\widehat{\boldsymbol{\rho}}_h,\btau_h)+b(\btau_h,\widehat{\bu}_h) & = &0&  \forall\btau_h\in\boldsymbol{\mathcal{H}}_{0,h}, \\
b(\widehat{\boldsymbol{\rho}}_h,\bv_h)& = & - c(\boldsymbol{f},\bv_h) &  \forall\bv_h\in \mathbf{Q}_h.
\end{array}
\right.
\end{equation}

The discrete kernel of $b(\cdot,\cdot)$ is defined by 
$\boldsymbol{\mathcal{V}}_h:=\{\btau_h\in \boldsymbol{\mathcal{H}}_{0,h}: b(\btau_h,\bv_h)=0\quad\forall \bv_h\in \mathbf{Q}_h\}\subset \boldsymbol{\mathcal{V}}$. Hence, invoking Lemma \ref{lmm:stab} it is direct to prove that $a_h(\cdot,\cdot)$ is $\boldsymbol{\mathcal{V}}_h$-coercive. More precisely, there exists a positive constant $\widehat{\alpha}$, independent of $h$ and $\lambda$, such that
$$a_h(\btau_h,\btau_h)\geq \widehat{\alpha}\|\btau_h\|_{\bdiv,\O}\quad\forall\btau_h\in\boldsymbol{\mathcal{V}}_h.$$
Moreover, the following inf-sup condition is satisfied (see \cite[Lemma 3.1]{MR3453481})
\begin{equation}\label{eq_inf-sup}
\sup_{\boldsymbol{0}\neq\btau_h\in \boldsymbol{\mathcal{H}}_{0,h}}\dfrac{b(\btau_h,\bv_h)}{\|\btau_h\|_{\bdiv,\O}}\geq \widetilde{\beta} \|\bv_h\|_{0,\O},
\end{equation} with $\widetilde{\beta}>0$ independent of $h$. By the previous results, together with the Babu\v{s}ka-Brezzi theory, the discrete operators $\bT_h$ and $\boldsymbol{\mathcal{S}}_h$ are well defined.

%We collect here the approximation properties required in the next section. Let $\bcI_k^h: \mathbb{H}^t(\O) \to \bcW_h$ be the tensorial VEM interpolation operator. For $t>1/2$, it holds (see \cite[Lemma 6]{BeiraoVemAcoustic2017})
%\begin{equation}\label{asymp_0}
%\norm{\btau - \bcI_k^h \btau}{0,\O} \leq C h^{\min\{t, k+1\}} \norm{\btau}{t,\O} \qquad \forall \btau \in \mathbb{H}^t(\O).
%\end{equation}
%For $\btau \in \mathbb{H}^t(\O)\cap \H(\div;\O)$ with $t\in (0, 1/2]$, one has the estimate (see \cite[Theorem 3.16]{hiptmair})
%\begin{equation}\label{asymp_00}
%\norm{\btau - \bcI_k^h \btau}{0,\O} \leq C h^t (\norm{\btau}{t,\O}+ \norm{\btau}{\bdiv,\O}).
%\end{equation}
%Moreover, $\bcI_k^h$ satisfies the commuting diagram property (see \cite[Lemma 5]{BeiraoVemAcoustic2017})
%\begin{equation}\label{asymp_Div}
%\norm{\bdiv (\btau - \bcI_k^h\btau) }{0,\O} = \norm{\bdiv \btau - \mathcal{P}{k}^h \bdiv \btau }{0,\O}
%\leq C h^{\min\{t, k\}} \norm{\bdiv\btau}{t,\O},
%\end{equation}
%for $\bdiv \btau \in \H^t(\O)^{2}$, where $\mathcal{P}{k}^h$ denotes the $\LO^2$-orthogonal projection onto $\textsc{P}_{k}$. Finally, we set $\btau_I:=\bcI_k^h(\btau)|_E\in \bcW_h^E$. 

We collect here the approximation properties required in the next section. 
Let  $\bcI_k^h: \boldsymbol{\mathcal{H}}^t(\O) \to \boldsymbol{\mathcal{H}}_h$ be the tensorial version of the VEM-interpolation operator. For $t>1/2$, it holds (see \cite[Lemma 6]{MR3660301})
\begin{equation*}\label{asymp0}
 \norm{\btau - \bcI_k^h \btau}_{0,\O} \leq C h^{\min\{t, k+1\}} \norm{\btau}_{t,\O} \qquad \forall \btau \in \boldsymbol{\mathcal{H}}^t(\O).
\end{equation*}
For $\btau \in \boldsymbol{\mathcal{H}}^t(\O)\cap \boldsymbol{\mathcal{H}}$ with $t\in (0, 1/2]$, the following estimate holds (see \cite[Theorem 3.16]{MR2009375})
\begin{equation*}\label{asymp00}
 \norm{\btau - \bcI_k^h \btau}_{0,\O} \leq C h^t (\norm{\btau}_{t,\O}
 + \norm{\btau}_{\bdiv,\O}).
\end{equation*}
A further property of $\mathcal{I}_k^h$ is the commuting diagram relation stated below (see \cite[Lemma 5]{MR3660301})
\begin{equation*}\label{asympDiv}
 \norm{\bdiv (\btau - \bcI_k^h\btau) }_{0,\O} = \norm{\bdiv \btau - \mathcal{P}_{k}^h \bdiv \btau }_{0,\O} 
 \leq C h^{\min\{t, k\}} \norm{\bdiv\btau}_{t,\O},
\end{equation*}
for  $\bdiv \btau \in \H^t(\O)^{2}$ with $t\in (0, 1/2]$, where  $\mathcal{P}_{k}^h$ denotes the $\L^2$-orthogonal projection onto $\textsc{P}_{k}$. Finally, and for simplicity, we define
$\btau_I:=\bcI_k^h(\btau)|_E\in \boldsymbol{\mathcal{H}}_h^E$.

As we did on the continuous setting, for all $\boldsymbol{\rho}_h,\btau_h\in\boldsymbol{\mathcal{H}}_{0,h}$ and for all $\bv_h,\bu_h\in\mathbf{Q}_h$, we define  the following discrete bilinear form 
\begin{equation*}
A_h((\boldsymbol{\rho}_h,\bu_h),(\btau_h,\bv_h)):=a_h(\boldsymbol{\rho}_h,\btau_h)+b(\btau_h,\bu_h)+b(\boldsymbol{\rho}_h,\bv_h).
\end{equation*}

\section{Spectral convergence}
\label{sec:conv}

The aim of this section is to prove spectral convergence. To conclude this convergence, it is necessary to analyze the convergence of operator $\bT_h$ to $\bT$ as $h\rightarrow 0$, since this property in conjunction with the classic theory of compact operators \cite{MR1115235}, let us conclude that both, the discrete eigenvalues and eigenfunctions  converge to the corresponding continuous counterparts.

The following result states  that the discrete solution operator $\bT_h$ converges to $\bT$ as $h\rightarrow 0$. 
\begin{theorem}
\label{thm:erroroperator} 
Let $\boldsymbol{f}\in\mathbf{Q}$ be such that $\widehat{\bu}:=\bT\boldsymbol{f}$, $\widehat{\boldsymbol{\rho}}:= \boldsymbol{\mathcal{S}}\boldsymbol{f}$ and $\widehat{\bu}_h:=\bT\boldsymbol{f}$,  $\widehat{\boldsymbol{\rho}}_h:= \boldsymbol{\mathcal{S}}_h\boldsymbol{f}$, where $(\widehat{\boldsymbol{\rho}},\widehat{\bu})\in\boldsymbol{\mathcal{H}}_0\times\mathbf{Q}$
is the solution of \eqref{def:sourcel_1} and $(\widehat{\boldsymbol{\rho}}_h,\widehat{\bu}_h)\in\boldsymbol{\mathcal{H}}_{0,h}\times\mathbf{Q}_h$ is the solution of \eqref{def:sourcel_1_disc}. Then, the following estimate holds
$$
\|(\boldsymbol{\mathcal{S}}-\boldsymbol{\mathcal{S}}_h)\boldsymbol{f}\|_{\bdiv,\O}+\|(\bT-\bT_h)\boldsymbol{f}\|_{0,\O}\leq  \widehat{C}h^s\left(\|\widehat{\bsig}\|_{s,\O}+\|\widehat{\bu}\|_{1+s,\O}\right)\leq  \widehat{C}h^s\|\boldsymbol{f}\|_{0,\O},%.\|\boldsymbol{f}\|_{0,\O},%+|\tilde{u}-\tilde{u}_I|_{1,\Omega}+|\tilde{u}-\tilde{u}_{\pi}|_{1,\Omega}\right).
$$
where the constant $\widehat{C}>0$ is independent of $h$ and $\lambda$,  and $s$ is the regularity exponent provided by Lemma \ref{lmm:add_eigen}.
\end{theorem}

\begin{proof}
Let $\widehat{\bu}_I$ and $\bcI_k^h\widehat{\brho}$  be the best approximations of $\widehat{\bu}$ and $\widehat{\boldsymbol{\rho}}$, respectively.  We begin by analyzing the error associated to the displacement. From the definitions of the solution operators $\bT$ and $\bT_h$, elementary calculations lead to 
\begin{equation}\label{eq_initial}
\|(\bT-\bT_h)\bF\|_{0,\O}=\|\widehat{\bu}-\widehat{\bu}_h\|_{0,\O}\leq \underbrace{\|\widehat{\bu}-\mathcal{P}_{k}^h(\widehat{\bu})\|_{0,\O}}_{\textrm{I}}+\underbrace{\|\mathcal{P}_{k}^h(\widehat{\bu})-\widehat{\bu}_h\|_{0,\O},}_{\textrm{II}}
\end{equation}
where the task is to estimate terms $\textrm{I}$ and $\textrm{II}$. For $\textrm{I}$, using the regularity of $\widehat{\bu}$  and the dependency on the data $\bF$, we have 
\begin{equation}
\label{eq:bound_I}
\textrm{I}\,\leq C h^s\|\bF\|_{0,\O}.
\end{equation}

Now for the term $\textrm{II}$, we taking $\bv_h:=\boldsymbol{\mathcal{P}}_k(\widehat{\bu})-\widehat{\bu}_h\in \mathbf{Q}_h$ in the discrete inf-sup condition \eqref{eq_inf-sup}, we obtain
\begin{equation*}
\|\mathcal{P}_{k}^h(\widehat{\bu})-\widehat{\bu}_h\|_{0,\O}\leq\dfrac{1}{\widetilde{\beta}}\sup_{\boldsymbol{0}\neq\btau_h\in \boldsymbol{\mathcal{H}}_{0,h}}\dfrac{b(\btau_h,\mathcal{P}_{k}^h(\widehat{\bu})-\widehat{\bu}_h)}{\|\btau_h\|_{\bdiv,\O}}.
\end{equation*}
%\begin{multline*}
%\|(\bT-\bT_h)\bF\|_{0,\O}\leq \|(\bT-\bT_h)\bF\|_{0,\O}+\|(\boldsymbol{\mathcal{S}}-\boldsymbol{\mathcal{S}})\bF\|_{\bdiv,\O}\\
%\leq \underbrace{\|\widehat{\bu}-\widehat{\bu}_I\|_{0,\O}+\|\widehat{\boldsymbol{\rho}}-\bcI_k^h\widehat{\brho}\|_{\bdiv,\O}}_{\boldsymbol{\textrm{I}}}+
%\underbrace{\|\widehat{\bu}_I-\widehat{\bu}_h\|_{0,\O}+\|\bcI_k^h\widehat{\brho}-\widehat{\boldsymbol{\rho}}_h\|_{\bdiv,\O}}_{\boldsymbol{\textrm{II}}},
%\end{multline*}

Now, since $\btau_h\in\boldsymbol{\mathcal{H}}_{0,h}$, we have that  $\bdiv\btau_h\in\mathbf{Q}_h$. Then, recalling that $\boldsymbol{\mathcal{P}}_k$ is the $\L^2$-orthogonal projection and using the first equality in \eqref{def:sourcel_1} and \eqref{def:sourcel_1_disc},, we obtain 
\begin{multline*}
b(\btau_h,\mathcal{P}_{k}^h(\widehat{\bu})-\widehat{\bu})=b(\btau_h,\mathcal{P}_{k}^h(\widehat{\bu}))-b(\btau_h,\widehat{\bu})=b(\btau_h,\widehat{\bu})-b(\btau_h,\widehat{\bu}_h)\\
=a_h(\widehat{\brho}_h,\btau_h)-a(\widehat{\brho},\btau_h)=a_h(\widehat{\brho}_h-\bcI_k^h\widehat{\brho},\btau_h) +a_h(\bcI_k^h\widehat{\brho},\btau_h)-a(\widehat{\brho},\btau_h)\\
=a_h(\widehat{\brho}_h-\bcI_k^h\widehat{\brho},\btau_h) +\sum_{E\in\CT_h} [a_h^E(\bcI_k^h\widehat{\brho}-\Pi_h^E\widehat{\boldsymbol{\rho}},\btau_h)+a^E(\Pi_h^E\widehat{\boldsymbol{\rho}}-\widehat{\boldsymbol{\rho}},\btau_h)]\\
\leq C\left(\|\widehat{\brho}_h-\bcI_k^h\widehat{\brho}\|_{0,\O}\|\btau_h\|_{0,\O}\right.\\
\left.+\sum_{E\in\CT_h}\left(\|\widehat{\boldsymbol{\rho}}-\bcI_k^h\widehat{\brho}\|_{0,E}+2\|\widehat{\boldsymbol{\rho}}-\Pi_h^E\widehat{\boldsymbol{\rho}}\|_{0,E}\right)\|\btau_h\|_{0,E}\right).
\end{multline*}
Thanks to the previous estimate, we conclude that:
\begin{equation}\label{eq_cotapro}
\|\mathcal{P}_{k}^h(\widehat{\bu})-\widehat{\bu}_h\|_{0,\O}\leq\dfrac{C}{\widetilde{\beta}}\left(\|\widehat{\brho}_h-\bcI_k^h\widehat{\brho}\|_{0,\O}+\sum_{E\in\CT_h}\left(\|\widehat{\boldsymbol{\rho}}-\bcI_k^h\widehat{\brho}\|_{0,E}+2\|\widehat{\boldsymbol{\rho}}-\Pi_h^E\widehat{\boldsymbol{\rho}}\|_{0,E}\right)\right).
\end{equation}
The following step consists in bounding the first term of the previous estimate, implying that now we need to focus on the term $\|(\boldsymbol{\mathcal{S}}-\boldsymbol{\mathcal{S}}_h)\boldsymbol{f}\|_{\bdiv,\O}$. With this purpose, the commutative diagram property reveals
$$\bdiv\bcI_k^h\widehat{\brho}=\mathcal{P}_{k}^h(\bdiv\brho)=-\mathcal{P}_{k}^h(\boldsymbol{f})=\bdiv\widehat{\brho}_h,$$
 and as a consdequence, $\bcI_k^h\widehat{\brho}-\widehat{\brho}_h\in \boldsymbol{\mathcal{V}}_h$. Now, by  defining $\bphi_h:=\widehat{\brho}_h-\bcI_k^h\widehat{\brho}$\, we have
\begin{multline*}
C\|\widehat{\brho}_h-\bcI_k^h\widehat{\brho}\|_{\bdiv,\O}^2=C\|\widehat{\brho}_h-\bcI_k^h\widehat{\brho}\|_{0,\O}^2\leq a_h(\widehat{\brho}_h-\bcI_k^h\widehat{\brho},\bphi_h)\\
=a_h(\widehat{\brho}_h,\bphi_h)-a_h(\bcI_k^h\widehat{\brho},\bphi_h)=-b(\bphi_h,\widehat{\bu}_h)-a_h(\bcI_k^h\widehat{\brho},\bphi_h)=-a_h(\bcI_k^h\widehat{\brho},\bphi_h)\\
=-\sum_{E\in\CT_h}[a_h^E(\bcI_k^h\widehat{\brho}-\Pi_h^E\widehat{\boldsymbol{\rho}},\bphi_h)+a^E(\Pi_h^E\widehat{\boldsymbol{\rho}}-\widehat{\boldsymbol{\rho}},\bphi_h)+a^E(\widehat{\boldsymbol{\rho}},\bphi_h)]\\
=\sum_{E\in\CT_h}[a_h^E(\Pi_h^E\widehat{\boldsymbol{\rho}}-\bcI_k^h\widehat{\brho},\bphi_h)+a^E(\widehat{\boldsymbol{\rho}}-\Pi_h^E\widehat{\boldsymbol{\rho}},\bphi_h)]\\
\leq C\sum_{E\in\CT_h}\left(\|\widehat{\boldsymbol{\rho}}-\bcI_k^h\widehat{\brho}\|_{0,\O}+2\|\widehat{\boldsymbol{\rho}}-\Pi_h^E\widehat{\boldsymbol{\rho}}\|_{0,\O}\right)\|\bphi_h\|_{0,\O}.
\end{multline*}
Then
$$\|\widehat{\brho}_h-\bcI_k^h\widehat{\brho}\|_{0,\O}\leq C\sum_{E\in\CT_h}\left(\|\widehat{\boldsymbol{\rho}}-\bcI_k^h\widehat{\brho}\|_{0,\O}+2\|\widehat{\boldsymbol{\rho}}-\Pi_h^E\widehat{\boldsymbol{\rho}}\|_{0,\O}\right).$$
Replacing the previous estimate in \eqref{eq_cotapro}, we arrive to
\begin{equation}\label{eq_cotapro2}
\|\mathcal{P}_{k}^h(\widehat{\bu})-\widehat{\bu}_h\|_{0,\O}\leq\dfrac{C}{\widetilde{\beta}}\left(\sum_{E\in\CT_h}\left(2\|\widehat{\boldsymbol{\rho}}-\bcI_k^h\widehat{\brho}\|_{0,E}+4\|\widehat{\boldsymbol{\rho}}-\Pi_h^E\widehat{\boldsymbol{\rho}}\|_{0,E}\right)\right)\leq \dfrac{C}{\widetilde{\beta}}h^s\|\bF\|_{0,\O}.
\end{equation}
Thus, combining \eqref{eq_initial}, \eqref{eq:bound_I}, and \eqref{eq_cotapro2}, the proof is complete.
%\begin{multline}
%\label{eq:term_II}
%\textrm{II}\leq A_h((\bcI_k^h\widehat{\brho}-\widehat{\boldsymbol{\rho}}_h,\widehat{\bu}_I-\widehat{\bu}_h),(\btau_h,\bv_h))=A_h((\bcI_k^h\widehat{\brho},\widehat{\bu}_I),(\btau_h,\bv_h))-c_h(\bF,\bv_h)\\
%=\sum_{E\in\CT_h} [a_h^E(\bcI_k^h\widehat{\brho},\btau_h)+b^E(\btau_h,\widehat{\bu}_I)+b^E(\bcI_k^h\widehat{\brho},\bv_h)]-c_h(\bF,\bv_h)\\
%=\sum_{E\in\CT_h}[a_h^E(\bcI_k^h\widehat{\brho}-\widehat{\boldsymbol{\rho}}_\pi,\btau_h)+a^E(\widehat{\boldsymbol{\rho}}_\pi-\widehat{\boldsymbol{\rho}},\btau_h)]\\
%+a(\widehat{\boldsymbol{\rho}},\btau_h)+b(\btau_h,\widehat{\bu}_I)+b(\bcI_k^h\widehat{\brho},\bv_h)-c_h(\bF,\bv_h)\\
%=\underbrace{\sum_{E\in\CT_h}[a_h^E(\bcI_k^h\widehat{\brho}-\widehat{\boldsymbol{\rho}}_\pi,\btau_h)+a^E(\widehat{\boldsymbol{\rho}}_\pi-\widehat{\boldsymbol{\rho}},\btau_h)]}_{\mathcal{A}}+\underbrace{b(\btau_h,\widehat{\bu}_I-\widehat{\bu})}_{\mathcal{B}}\\
%+\underbrace{b(\bcI_k^h\widehat{\brho}-\widehat{\boldsymbol{\rho}},\bv_h)}_{\mathcal{C}}+\underbrace{c(\bF,\bv_h)-c_h(\bF,\bv_h)}_{\mathcal{D}}
%%=\sum_{E\in\CT_h} a_h^E(\bcI_k^h\widehat{\brho}-\boldsymbol{\rho}_\pi,\btau_h)+a_h^E(\widehat{\boldsymbol{\rho}}_\pi,\btau_h)+b^E(\btau_h,\widehat{\bu}_I)+b(\bcI_k^h\widehat{\brho},\bv_h)-c_h(\bF,\bv_h)
%\end{multline}
\end{proof}

The convergence result of Theorem \ref{thm:erroroperator}, in conjunction with the spectral theory of \cite[Chapter IV]{MR0203473} and \cite[Theorem 9.1]{MR2652780}, suffices to certify that the discrete eigenvalue problem is free of spurious modes, as the following theorem makes explicit.
%With Theorem \ref{thm:erroroperator} at hand, and invoking the  well established theory of  \cite[Chapter IV]{MR0203473} and \cite[Theorem 9.1]{MR2652780}, \FL{we are now in a position to conclude that our numerical method does not introduce spurious eigenvalues}. This is stated in the following theorem.
\begin{theorem}
	\label{thm:spurious_free}
	Let $V\subset\mathbb{C}$ be an open set containing $\sp(\bT)$. Then, there exists $h_0>0$ such that $\sp(\bT_h)\subset V$ for all $h<h_0$.
\end{theorem}

We end this section providing an improvement for the $\L^2$ order of the displacement.
\begin{lemma}
\label{lmm:L2order}
There exists a constant $C>0$, independent of $\lambda$ and $h$, such that
$$\|\bu-\bu_h\|_{0,\O}\leq Ch^{2s}.$$
\end{lemma}
\begin{proof}
We use  a standard duality argument. Let us consider the following auxiliary problem: Find $(\boldsymbol{\varphi},\boldsymbol{\psi})\in\boldsymbol{\mathcal{H}}_0\times\mathbf{Q}$ such that
 \begin{equation}\label{def:dual_1}
\left\{
\begin{array}{rcll}
a(\boldsymbol{\varphi},\btau)+b(\btau,\boldsymbol{\psi}) & = &0&  \forall\btau\in\boldsymbol{\mathcal{H}}_{0}, \\
b(\boldsymbol{\varphi},\bv)& = & - c(\widehat{\bu}-\widehat{\bu}_h,\bv)&  \forall\bv\in \mathbf{Q}.
\end{array}
\right.
\end{equation}
In this case the solution satisfies:
\begin{equation*}
\|\boldsymbol{\varphi}\|_{s,\O}+\|\bdiv \boldsymbol{\varphi}\|_{1+s,\O}+\|\boldsymbol{\psi}\|_{1+s}\leq C \|\widehat{\bu}-\widehat{\bu}_h\|_{0,\O}.
\end{equation*}

Let $\bcI_k^h\boldsymbol{\varphi}\in\boldsymbol{\mathcal{H}}_h$, be  the virtual interpolator of $\boldsymbol{\varphi}$. Adding and subtracting $\bcI_k^h\boldsymbol{\varphi}$ and testing \eqref{def:dual_1}  with $\bv=\widehat{\bu}-\widehat{\bu}_h$ yield to 
\begin{equation}
\label{eq:duality0}
\|\widehat{\bu}-\widehat{\bu}_h\|_{0,\O}^2=b(\boldsymbol{\varphi}-\bcI_k^h\boldsymbol{\varphi},\widehat{\bu}_h-\widehat{\bu})+b(\bcI_k^h\boldsymbol{\varphi},\widehat{\bu}_h-\widehat{\bu}).
\end{equation}

Using the continuous and discrete source problems \eqref{def:sourcel_1} and \eqref{def:sourcel_1_disc}, respectively, testing with discrete functions in \eqref{def:sourcel_1},  and subtracting the first equation on these  problems, we obtain  the following identity
\begin{equation*}
\label{eq:duality1}
a(\widehat{\boldsymbol{\rho}},\btau_h)-a_h(\widehat{\boldsymbol{\rho}},\btau_h)=b(\btau_h,\widehat{\bu}_h-\widehat{\bu})\quad\forall\btau_h\in\boldsymbol{\mathcal{H}}_{0,h},
\end{equation*}
whereas if we consider the second equation on each source problem, where the continuous one is tested with $\bv_h\in\mathbf{Q}_h$, we obtain after the subtraction the following equation
\begin{equation}
\label{eq:duality2}
b(\widehat{\boldsymbol{\rho}}-\widehat{\boldsymbol{\rho}}_h,\bv_h)=0\quad\forall\bv\in\mathbf{Q}_h.
\end{equation}

Now, from  \eqref{eq:duality0} we obtain the following identity
\begin{align*}
\|\widehat{\bu}-\widehat{\bu}_h\|_{0,\O}^2&=b(\boldsymbol{\varphi}-\bcI_k^h\boldsymbol{\varphi},\widehat{\bu}_h-\widehat{\bu})+a(\widehat{\boldsymbol{\rho}},\bcI_k^h\boldsymbol{\varphi})-a_h(\widehat{\boldsymbol{\rho}}_h,\bcI_k^h\boldsymbol{\varphi})\\
&=\underbrace{b(\boldsymbol{\varphi}-\bcI_k^h\boldsymbol{\varphi},\widehat{\bu}_h-\widehat{\bu})}_{\textrm{T}_1}+\underbrace{a(\widehat{\boldsymbol{\rho}}-\widehat{\boldsymbol{\rho}}_h,\bcI_k^h\boldsymbol{\varphi})}_{\textrm{T}_2}+\underbrace{a(\widehat{\boldsymbol{\rho}}_h,\bcI_k^h\boldsymbol{\varphi})-a_h(\widehat{\boldsymbol{\rho}}_h,\bcI_k^h\boldsymbol{\varphi})}_{\textrm{T}_3},
\end{align*}
where the task now is to estimate the terms $\textrm{T}_1$, $\textrm{T}_2$, and $\textrm{T}_3$.  For $\textrm{T}_2$ we have
\begin{align*}
\textrm{T}_2&= a(\widehat{\boldsymbol{\rho}}-\widehat{\boldsymbol{\rho}}_h,\bcI_k^h\boldsymbol{\varphi}-\boldsymbol{\varphi})+a(\widehat{\boldsymbol{\rho}}-\widehat{\boldsymbol{\rho}}_h,\boldsymbol{\varphi})\\
&= a(\widehat{\boldsymbol{\rho}}-\widehat{\boldsymbol{\rho}}_h,\bcI_k^h\boldsymbol{\varphi}-\boldsymbol{\varphi})+a(\boldsymbol{\varphi}, \widehat{\boldsymbol{\rho}}-\widehat{\boldsymbol{\rho}}_h)\\
&= a(\widehat{\boldsymbol{\rho}}-\widehat{\boldsymbol{\rho}}_h,\bcI_k^h\boldsymbol{\varphi}-\boldsymbol{\varphi})-b(\widehat{\boldsymbol{\rho}}-\widehat{\boldsymbol{\rho}}_h,\boldsymbol{\psi})\\
&=a(\widehat{\boldsymbol{\rho}}-\widehat{\boldsymbol{\rho}}_h,\bcI_k^h\boldsymbol{\varphi}-\boldsymbol{\varphi})-b(\widehat{\boldsymbol{\rho}}-\widehat{\boldsymbol{\rho}}_h,\boldsymbol{\psi}-\mathcal{P}_{k}^h(\boldsymbol{\psi}))-b(\widehat{\boldsymbol{\rho}}-\widehat{\boldsymbol{\rho}}_h,\mathcal{P}_{k}^h(\boldsymbol{\psi}))\\
&=a(\widehat{\boldsymbol{\rho}}-\widehat{\boldsymbol{\rho}}_h,\bcI_k^h\boldsymbol{\varphi}-\boldsymbol{\varphi})-b(\widehat{\boldsymbol{\rho}}-\widehat{\boldsymbol{\rho}}_h,\boldsymbol{\psi}-\mathcal{P}_{k}^h(\boldsymbol{\psi})),
\end{align*}
where for the last equality we have used \eqref{eq:duality2}. For $\textrm{T}_3$, we have to operate element by element as follows
\begin{align*}
T_3&=\sum_{E\in\CT_h}\left(a^E(\widehat{\boldsymbol{\rho}}_h,\bcI_k^h\boldsymbol{\varphi})-a_h^E(\widehat{\boldsymbol{\rho}}_h,\bcI_k^h\boldsymbol{\varphi})\right)\\
&=\sum_{E\in\CT_h}\left(a^E(\widehat{\boldsymbol{\rho}}_h-\Pi_h^E\widehat{\boldsymbol{\rho}},\bcI_k^h\boldsymbol{\varphi})-a_h^E(\widehat{\boldsymbol{\rho}}_h-\Pi_h^E\widehat{\boldsymbol{\rho}}),\bcI_k^h\boldsymbol{\varphi})\right)\\
&=\sum_{E\in\CT_h}\left(a^E(\widehat{\boldsymbol{\rho}}_h-\Pi_h^E\widehat{\boldsymbol{\rho}}),\bcI_k^h\boldsymbol{\varphi}-\Pi_h^E\boldsymbol{\varphi})-a_h^E(\widehat{\boldsymbol{\rho}}_h-\Pi_h^E\widehat{\boldsymbol{\rho}},\bcI_k^h\boldsymbol{\varphi}-\Pi_h^E\boldsymbol{\varphi})\right),
\end{align*}
where n the previous equality, we have used that $\Pi_h^E\widehat{\boldsymbol{\rho}}\in\boldsymbol{\mathcal{L}}^2(\O)$ and satisfies that $\widehat{\boldsymbol{\rho}}_\p|_{E}\in\boldsymbol{\mathcal{P}}_{k}(E)$ for all $E\in\CT_h$ and the same satisfies the term $\Pi_h^E\boldsymbol{\varphi}$.
Now, combining all the estimates, we obtain that
\begin{multline*}
\|\widehat{\bu}-\widehat{\bu}_h\|_{0,\O}^2=b(\boldsymbol{\varphi}-\bcI_k^h\boldsymbol{\varphi},\widehat{\bu}_h-\widehat{\bu})+a(\widehat{\boldsymbol{\rho}}-\widehat{\boldsymbol{\rho}}_h,\bcI_k^h\boldsymbol{\varphi}-\boldsymbol{\varphi})-b(\widehat{\boldsymbol{\rho}}-\widehat{\boldsymbol{\rho}}_h,\boldsymbol{\psi}-\mathcal{P}_{k}^h(\boldsymbol{\psi}))\\
+\sum_{E\in\CT_h}\left(a^E(\widehat{\boldsymbol{\rho}}_h-\Pi_h^E\widehat{\boldsymbol{\rho}}),\bcI_k^h\boldsymbol{\varphi}-\Pi_h^E\boldsymbol{\varphi})-a_h^E(\widehat{\boldsymbol{\rho}}_h-\Pi_h^E\widehat{\boldsymbol{\rho}},\bcI_k^h\boldsymbol{\varphi}-\Pi_h^E\boldsymbol{\varphi})\right)\\
\leq C\left(\|\bdiv(\boldsymbol{\varphi}-\bcI_k^h\boldsymbol{\varphi})\|_{0,\O}\|\widehat{\bu}-\widehat{\bu}_h\|_{0,\O}+\|\widehat{\boldsymbol{\rho}}-\widehat{\boldsymbol{\rho}}_h\|_{0,\O}\|\boldsymbol{\varphi}-\bcI_k^h\boldsymbol{\varphi}\|_{0,\O}\right.\\
\left.+\|\bdiv(\widehat{\boldsymbol{\rho}}-\widehat{\boldsymbol{\rho}}_h)\|_{0,\O}\|\boldsymbol{\psi}-\mathcal{P}_{k}^h(\boldsymbol{\psi})\|_{0,\O}\right.\\
\left. +2\sum_{E\in\CT_h}(\|\widehat{\boldsymbol{\rho}}-\widehat{\boldsymbol{\rho}}_h\|_{0,E}+\|\widehat{\boldsymbol{\rho}}-\Pi_h^E\widehat{\boldsymbol{\rho}}\|_{0,E})(\|\boldsymbol{\varphi}-\bcI_k^h\boldsymbol{\varphi}\|_{0,E}+\|\boldsymbol{\varphi}-\Pi_h^E\boldsymbol{\varphi}\|_{0,E})\right) \\
\leq Ch^{2s}\|\widehat{\bu}-\widehat{\bu}_h\|_{0,\O}.
\end{multline*}
This concludes the proof.
\end{proof}

\subsection{Error estimates}
For completeness, we recall the construction of spectral projection operators. Suppose $\eta$ is a nonzero isolated eigenvalue of $\mathbf{T}$ of algebraic multiplicity $m$, and let $\Gamma$ be a sufficiently small disk in $\mathbb{C}$ centered at $\eta$, enclosing no other spectral value of $\mathbf{T}$, with $\partial\Gamma \cap \text{sp}(\mathbf{T}) = \emptyset$. The corresponding spectral projector $\mE$ onto the generalized eigenspace $R(\mE)$ of $\mathbf{T}$ associated with $\eta$ is then defined as
 %We first recall the definition of spectral projectors. Let $\eta$ be a nonzero isolated eigenvalue of $\bT$ with algebraic multiplicity $m$ and let $\Gamma$
%be a disk of the complex plane centered in $\eta$, such that $\eta$ is the only eigenvalue of $\bT$ lying in $\Gamma$ and $\partial\Gamma\cap\sp(\bT)=\emptyset$. With these considerations at hand, we define the spectral projection of \FL{$\mE$ onto the generalized eigenspace $R(\mE)$} associated to $\bT$ \FL{and $\eta$ }as follows:
 $$\displaystyle \mE:=\frac{1}{2\pi i}\int_{\partial\Gamma} (z\boldsymbol{I}-\bT)^{-1}\,dz,$$
%\item The spectral projector of $\bT^*$ associated to $\bar{\mu}$ is $\displaystyle \boldsymbol{E}^*:=\frac{1}{2\pi i}\int_{\partial\Gamma} (z\boldsymbol{I}-\bT^*)^{-1}\,dz,$
where $\boldsymbol{I}$ represents the identity operator. On the other hand, for the discrete version of $\mE$ we need to take into consideration that 
Theorem \ref{thm:erroroperator} implies the existence of  $m$ eigenvalues, which lie in $\Gamma$, namely $\eta_{h}^{(1)},\ldots,\eta_{h}^{(m)}$, repeated according their respective multiplicities, that converge to $\eta$ as $h$ goes to zero. Hence, the projection onto the discrete invariant subspace $R(\mE_h)$ of $\bT$, spanned by the generalized eigenvector of $\bT_h$ corresponding to 
 $\eta_{h}^{(1)},\ldots,\eta_{h}^{(m)}$ is defined by 
\begin{equation*}
\mE_h:=\frac{1}{2\pi i}\int_{\partial\Gamma} (z\boldsymbol{I}-\bT_h)^{-1}\,dz.
\end{equation*}

To obtain the error estimate for the eigenfunctions, we need to  recall the definition of the \textit{gap} $\hdel(\cdot,\cdot)$ between two closed
subspaces $\mathfrak{X}$ and $\mathfrak{Y}$ of $\boldsymbol{\L}^2(\O)$:
$$
\hdel(\mathfrak{X},\mathfrak{Y})
:=\max\big\{\delta(\mathfrak{X},\mathfrak{Y}),\delta(\mathfrak{Y},\mathfrak{X})\big\}, \text{ where } \delta(\mathfrak{X},\mathfrak{Y})
:=\sup_{\underset{\left\|\boldsymbol{x}\right\|_{0,\O}=1}{\boldsymbol{x}\in\mathfrak{X}}}
\left(\inf_{\boldsymbol{y}\in\mathfrak{Y}}\left\|\boldsymbol{x}-\boldsymbol{y}\right\|_{0,\O}\right),
$$
since the desire error estimate follows from the approximation between the continuous and discrete generalized  eigenspaces. Hence, the following result from \cite[Theorems 14.9 and 14.11]{MR2652780} holds for our case.%cite[Theorems 7.1 and 7.3]{MR1115235}.
\begin{theorem}
\label{thm:errors1}
The following estimates hold
\begin{equation*}
\hdel(R(\mE),R(\mE_h))\leq C  h^{\min\{r,k+1\}}\quad\text{and}\quad
|\eta-\eta_{h}^{(i)}|\leq C h^{\min\{r,k+1\}}.\qquad i=1,...,m,
\end{equation*}
where the regularity parameter $r>0$ is the same as in Lemma \ref{lmm:add_eigen}, $k\geq 0$ is the polynomial degree of approximation of the discrete spaces,  and the constant $C>0$ is independent of $\lambda$ and $h$.

%where the hidden constants are independent of $h$.
\end{theorem}

Although Theorem \ref{thm:errors1} yields a linear convergence rate for eigenvalue approximation, a sharper quadratic estimate can be derived. Indeed, the error bound established for $\eta$ of $\mathbf{T}$ transfers directly to an analogous bound for $\kappa = \frac{1}{\eta}$ in \eqref{def:spectral_1}, associated with eigenspace $\mathcal{E}$. Setting $\kappa_h^{(i)} = \frac{1}{\eta_h^{(i)}}$, $1 \leq i \leq m$, as the eigenvalues of \eqref{def:spectral_disc} corresponding to the invariant subspace $\mathcal{E}_h$, we obtain the following result.
%\FL{Particularly for the eigenvalues, Theorem \ref{thm:errors1} establishes a linear order of approximation. However, this order of convergence can be improved to a quadratic order.   With this aim in mind, first} we note that the error estimate for the eigenvalue \FL{$\eta$} of $\bT$ leads to an analogous estimate for the approximation of the eigenvalue $\kappa= \dfrac{1}{\eta}$ of \eqref{def:spectral_1} with eigenspace $\mE$. Let  $\kappa_{h}^{(i)}=\dfrac{1}{\eta_{h}^{(i)}}$, $1\leq  i\leq m$ be the eigenvalues of \eqref{def:spectral_disc}  with invariant subspace $\mE_{h}$. Therefore we have the following result.
\begin{theorem}
\label{thm:double:order}
There exist positive constants $C$, independent of $\lambda$ and $h$, and $h_0$, such that for $h<h_0$
\begin{equation*}
|\kappa-\kappa_{h}^{(i)}|\leq Ch^{2\min\{r,k+1\}},\qquad i=1,\ldots, m.
\end{equation*}
\end{theorem}
\begin{proof}
Let $\bu_h\in\mE_h$ be an eigenfunction associated to the eigenvalue $\kappa_h^{(i)}$, with $i=1,\ldots,m$, such that $c_h(\bu_h,\bu_h)=\|\bu_h\|_{0,\O}^2=1$. Let us consider  the following eigenvalue problem
\begin{equation*}
A((\boldsymbol{\rho},\bu),(\boldsymbol{\rho},\bu))=-\kappa c(\bu,\bu)\quad\forall (\boldsymbol{\rho},\bu)\in\boldsymbol{\mathcal{H}}_0\times\mathbf{Q},
\end{equation*}
and its VEM discretization given as the following problem
\begin{equation*}
A_h((\boldsymbol{\rho}_h,\bu_h),(\boldsymbol{\rho}_h,\bu_h))=-\kappa_{h}^{(i)} c_h(\bu_h,\bu_h)\quad\forall (\boldsymbol{\rho}_h,\bu_h)\in\boldsymbol{\mathcal{H}}_{0,h}\times\mathbf{Q}_h.
\end{equation*}
These eigenvalue problems lead to the following well known algebraic identity
\begin{multline}
\label{eq:padra}
(\kappa-\kappa_h^{(i)})c(\bu_h,\bu_h)=\underbrace{A((\boldsymbol{\rho}_h-\boldsymbol{\rho}_h,\bu-\bu_h),(\boldsymbol{\rho}_h-\boldsymbol{\rho}_h,\bu-\bu_h))}_{\Lambda_1}\\
+\underbrace{\kappa c(\bu-\bu_h,\bu-\bu_h)}_{\Lambda_2}
+\underbrace{\sum_{E\in\CT_h} a_h^E(\boldsymbol{\rho}_h,\boldsymbol{\rho}_h)- a^E(\boldsymbol{\rho}_h,\boldsymbol{\rho}_h)}_{\Lambda_3},
\end{multline}
where each of the contributions on  the right-hand side of \eqref{eq:padra} must be estimated.  We begin with $\mathrm{T}_1$. From the definition of $A(\cdot,\cdot)$, triangle inequality, Cauchy-Schwarz inequality, and the regularity of the eigenfunctions on $\mE$, we have
\begin{multline}
\label{eq:T1_control}
|\Lambda_1|\leq |a(\boldsymbol{\rho}-\boldsymbol{\rho}_h, \boldsymbol{\rho}-\boldsymbol{\rho}_h)|+2|b(\boldsymbol{\rho}-\boldsymbol{\rho}_h,\bu-\bu_h)|\\
\leq\frac{1}{\mu}\|(\boldsymbol{\rho}-\boldsymbol{\rho}_h)^{\texttt{d}}\|_{0,\O}^2+\dfrac{1}{4\lambda+6\mu}\|\tr(\boldsymbol{\rho}-\boldsymbol{\rho}_h)\|_{0,\O}^2+2\|\bdiv(\boldsymbol{\rho}-\boldsymbol{\rho}_h)\|_{0,\O}\|\bu-\bu_h\|_{0,\O}\\
\leq\frac{3(2+\sqrt{2})^2+2}{12\mu}\|(\boldsymbol{\rho}-\boldsymbol{\rho}_h)\|_{0,\O}^2+\|\bdiv(\boldsymbol{\rho}-\boldsymbol{\rho}_h)\|_{0,\O}^2+\|\bu-\bu_h\|_{0,\O}^2\\
\leq\max\left\{\frac{3(2+\sqrt{2})^2+2}{12\mu},1\right\}(\|\boldsymbol{\rho}-\boldsymbol{\rho}_h\|_{\bdiv,\O}^2+\|\bu-\bu_h\|_{0,\O}^2)\\
\leq C\max\left\{\frac{3(2+\sqrt{2})^2+2}{12\mu},1\right\}h^{2\min\{r,k+1\}}.
\end{multline}
The control of $\Lambda_2$ is straightforward: 
\begin{equation}
\label{eq:T2_control}
\Lambda_2\leq \|\bu-\bu_h\|_{0,\O}^2\leq C h^{2\min\{r,k+1\}}.
\end{equation}
%For $\Lambda_3$, we invoke the orthogonality properties of $\Pi_0^E$ to obtain
%\begin{multline}
%\label{eq:T3_control}
%\Lambda_3=\sum_{E\in\CT_h}c^E(\bu_h,\bu_h)-c_h^E(\bu_h,\bu_h)=\sum_{E\in\CT_h}\int_E(\bu_h-\Pi_0^E\bu_h)(\bu_h-\Pi_0^E\bu_h)\\
%=\sum_{E\in\CT_h}\|\bu_h-\Pi_0^E\bu_h\|_{E,\O}^2\leq Ch^{2\min\{r,\GR{k+1}\}}.
%\end{multline}
Now, for $\Lambda_3$ we invoke the consistency in order to  obtain the following estimate
\begin{multline}
\label{eq:T4_control}
|\Lambda_3|=\left|\sum_{E\in\CT_h} a_h^E(\boldsymbol{\rho}_h,\boldsymbol{\rho}_h)- a^E(\boldsymbol{\rho}_h,\boldsymbol{\rho}_h)\right|\\
=\left|\sum_{E\in\CT_h} a_h^E(\boldsymbol{\rho}_h-\Pi_h^E\boldsymbol{\rho},\boldsymbol{\rho}_h)- a^E(\boldsymbol{\rho}_h-\Pi_h^E\boldsymbol{\rho},\boldsymbol{\rho}_h)\right|\\
=\left|\sum_{E\in\CT_h} a_h^E(\boldsymbol{\rho}_h-\Pi_h^E\boldsymbol{\rho},\boldsymbol{\rho}_h-\Pi_h^E\boldsymbol{\rho})- a^E(\boldsymbol{\rho}_h-\Pi_h^E\boldsymbol{\rho},\boldsymbol{\rho}_h-\Pi_h^E\boldsymbol{\rho})\right|\\
\leq\sum_{E\in\CT_h} \|\boldsymbol{\rho}_h-\Pi_h^E\boldsymbol{\rho}\|_{0,E}^2
\leq\sum_{E\in\CT_h} \|\boldsymbol{\rho}_h-\boldsymbol{\rho}\|_{0,E}^2+\|\Pi_h^E\boldsymbol{\rho}-\boldsymbol{\rho}\|_{0,E}^2\leq Ch^{2\min\{r,k+1\}}.
\end{multline}

Hence, gathering \eqref{eq:T1_control}, \eqref{eq:T2_control} and \eqref{eq:T4_control}, and replacing all these estimates in \eqref{eq:padra}, we conclude the proof.
\end{proof}

\section{Numerical experiments}
\label{sec:numerics}
Given all the previous theoretical results, now the task is to validate the proposed mixed  numerical method. In this we report  a number of  numerical experiments to analyze the performance of the numerical scheme on different geometries and physical configurations.
The tests that we present were performed with a MATLAB code, using different polygonal meshes like the ones presented in Figure \ref{FIG:Meshes}. Let us precise that  the eigenvalue problem that we solve is of the form $\mathbf{A}\mathbf{x}_h=\kappa_h\mathbf{B}\mathbf{x}_h$, where the discrete eigenvalues $\kappa_h$ are obtained with the command \texttt{eigs}. The plots of the different eigenfunctions that we present were also obtained with the visual tools of Matlab.

\begin{figure}[H]
	\begin{center}
		\begin{minipage}{13cm}
			\centering\includegraphics[height=3.8cm, width=3.8cm]{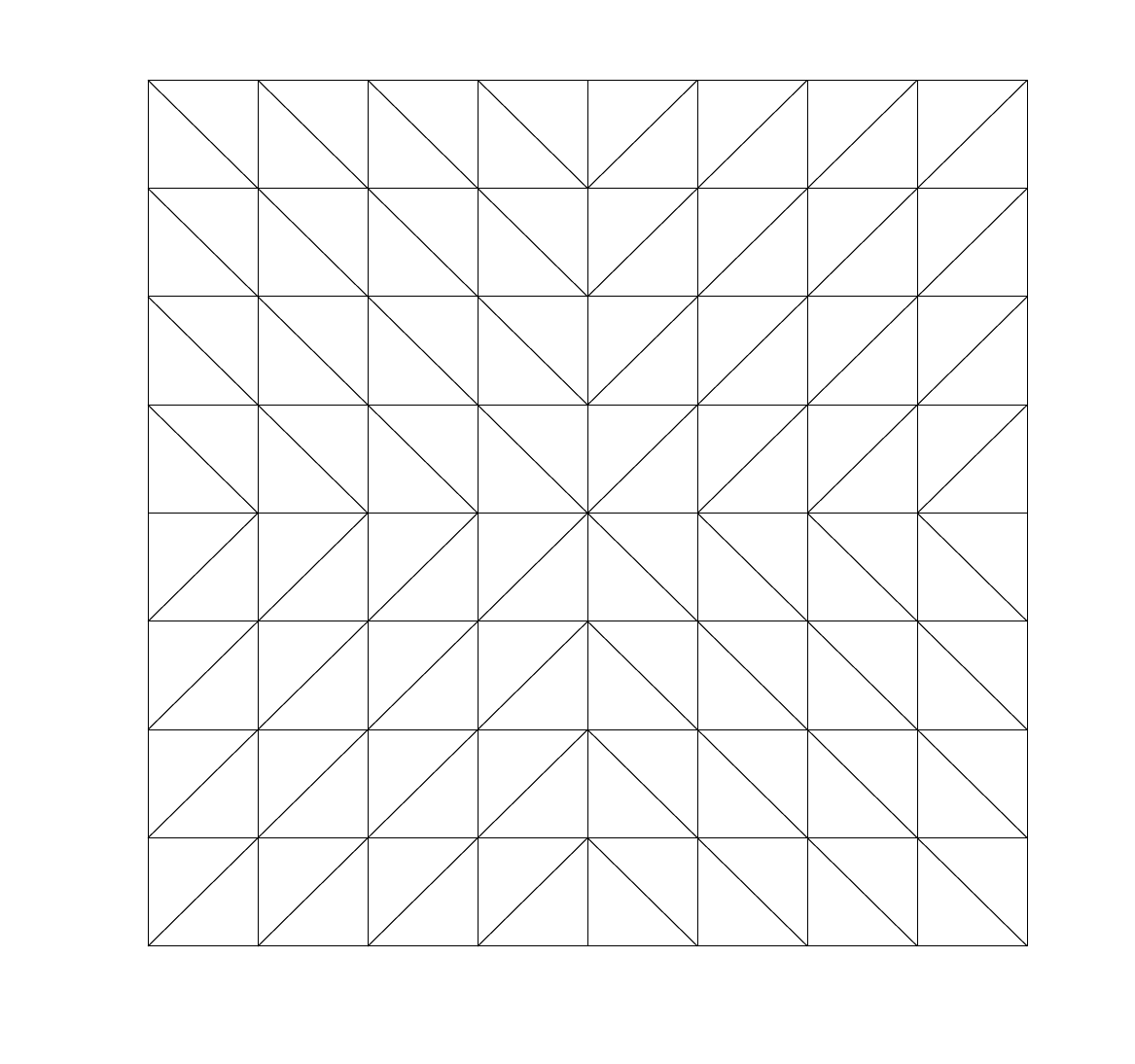}
			\centering\includegraphics[height=3.8cm, width=3.8cm]{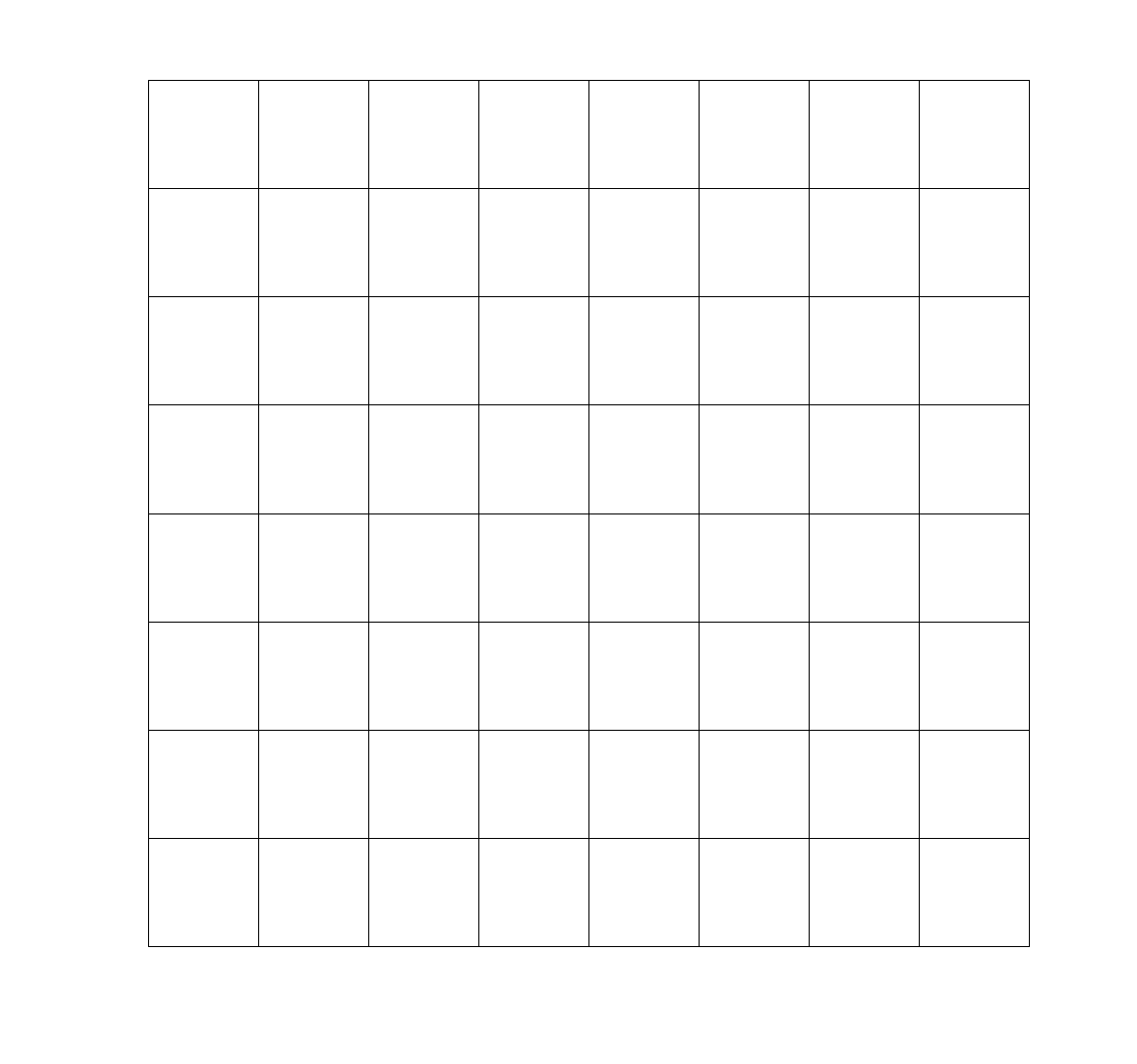}
			\centering\includegraphics[height=3.8cm, width=3.8cm]{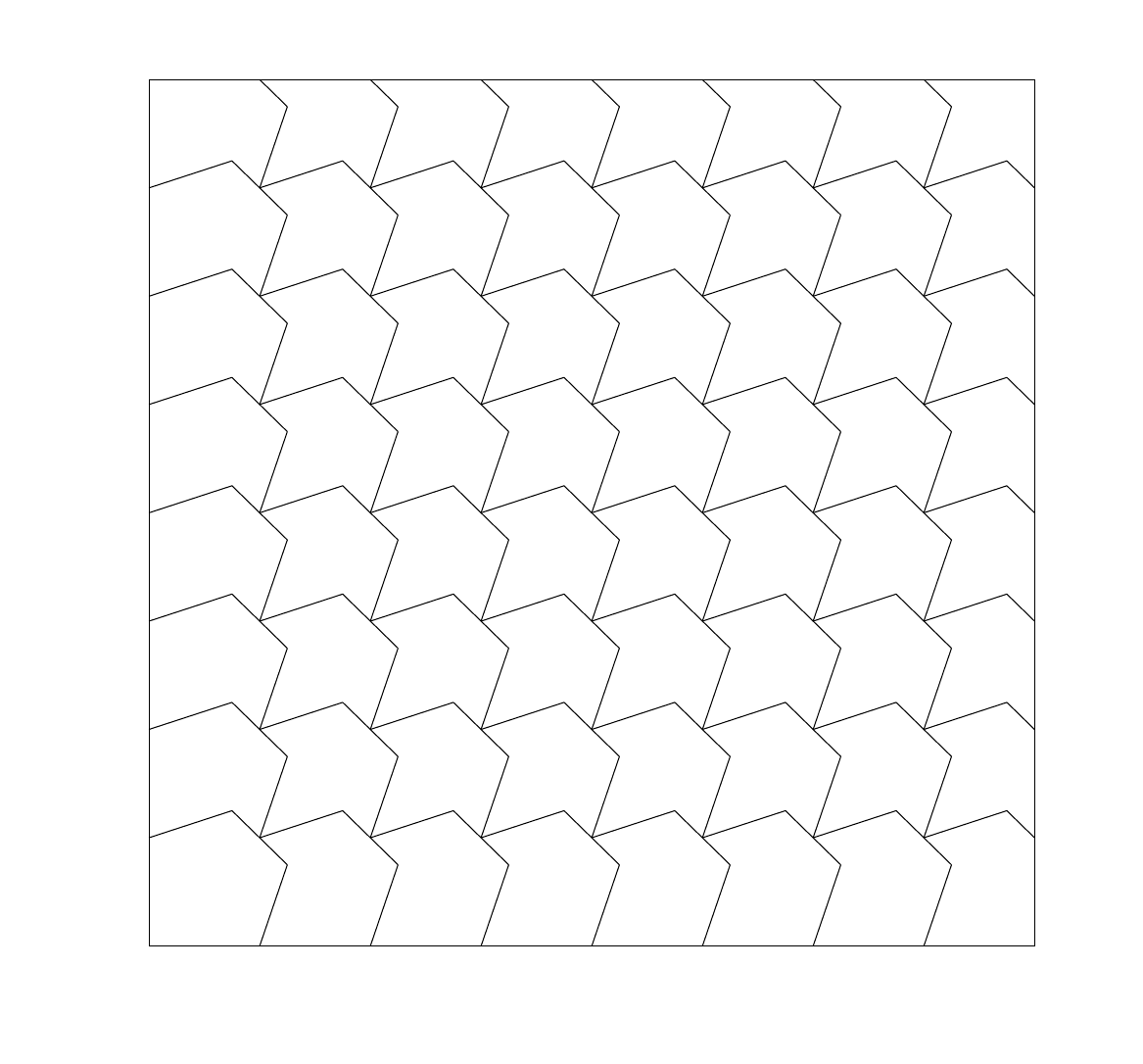}
         \end{minipage}
         	\begin{minipage}{13cm}
			\centering\includegraphics[height=3.8cm, width=3.8cm]{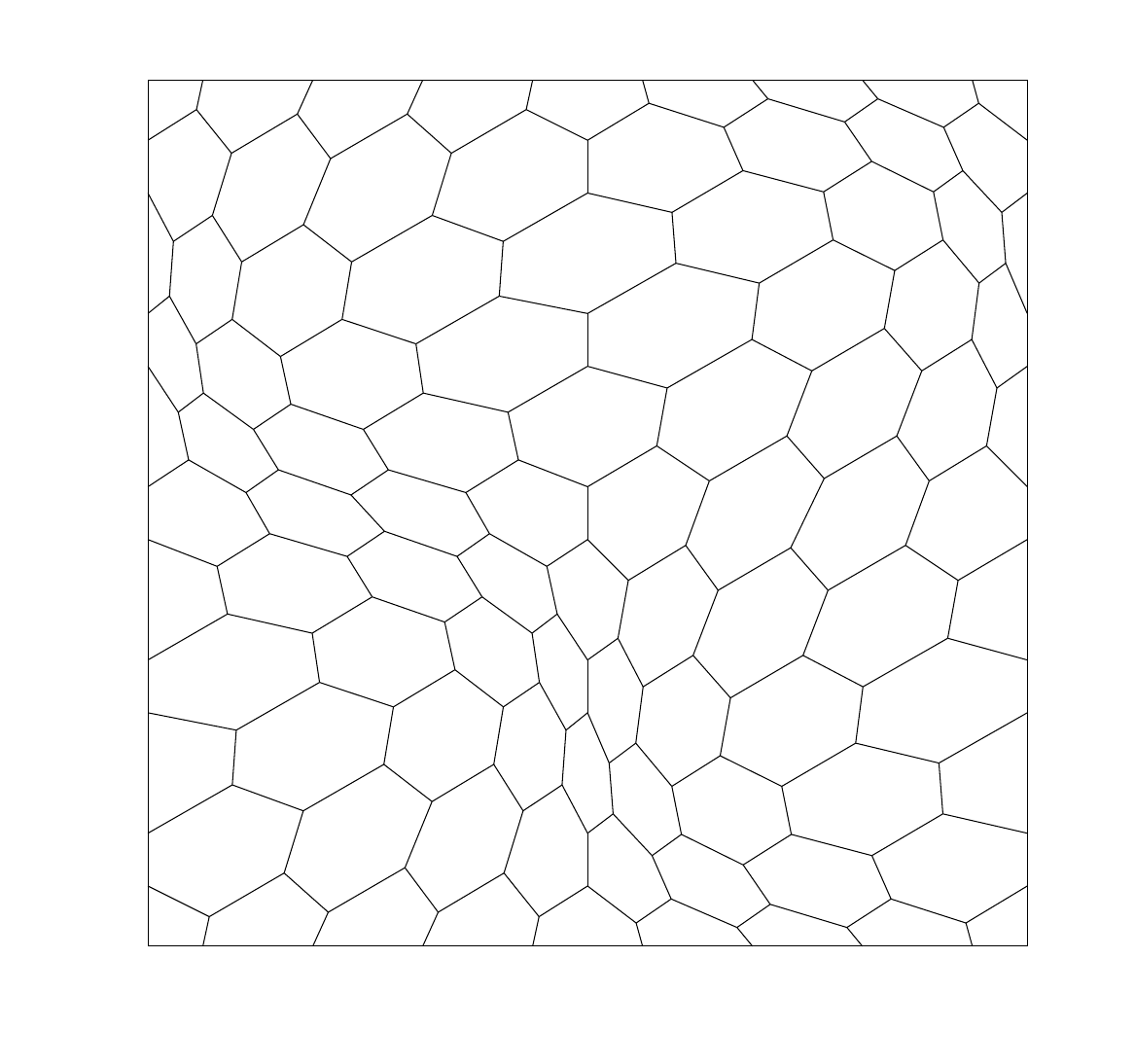}
			\centering\includegraphics[height=3.8cm, width=3.8cm]{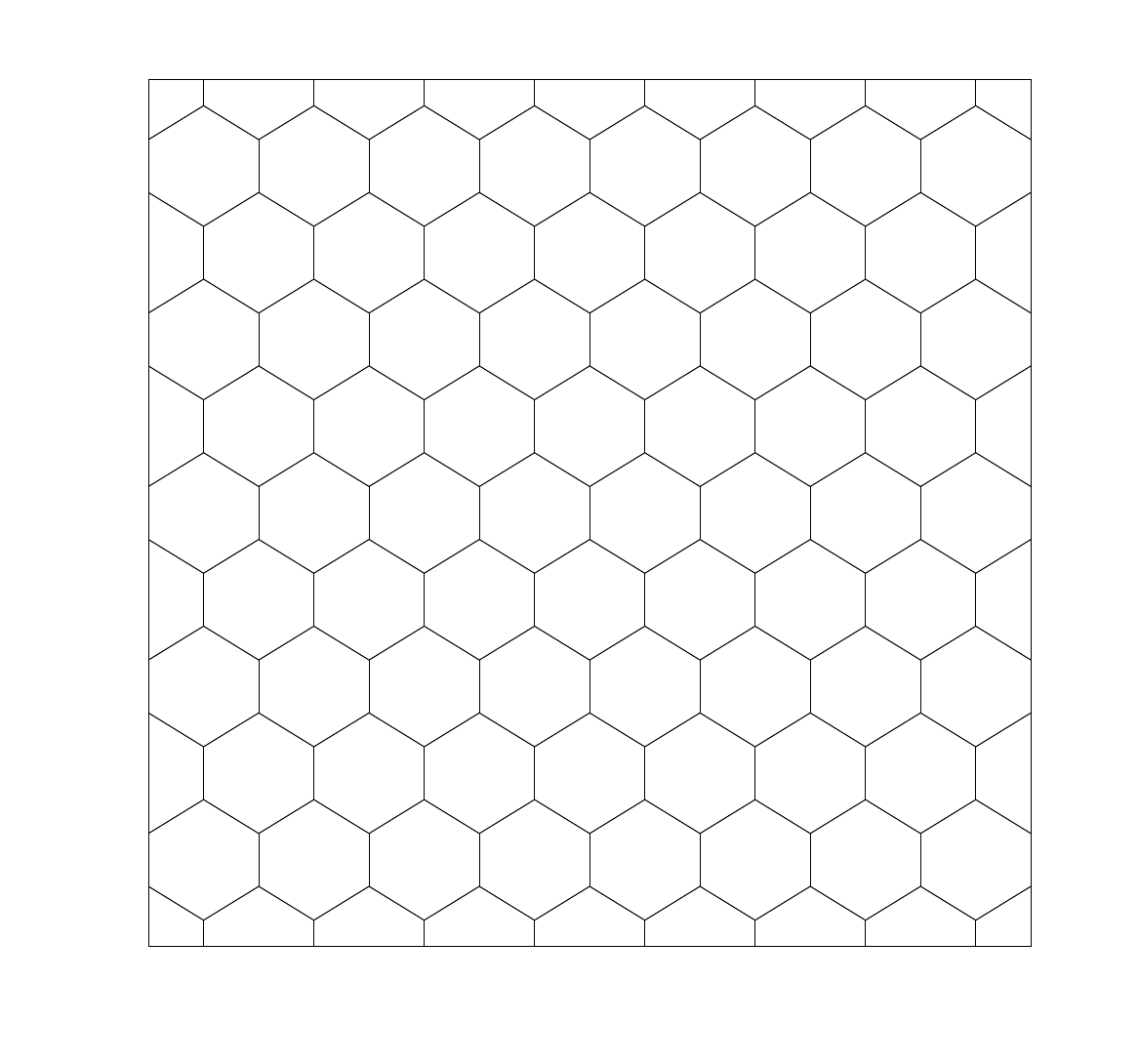}
			\centering\includegraphics[height=3.8cm, width=3.8cm]{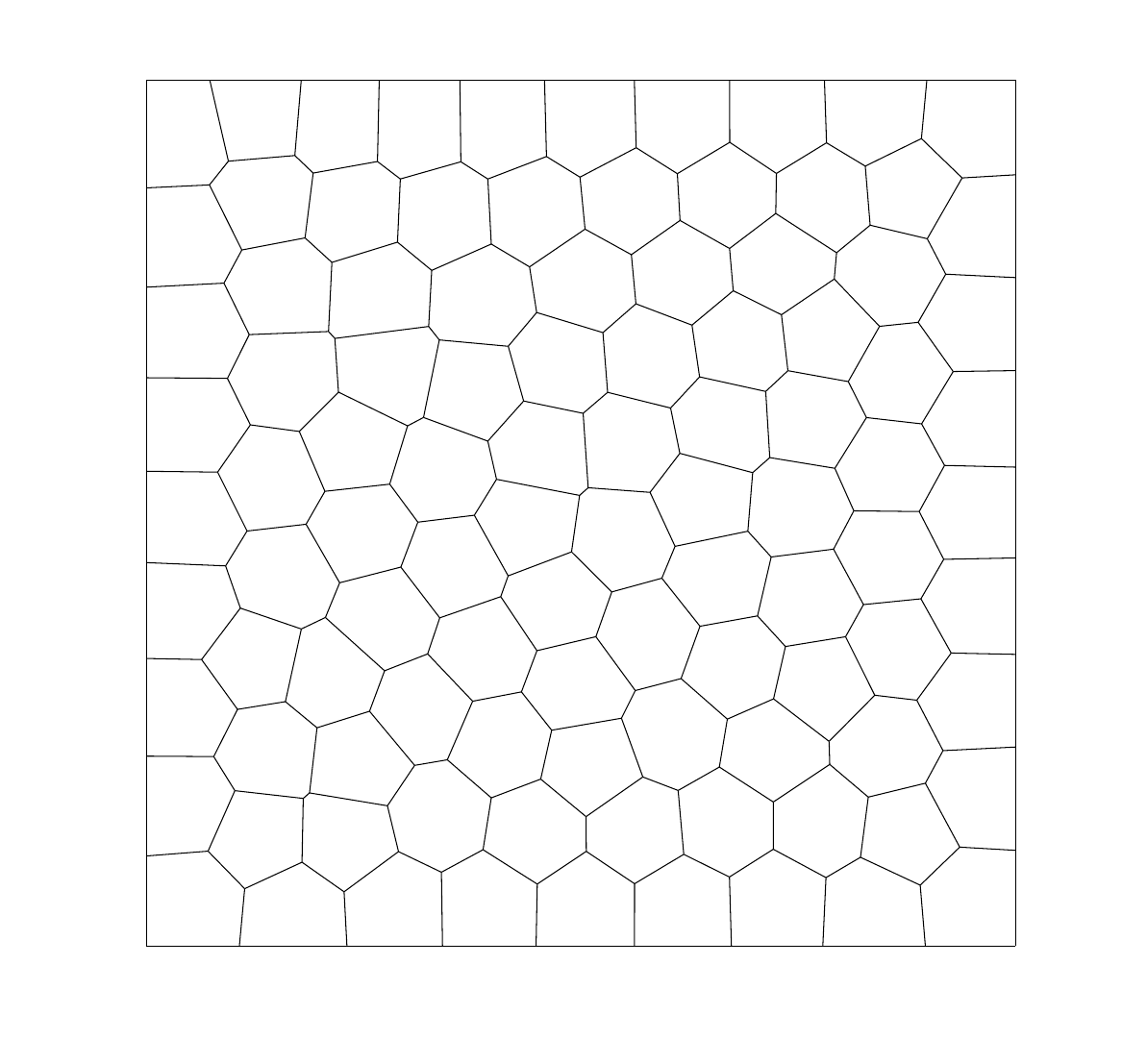}
			         \end{minipage}

			\caption{ Sample meshes: $\CT_{h}^{1}$ (top left), $\CT_{h}^{2}$ (top middle), $\CT_{h}^{3}$ (top right),$\CT_{h}^{4}$ (bottom left), $\CT_{h}^{5}$ (bottom middle) and  $\CT_{h}^{6}$ (bottom right).}
		\label{FIG:Meshes}
	\end{center}
\end{figure}

For the numerical scheme, we have implemented the lowest order of approximation $k=0$. Let us define  $\omega_h:=\sqrt{\kappa_h}$ as the frequency computed with our numerical method. The order of convergence for the frequencies  that we present in the forthcoming tables were obtained with a least-square fitting of the form 
\begin{equation}
\label{eq:fitting}
\omega_h\approx\omega_{\text{extr}}+Ch^{\alpha},
\end{equation}
where $\omega_{\text{extr}}:=\sqrt{\kappa_{\text{extr}}}$ is an extrapolated frequency given by the fitting and $\alpha$ represents the order of convergence. We remark that the constant  $C$ of \eqref{eq:fitting} is independent of the mesh-size.

In the tests, we focus on two specific aspects: on the one hand, the computation of the spectrum and the computational orders of approximation for the eigenvalues, and on the other hand, the assessment of the performance of the method with respect to the Poisson ratio $\nu$. More precisely, recalling the definition of the Lam\'e constants
\begin{equation*}
\lambda:=\frac{E\nu}{(1+\nu)(1-2\nu)}\quad\text{and}\quad\mu:=\frac{E}{2(1+\nu)},
\end{equation*}
clearly when $\nu\rightarrow 1/2$, the constant $\lambda$ blows up. Hence, we aim to observe that the mixed scheme is locking-free with respect to $\nu$. In fact, when $\nu=1/2$, which corresponds to the perfect incompressible case, the bilinear form $a(\cdot,\cdot)$ must be modified in the sense that
\begin{equation}
\label{eq:bilinear_limit}
a(\boldsymbol{\rho},\btau):=\frac{1}{\mu}\int_{\O}\boldsymbol{\rho}^{\texttt{d}}:\btau^{\texttt{d}}\quad\forall\boldsymbol{\rho},\btau\in\boldsymbol{\mathcal{H}}_0,
\end{equation} 
and hence, the computational code must be coherent with this modification.

Regarding the VEM scheme, as in \cite{MR2997471}, a choice for $S^{E}(\cdot,\cdot)$ is given by
\begin{equation*}
S^{E}(\boldsymbol{\rho}_{h},\btau_{h}):=\gamma_E\sum_{k=1}^{N_{E}^{k}}m_{i,E}(\boldsymbol{\rho}_{h})m_{i,E}(\btau_{h}),
\end{equation*}
where the set $\left\{m_{i,E}(\boldsymbol{\rho}_{h})\right\}_{i=1}^{N_{E}^{k}}$ corresponds to all the $E$-moments of $\boldsymbol{\rho}_{h}$ associated with degrees of freedom \eqref{eq:dof_normal}--\eqref{eq:dof_rot} and $\gamma_E$ denotes the stability constant, which is assumed to be of order one. Finally,  the refinement parameter $N$ used to label each mesh is the number of elements on each edge.
%%%%%%%%%%%%%%%%%%%%%%%%%%%%%%%%%%%%%%%%%%%%%%%%%%%%%%%%%%%%%
%%%%%%%%%%%%%%%%%%%%%%%%%%%%%%%%%%%%%%%%%%%%%%%%%%%%%%%%%%%%%
%%%%%%%%%%%%%%%%%%%%%%%%%%%%%%%%%%%%%%%%%%%%%%%%%%%%%%%%%%%%%
%%%%%%%%%%%%%%%%%%%%%%%%%%%%%%%%%%%%%%%%%%%%%%%%%%%%%%%%%%%%%
%%%%%%%%%%%%%%%%%%%%%%%%%%%%%%%%%%%%%%%%%%%%%%%%%%%%%%%%%%%%%
%%%%%%%%%%%%%%%%%%%%%%%%%%%%%%%%%%%%%%%%%%%%%%%%%%%%%%%%%%%%%
%%%%%%%%%%%%%%%%%%%%%%%%%%%%%%%%%%%%%%%%%%%%%%%%%%%%%%%%%%%%%
%%%%%%%%%%%%%%%%%%%%%%%%%%%%%%%%%%%%%%%%%%%%%%%%%%%%%%%%%%%%%
\subsection{Test 1: The unit square}
We begin by testing our method in the unit square $\O:=(0,1)^2$. The boundary condition for this test is $\bu=\boldsymbol{0}$, that represents a clamped square. With this geometrical configuration, the eigenfunctions of the problems are sufficiently regular and hence, the order of convergence for the corresponding eigenvalues must be optimal. 

 We prove different meshes and different values of the Poisson ratio, particularly $\nu\in\{0.35, 0.49, 0.5\}$. Clearly $\nu=0.5$ is the limit case where $\lambda=\infty$ implying that bilinear form \eqref{eq:bilinear_limit} is the one that we consider for this case. In the following tables we report the obtained results.  The parameter $N$ is such that $N\sim h^{-1}$, $\alpha$ is the computed order obtained by  \eqref{eq:fitting} and the last column of each table presents the values obtained by the finite element method studied in \cite{MR4570534}.

\begin{table}[H]
\centering
\caption{Test 1. Computed lowest eigenvalues $\omega_{hi}$, $1\leq i\leq 4$, on different meshes.}
\label{TABLA:COMBINADA}

% --- FILA SUPERIOR: Mallas 1 y 2 ---
\begin{subtable}[t]{0.48\textwidth}
\centering
\caption{$\CT_h^1$ (Triangle meshes)}
\resizebox{\linewidth}{!}{%
\begin{tabular}{@{}c ccc ccc@{}}
\toprule
$\omega_{hi}$ & $N=16$ & $N=32$ & $N=64$ & $\alpha$ & Extrap. & \cite{MR4570534} \\
\midrule
\multicolumn{7}{c}{$\nu=0.35$} \\
\midrule
$\omega_{h1}$ & 4.1220 & 4.1747 & 4.1884 & 1.93 & 4.1933 & 4.1931 \\
$\omega_{h2}$ & 4.1220 & 4.1747 & 4.1884 & 1.93 & 4.1933 & 4.1931 \\
$\omega_{h3}$ & 4.3086 & 4.3559 & 4.3681 & 1.96 & 4.3723 & 4.3722 \\
$\omega_{h4}$ & 5.7435 & 5.8829 & 5.9203 & 1.90 & 5.9340 & 5.9332 \\
\midrule
\multicolumn{7}{c}{$\nu=0.49$} \\
\midrule
$\omega_{h1}$ & 4.1248 & 4.1722 & 4.1845 & 1.95 & 4.1887 & 4.1886 \\
$\omega_{h2}$ & 5.3520 & 5.4748 & 5.5068 & 1.94 & 5.5180 & 5.5176 \\
$\omega_{h3}$ & 5.3520 & 5.4748 & 5.5068 & 1.94 & 5.5180 & 5.5176 \\
$\omega_{h4}$ & 6.3010 & 6.4797 & 6.5272 & 1.91 & 6.5445 & 6.5434 \\
\midrule
\multicolumn{7}{c}{$\nu=0.5$} \\
\midrule
$\omega_{h1}$ & 4.1132 & 4.1607 & 4.1730 & 1.95 & 4.1773 & 4.1771 \\
$\omega_{h2}$ & 5.3776 & 5.4991 & 5.5308 & 1.94 & 5.5419 & 5.5415 \\
$\omega_{h3}$ & 5.3776 & 5.4991 & 5.5308 & 1.94 & 5.5419 & 5.5415 \\
$\omega_{h4}$ & 6.2945 & 6.4735 & 6.5212 & 1.91 & 6.5384 & 6.5373 \\
\bottomrule
\end{tabular}%
}
\end{subtable}%
\hfill
\begin{subtable}[t]{0.48\textwidth}
\centering
\caption{$\CT_h^2$ (Square meshes)}
\resizebox{\linewidth}{!}{%
\begin{tabular}{@{}c ccc ccc@{}}
\toprule
$\omega_{hi}$ & $N=16$ & $N=32$ & $N=64$ & $\alpha$ & Extrap. & \cite{MR4570534} \\
\midrule
\multicolumn{7}{c}{$\nu=0.35$} \\
\midrule
$\omega_{h1}$ & 4.0653 & 4.1589 & 4.1843 & 1.88 & 4.1938 & 4.1931 \\
$\omega_{h2}$ & 4.0653 & 4.1589 & 4.1843 & 1.88 & 4.1938 & 4.1931 \\
$\omega_{h3}$ & 4.2719 & 4.3461 & 4.3656 & 1.93 & 4.3725 & 4.3722 \\
$\omega_{h4}$ & 5.6369 & 5.8512 & 5.9118 & 1.82 & 5.9358 & 5.9332 \\
\midrule
\multicolumn{7}{c}{$\nu=0.49$} \\
\midrule
$\omega_{h1}$ & 4.0880 & 4.1621 & 4.1819 & 1.91 & 4.1890 & 4.1886 \\
$\omega_{h2}$ & 5.2116 & 5.4357 & 5.4967 & 1.88 & 5.5194 & 5.5176 \\
$\omega_{h3}$ & 5.2116 & 5.4357 & 5.4967 & 1.88 & 5.5194 & 5.5176 \\
$\omega_{h4}$ & 6.1470 & 6.4366 & 6.5160 & 1.87 & 6.5458 & 6.5434 \\
\midrule
\multicolumn{7}{c}{$\nu=0.5$} \\
\midrule
$\omega_{h1}$ & 4.0763 & 4.1506 & 4.1704 & 1.91 & 4.1775 & 4.1771 \\
$\omega_{h2}$ & 5.2405 & 5.4609 & 5.5209 & 1.88 & 5.5432 & 5.5415 \\
$\omega_{h3}$ & 5.2405 & 5.4609 & 5.5209 & 1.88 & 5.5432 & 5.5415 \\
$\omega_{h4}$ & 6.1437 & 6.4311 & 6.5101 & 1.86 & 6.5403 & 6.5373 \\
\bottomrule
\end{tabular}%
}
\end{subtable}

\vspace{1.2em} % Espacio entre la fila superior e inferior

% --- FILA INFERIOR: Mallas 3 y 4 ---
\begin{subtable}[t]{0.48\textwidth}
\centering
\caption{$\CT_h^3$ (Non-convex meshes)}
\resizebox{\linewidth}{!}{%
\begin{tabular}{@{}c ccc ccc@{}}
\toprule
$\omega_{hi}$ & $N=16$ & $N=32$ & $N=64$ & $\alpha$ & Extrap. & \cite{MR4570534} \\
\midrule
\multicolumn{7}{c}{$\nu=0.35$} \\
\midrule
$\omega_{h1}$ & 4.0978 & 4.1678 & 4.1866 & 1.90 & 4.1934 & 4.1931 \\
$\omega_{h2}$ & 4.0991 & 4.1679 & 4.1866 & 1.87 & 4.1937 & 4.1931 \\
$\omega_{h3}$ & 4.3113 & 4.3560 & 4.3680 & 1.89 & 4.3725 & 4.3722 \\
$\omega_{h4}$ & 5.7201 & 5.8741 & 5.9178 & 1.82 & 5.9350 & 5.9332 \\
\midrule
\multicolumn{7}{c}{$\nu=0.49$} \\
\midrule
$\omega_{h1}$ & 4.1240 & 4.1711 & 4.1841 & 1.86 & 4.1890 & 4.1886 \\
$\omega_{h2}$ & 5.3046 & 5.4612 & 5.5031 & 1.90 & 5.5185 & 5.5176 \\
$\omega_{h3}$ & 5.3412 & 5.4716 & 5.5059 & 1.93 & 5.5180 & 5.5176 \\
$\omega_{h4}$ & 6.2785 & 6.4715 & 6.5247 & 1.86 & 6.5449 & 6.5434 \\
\midrule
\multicolumn{7}{c}{$\nu=0.5$} \\
\midrule
$\omega_{h1}$ & 4.1121 & 4.1594 & 4.1725 & 1.85 & 4.1776 & 4.1771 \\
$\omega_{h2}$ & 5.3352 & 5.4866 & 5.5274 & 1.89 & 5.5425 & 5.5415 \\
$\omega_{h3}$ & 5.3725 & 5.4972 & 5.5302 & 1.92 & 5.5420 & 5.5415 \\
$\omega_{h4}$ & 6.2742 & 6.4657 & 6.5187 & 1.85 & 6.5391 & 6.5373 \\
\bottomrule
\end{tabular}%
}
\end{subtable}%
\hfill
\begin{subtable}[t]{0.48\textwidth}
\centering
\caption{$\CT_h^4$ (Deformed hexagonal meshes)}
\resizebox{\linewidth}{!}{%
\begin{tabular}{@{}c ccc ccc@{}}
\toprule
$\omega_{hi}$ & $N=9$ & $N=35$ & $N=61$ & $\alpha$ & Extrap. & \cite{MR4570534} \\
\midrule
\multicolumn{7}{c}{$\nu=0.35$} \\
\midrule
$\omega_{h1}$ & 3.9389 & 4.1764 & 4.1876 & 2.00 & 4.1932 & 4.1931 \\
$\omega_{h2}$ & 3.9560 & 4.1766 & 4.1877 & 1.96 & 4.1932 & 4.1931 \\
$\omega_{h3}$ & 4.1836 & 4.3601 & 4.3682 & 2.02 & 4.3722 & 4.3722 \\
$\omega_{h4}$ & 5.3750 & 5.8926 & 5.9197 & 1.92 & 5.9339 & 5.9332 \\
\midrule
\multicolumn{7}{c}{$\nu=0.49$} \\
\midrule
$\omega_{h1}$ & 3.9941 & 4.1755 & 4.1843 & 1.98 & 4.1887 & 4.1886 \\
$\omega_{h2}$ & 5.0303 & 5.4859 & 5.5073 & 2.01 & 5.5177 & 5.5176 \\
$\omega_{h3}$ & 5.0720 & 5.4870 & 5.5075 & 1.97 & 5.5177 & 5.5176 \\
$\omega_{h4}$ & 5.8253 & 6.4912 & 6.5262 & 1.92 & 6.5443 & 6.5434 \\
\midrule
\multicolumn{7}{c}{$\nu=0.5$} \\
\midrule
$\omega_{h1}$ & 3.9818 & 4.1639 & 4.1728 & 1.98 & 4.1772 & 4.1771 \\
$\omega_{h2}$ & 5.0622 & 5.5104 & 5.5314 & 2.01 & 5.5417 & 5.5415 \\
$\omega_{h3}$ & 5.1134 & 5.5116 & 5.5316 & 1.95 & 5.5419 & 5.5415 \\
$\omega_{h4}$ & 5.8206 & 6.4850 & 6.5201 & 1.91 & 6.5387 & 6.5373 \\
\bottomrule
\end{tabular}%
}
\end{subtable}

\vspace{1.2em}

% --- QUINTA TABLA ---
\begin{subtable}[t]{0.48\textwidth}
\centering
\caption{$\CT_h^5$ (Hexagonal meshes)}
\resizebox{\linewidth}{!}{%
\begin{tabular}{@{}c ccc ccc@{}}
\toprule
$\omega_{hi}$ & $N=9$ & $N=35$ & $N=61$ & $\alpha$ & Extrap. & \cite{MR4570534} \\
\midrule
\multicolumn{7}{c}{$\nu=0.35$} \\
\midrule
$\omega_{h1}$ & 3.9006 & 4.1724 & 4.1863 & 1.94 & 4.1934 & 4.1931 \\
$\omega_{h2}$ & 3.9257 & 4.1747 & 4.1871 & 1.97 & 4.1932 & 4.1931 \\
$\omega_{h3}$ & 4.1673 & 4.3583 & 4.3676 & 1.98 & 4.3722 & 4.3722 \\
$\omega_{h4}$ & 5.2953 & 5.8850 & 5.9172 & 1.88 & 5.9347 & 5.9332 \\
\midrule
\multicolumn{7}{c}{$\nu=0.49$} \\
\midrule
$\omega_{h1}$ & 3.9786 & 4.1736 & 4.1836 & 1.93 & 4.1888 & 4.1886 \\
$\omega_{h2}$ & 4.9504 & 5.4795 & 5.5051 & 1.98 & 5.5180 & 5.5176 \\
$\omega_{h3}$ & 5.0553 & 5.4860 & 5.5071 & 1.97 & 5.5178 & 5.5176 \\
$\omega_{h4}$ & 5.7675 & 6.4862 & 6.5244 & 1.90 & 6.5450 & 6.5434 \\
\midrule
\multicolumn{7}{c}{$\nu=0.5$} \\
\midrule
$\omega_{h1}$ & 3.9664 & 4.1620 & 4.1721 & 1.93 & 4.1773 & 4.1771 \\
$\omega_{h2}$ & 4.9898 & 5.5043 & 5.5293 & 1.98 & 5.5418 & 5.5415 \\
$\omega_{h3}$ & 5.1024 & 5.5108 & 5.5313 & 1.95 & 5.5418 & 5.5415 \\
$\omega_{h4}$ & 5.7649 & 6.4800 & 6.5183 & 1.90 & 6.5387 & 6.5373 \\
\bottomrule
\end{tabular}%
}
\end{subtable}

\label{TABLA:1}
\end{table}
From Table \ref{TABLA:1} we observe that the double order of convergence predicted by Theorem \ref{thm:double:order} is attained in each case and is independent of the mesh under consideration, proving that the geometry does not affect the 
convergence rate. Additionally, the method is capable of capturing the spectrum of the problem for values of the Poisson ratio that includes the limit case $\nu=0.5$, proving that the mixed VEM scheme is locking free. Finally, we observe that the extrapolated values that we obtain with our VEM method are close to those of \cite{MR4570534}. In Figure \ref{fig:eigenfunction2} we display plots of some approximated eigenfunctions, obtained with different polygonal meshes.
%\vspace{-0.8cm}
\begin{figure}[H]
	\begin{center}
		\begin{minipage}{13cm}
			\centering\includegraphics[height=4.9cm, width=4.9cm]{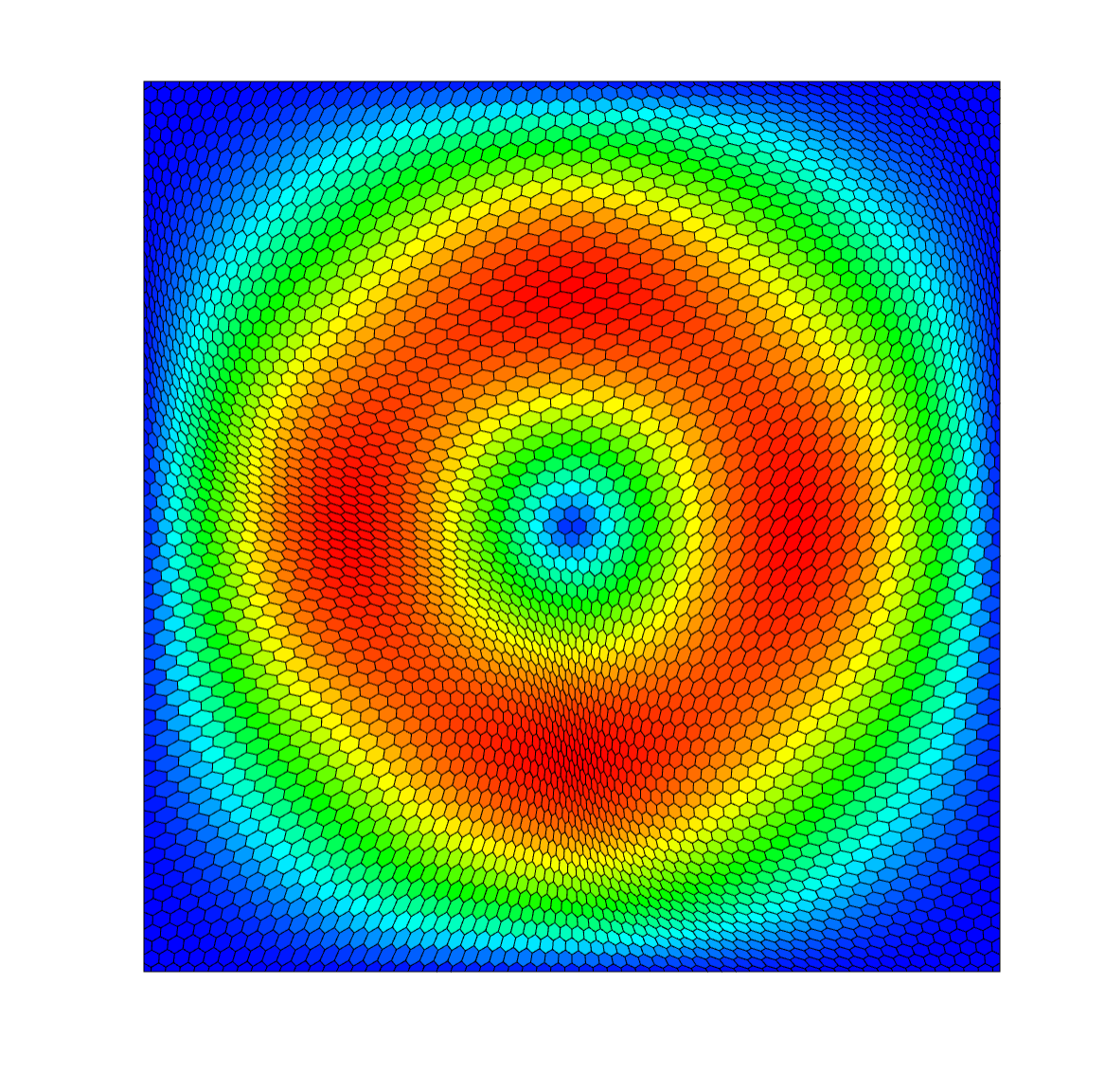}
			\centering\includegraphics[height=4.9cm, width=4.9cm]{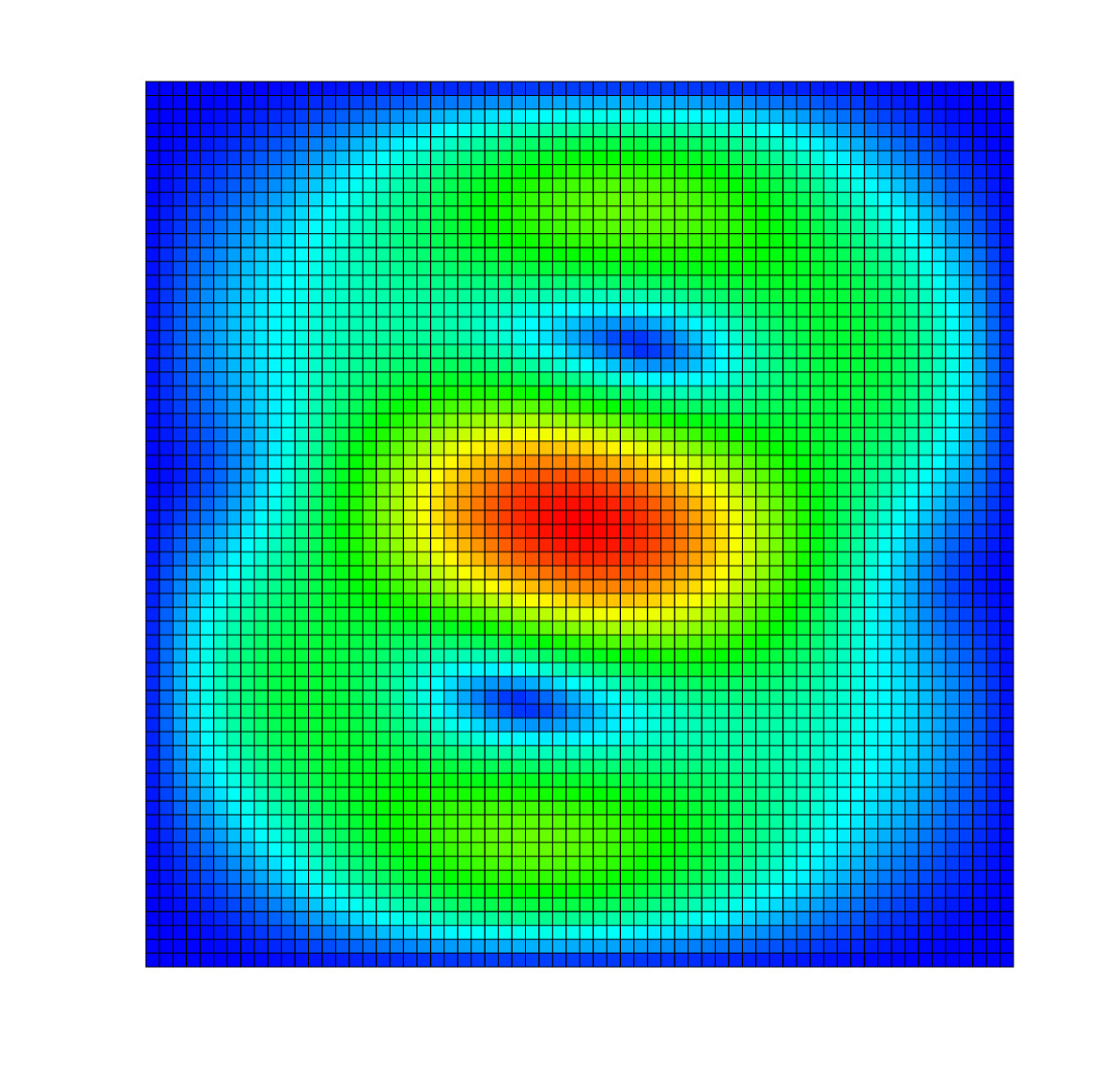}
			\centering\includegraphics[height=4.9cm, width=4.9cm]{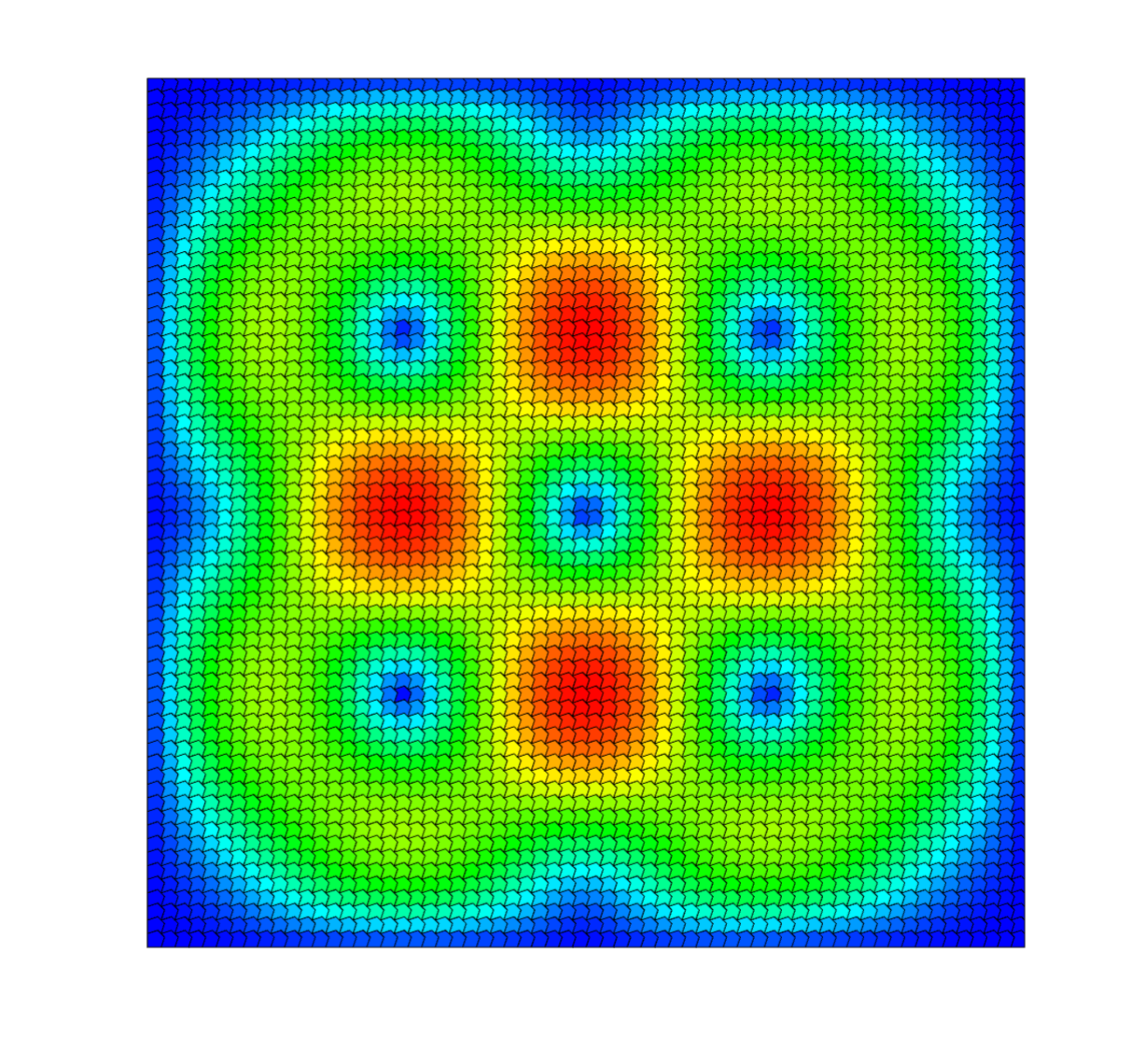}
			\centering\includegraphics[height=4.9cm, width=4.9cm]{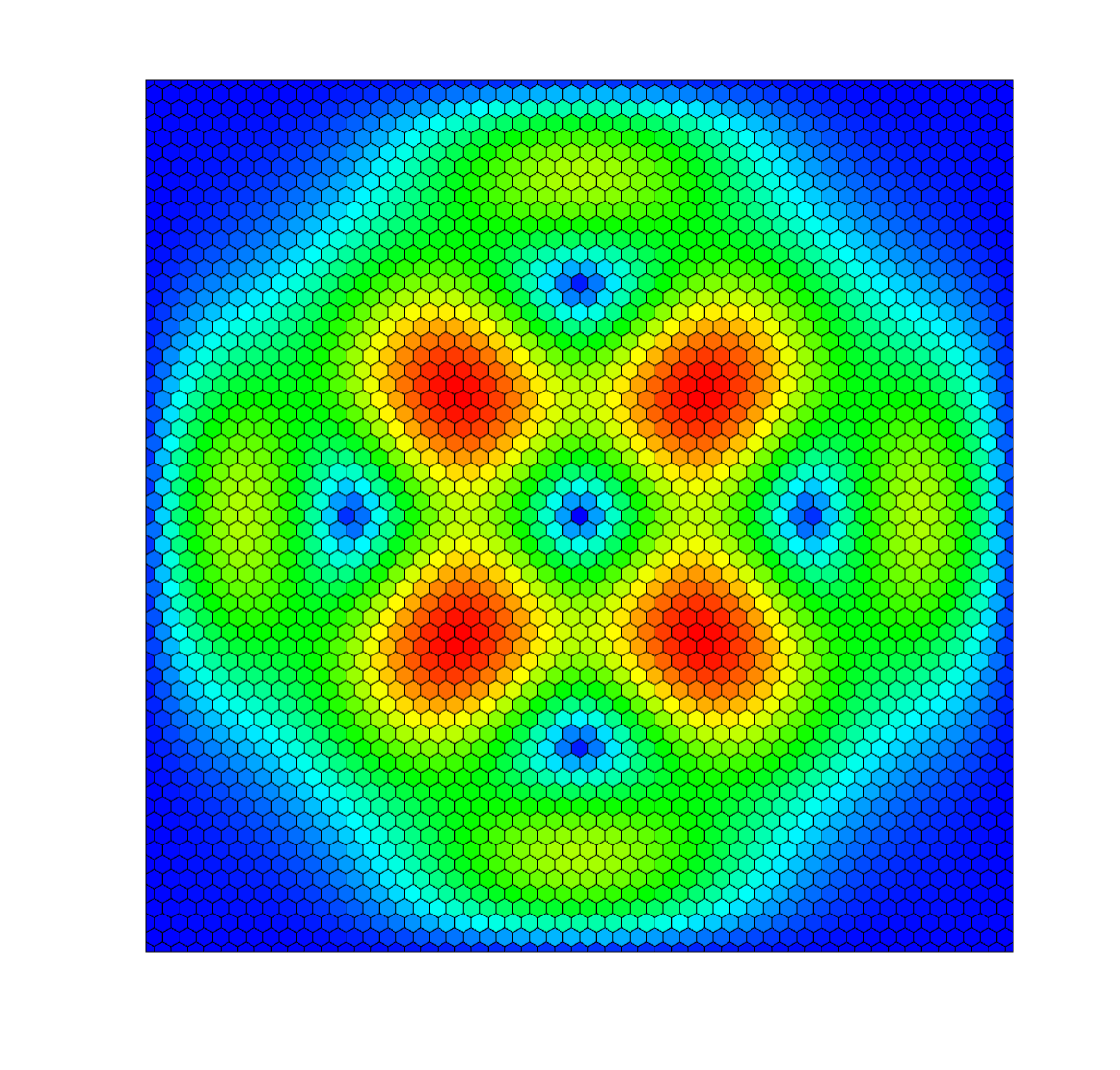}
         \end{minipage}
     				\caption{Test 1. Plots of the computed  magnitude of the first  eigenfunction for $\CT_{h}^{4}$(top left), magnitude of the second eigenfunction (top right) for $\CT_{h}^{2}$, magnitude of the fourth  eigenfunction for $\CT_{h}^{3}$(bottom left) and  magnitude of the fifth  eigenfunction for $\CT_{h}^{5}$(bottom right), with $\nu=0.5$.}
		\label{fig:eigenfunction2}
	\end{center}
\end{figure}
\subsection{Test 2: L-shaped domain}
In this test we consider a non-convex domain that  that we will call the L-shaped domain which is defined by $\O:=(-1,$$1)^2\setminus[-1$,$0]^2$. The importance of this geometry is that the geometrical singularity leads to non-sufficient smooth eigenfunctions for the elasticity eigenvalue problem. This is reflected on the numerical results when the order of convergence is computed. In the forthcoming tables we report the the first five computed frequencies by our method for different meshes and different values of the Poisson ratio,  the corresponding convergence rates, the extrapolated values, and the reference values that we use from reference  \cite{MR4570534}.
\begin{table}[H]
\centering
\caption{Test 2. Computed lowest eigenvalues $\omega_{hi}$, $1\leq i\leq 5$, on distorted square meshes.}
\label{TABLA:2}
% --- FILA 1: Mallas T_h^2 y T_h^4 ---
\begin{subtable}[t]{0.48\textwidth}
\centering
\caption[]{$\CT_h^2$( Square meshes)}
\resizebox{\linewidth}{!}{%
\begin{tabular}{@{}c cccc ccc@{}}
\toprule
$\omega_{hi}$ & $N=8$ & $N=16$ & $N=32$ & $N=64$ & $\alpha$ & Extrap. & \cite{MR4570534} \\
\midrule
\multicolumn{8}{c}{$\nu=0.35$} \\
\midrule
$\omega_{h1}$ & 2.2957 & 2.3530 & 2.3700 & 2.3756 & 1.72 & 2.3778 & 2.3783 \\
$\omega_{h2}$ & 2.6605 & 2.7587 & 2.7866 & 2.7946 & 1.81 & 2.7978 & 2.7976 \\
$\omega_{h3}$ & 3.0150 & 3.1920 & 3.2473 & 3.2680 & 1.62 & 3.2763 & 3.2777 \\
$\omega_{h4}$ & 3.3333 & 3.5383 & 3.5961 & 3.6147 & 1.79 & 3.6211 & 3.6215 \\
$\omega_{h5}$ & 3.4821 & 3.7011 & 3.7622 & 3.7803 & 1.83 & 3.7867 & 3.7866 \\
\midrule
\multicolumn{8}{c}{$\nu=0.49$} \\
\midrule
$\omega_{h1}$ & 2.9813 & 3.1669 & 3.2281 & 3.2530 & 1.53 & 3.2637 & 3.2661 \\
$\omega_{h2}$ & 3.1893 & 3.4154 & 3.4787 & 3.5005 & 1.78 & 3.5073 & 3.5085 \\
$\omega_{h3}$ & 3.4038 & 3.6302 & 3.6925 & 3.7108 & 1.84 & 3.7174 & 3.7172 \\
$\omega_{h4}$ & 3.6783 & 3.9394 & 4.0150 & 4.0356 & 1.80 & 4.0447 & 4.0427 \\
$\omega_{h5}$ & 3.7158 & 4.0670 & 4.1688 & 4.2007 & 1.77 & 4.2124 & 4.2125 \\
\midrule
\multicolumn{8}{c}{$\nu=0.50$} \\
\midrule
$\omega_{h1}$ & 2.9803 & 3.1682 & 3.2306 & 3.2563 & 1.52 & 3.2674 & 3.2699 \\
$\omega_{h2}$ & 3.1948 & 3.4200 & 3.4830 & 3.5047 & 1.78 & 3.5114 & 3.5126 \\
$\omega_{h3}$ & 3.4362 & 3.6541 & 3.7145 & 3.7324 & 1.83 & 3.7389 & 3.7388 \\
$\omega_{h4}$ & 3.6773 & 3.9374 & 4.0130 & 4.0336 & 1.80 & 4.0426 & 4.0408 \\
$\omega_{h5}$ & 3.8614 & 4.1650 & 4.2594 & 4.2863 & 1.71 & 4.2994 & 4.2968 \\
\bottomrule
\end{tabular}%
}
\end{subtable}%
\hfill
\begin{subtable}[t]{0.48\textwidth}
\centering
\caption{$\CT_h^4$ (Deformed hexagonal meshes)}
\resizebox{\linewidth}{!}{%
\begin{tabular}{@{}c cccc ccc@{}}
\toprule
$\omega_{hi}$ & $N=9$ & $N=19$ & $N=35$ & $N=45$ & $\alpha$ & Extrap. & \cite{MR4570534} \\
\midrule
\multicolumn{8}{c}{$\nu=0.35$} \\
\midrule
$\omega_{h1}$ & 2.3356 & 2.3668 & 2.3746 & 2.3759 & 1.74 & 2.3785 & 2.3783 \\
$\omega_{h2}$ & 2.7259 & 2.7798 & 2.7924 & 2.7944 & 1.83 & 2.7983 & 2.7976 \\
$\omega_{h3}$ & 3.1510 & 3.2407 & 3.2657 & 3.2698 & 1.60 & 3.2799 & 3.2777 \\
$\omega_{h4}$ & 3.4892 & 3.5882 & 3.6121 & 3.6156 & 1.80 & 3.6233 & 3.6215 \\
$\omega_{h5}$ & 3.6374 & 3.7510 & 3.7768 & 3.7805 & 1.88 & 3.7881 & 3.7866 \\
\midrule
\multicolumn{8}{c}{$\nu=0.49$} \\
\midrule
$\omega_{h1}$ & 3.1261 & 3.2217 & 3.2504 & 3.2555 & 1.48 & 3.2691 & 3.2661 \\
$\omega_{h2}$ & 3.3699 & 3.4718 & 3.4980 & 3.5017 & 1.72 & 3.5110 & 3.5085 \\
$\omega_{h3}$ & 3.5673 & 3.6814 & 3.7072 & 3.7110 & 1.88 & 3.7186 & 3.7172 \\
$\omega_{h4}$ & 3.8569 & 3.9984 & 4.0301 & 4.0350 & 1.89 & 4.0441 & 4.0427 \\
$\omega_{h5}$ & 3.9711 & 4.1497 & 4.1942 & 4.2010 & 1.75 & 4.2162 & 4.2125 \\
\midrule
\multicolumn{8}{c}{$\nu=0.50$} \\
\midrule
$\omega_{h1}$ & 3.1264 & 3.2240 & 3.2535 & 3.2588 & 1.47 & 3.2729 & 3.2699 \\
$\omega_{h2}$ & 3.3752 & 3.4763 & 3.5022 & 3.5060 & 1.72 & 3.5152 & 3.5126 \\
$\omega_{h3}$ & 3.5981 & 3.7049 & 3.7293 & 3.7329 & 1.87 & 3.7401 & 3.7388 \\
$\omega_{h4}$ & 3.8551 & 3.9964 & 4.0281 & 4.0330 & 1.89 & 4.0421 & 4.0408 \\
$\omega_{h5}$ & 4.0646 & 4.2380 & 4.2790 & 4.2857 & 1.81 & 4.2987 & 4.2968 \\
\bottomrule
\end{tabular}%
}
\end{subtable}

\vspace{1.2em} % Espacio entre filas de subtables

% --- FILA 2: Malla T_h^5 ---
\begin{subtable}[t]{0.48\textwidth}
\centering
\caption{$\CT_h^5$ (Hexagonal meshes)}
\resizebox{\linewidth}{!}{%
\begin{tabular}{@{}c cccc ccc@{}}
\toprule
$\omega_{hi}$ & $N=9$ & $N=19$ & $N=35$ & $N=45$ & $\alpha$ & Extrap. & \cite{MR4570534} \\
\midrule
\multicolumn{8}{c}{$\nu=0.35$} \\
\midrule
$\omega_{h1}$ & 2.3341 & 2.3661 & 2.3743 & 2.3757 & 1.70 & 2.3786 & 2.3783 \\
$\omega_{h2}$ & 2.7199 & 2.7777 & 2.7918 & 2.7940 & 1.77 & 2.7987 & 2.7976 \\
$\omega_{h3}$ & 3.1351 & 3.2349 & 3.2638 & 3.2686 & 1.54 & 3.2813 & 3.2777 \\
$\omega_{h4}$ & 3.4694 & 3.5817 & 3.6104 & 3.6144 & 1.73 & 3.6245 & 3.6215 \\
$\omega_{h5}$ & 3.6209 & 3.7456 & 3.7753 & 3.7795 & 1.82 & 3.7889 & 3.7866 \\
\midrule
\multicolumn{8}{c}{$\nu=0.49$} \\
\midrule
$\omega_{h1}$ & 3.1071 & 3.2145 & 3.2481 & 3.2539 & 1.44 & 3.2705 & 3.2661 \\
$\omega_{h2}$ & 3.3423 & 3.4627 & 3.4954 & 3.5001 & 1.65 & 3.5127 & 3.5085 \\
$\omega_{h3}$ & 3.5510 & 3.6756 & 3.7057 & 3.7100 & 1.80 & 3.7197 & 3.7172 \\
$\omega_{h4}$ & 3.8433 & 3.9925 & 4.0282 & 4.0337 & 1.80 & 4.0453 & 4.0427 \\
$\omega_{h5}$ & 3.9457 & 4.1421 & 4.1918 & 4.1994 & 1.73 & 4.2168 & 4.2125 \\
\midrule
\multicolumn{8}{c}{$\nu=0.50$} \\
\midrule
$\omega_{h1}$ & 3.1062 & 3.2163 & 3.2510 & 3.2571 & 1.42 & 3.2748 & 3.2699 \\
$\omega_{h2}$ & 3.3480 & 3.4673 & 3.4997 & 3.5043 & 1.65 & 3.5168 & 3.5126 \\
$\omega_{h3}$ & 3.5807 & 3.6988 & 3.7277 & 3.7318 & 1.78 & 3.7415 & 3.7388 \\
$\omega_{h4}$ & 3.8419 & 3.9905 & 4.0262 & 4.0318 & 1.80 & 4.0432 & 4.0408 \\
$\omega_{h5}$ & 4.0510 & 4.2314 & 4.2766 & 4.2841 & 1.73 & 4.2999 & 4.2968 \\
\bottomrule
\end{tabular}%
}
\end{subtable}

\end{table}

We observe from Table \ref{TABLA:2} that the numerical results agree with both the theoretical analysis and the findings reported in the literature. On the one hand, the presence of the singularity clearly influences the regularity of certain eigenfunctions, in particular the first one. As a consequence, the convergence rate of the corresponding frequency is not optimal, while the remaining frequencies exhibit convergence close to the expected order $\mathcal{O}(h^2)$. Furthermore, the results once again indicate that the method is locking-free, since the limiting case $\nu=0.5$ remains stable. In addition, the type of mesh employed does not affect the convergence order of the computed frequencies, confirming the robustness of the method with respect to the geometry of the mesh elements. 

Finally, we conclude this test by presenting plots of the displacement magnitudes corresponding to the first four eigenfunctions, computed using different polygonal meshes.
\begin{figure}[H]
	\begin{center}
		\begin{minipage}{13cm}
			\centering\includegraphics[height=5.0cm, width=5.0cm]{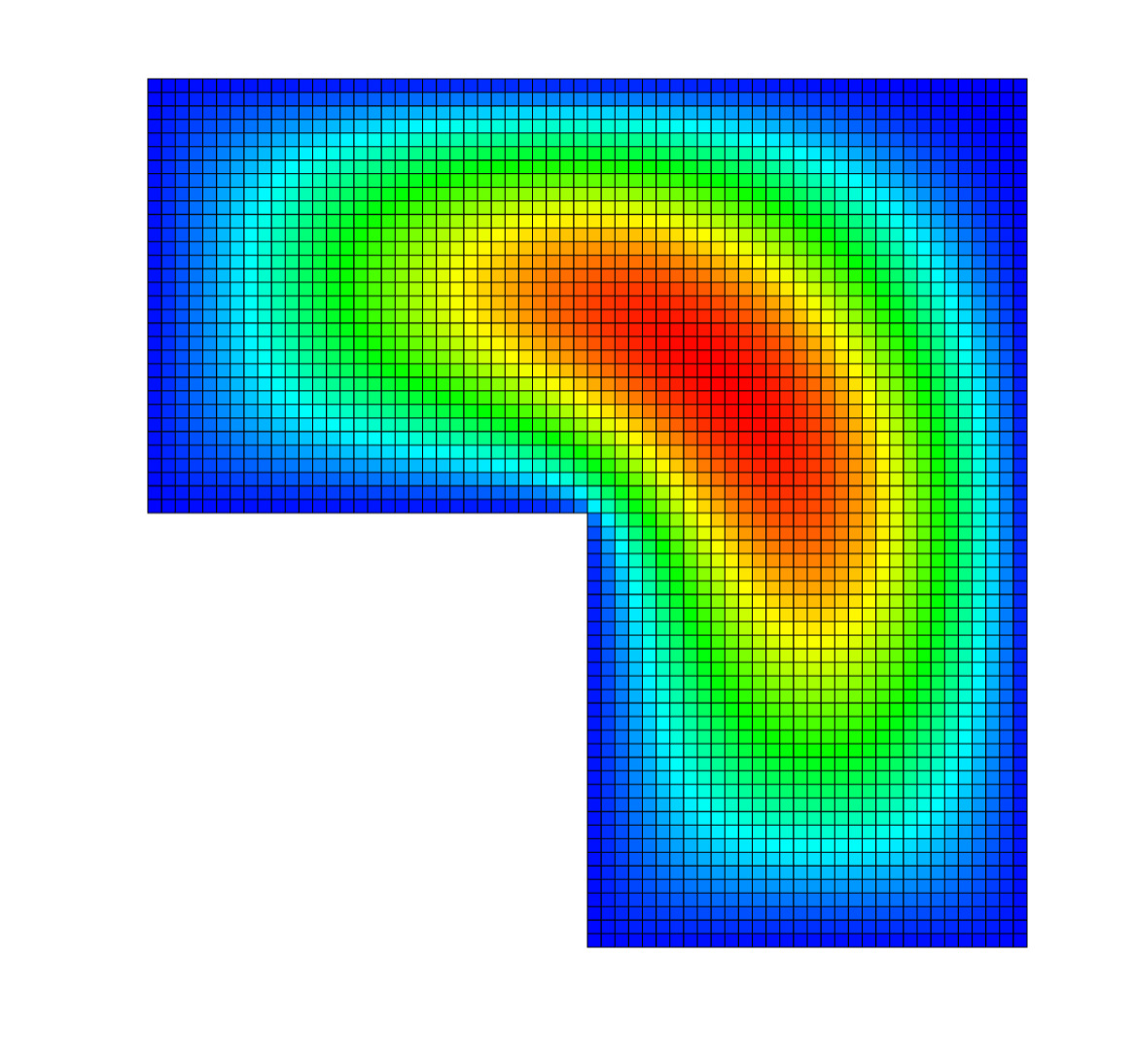}
			\centering\includegraphics[height=5.0cm, width=5.0cm]{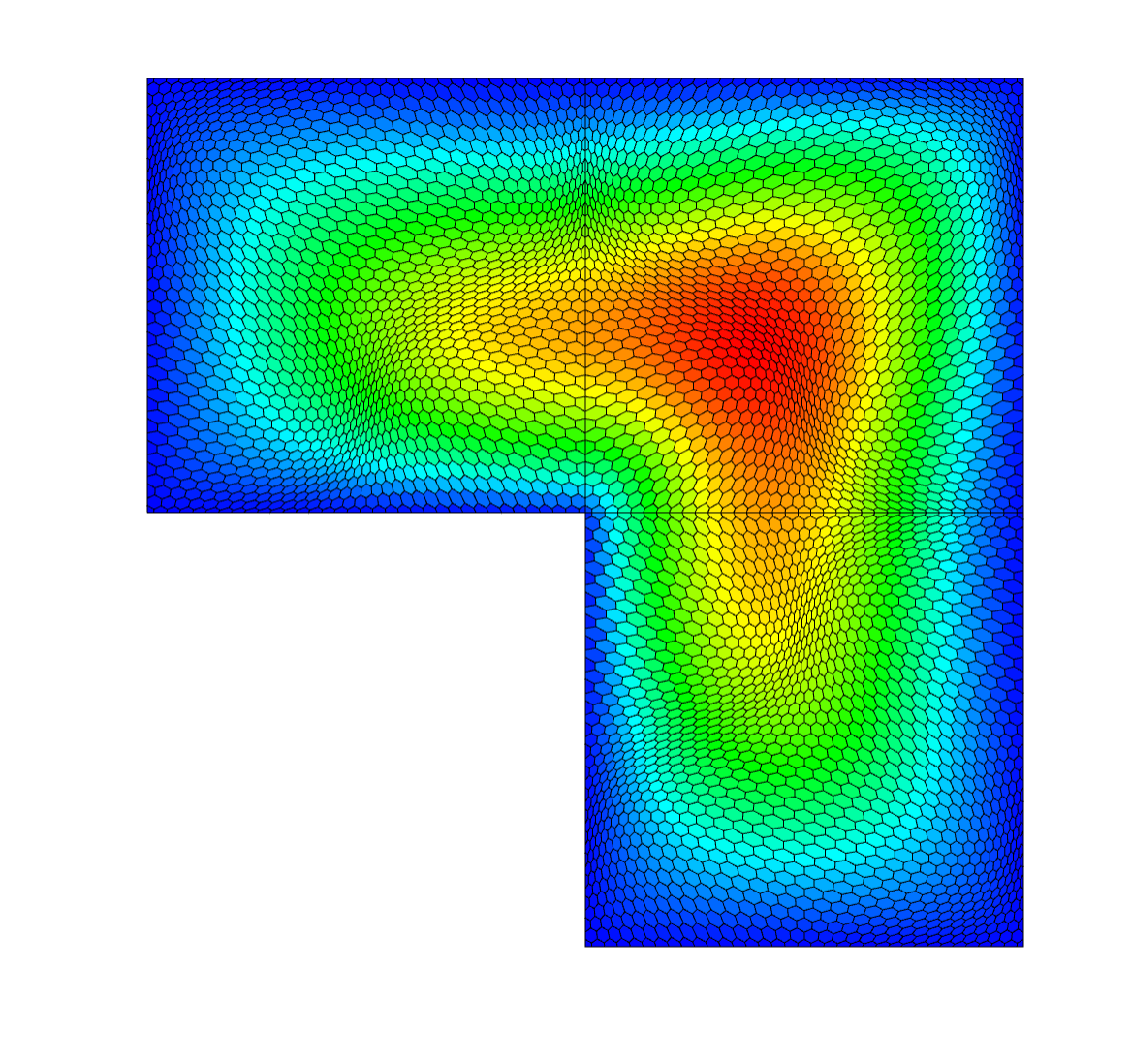}\\
			\centering\includegraphics[height=5cm, width=5cm]{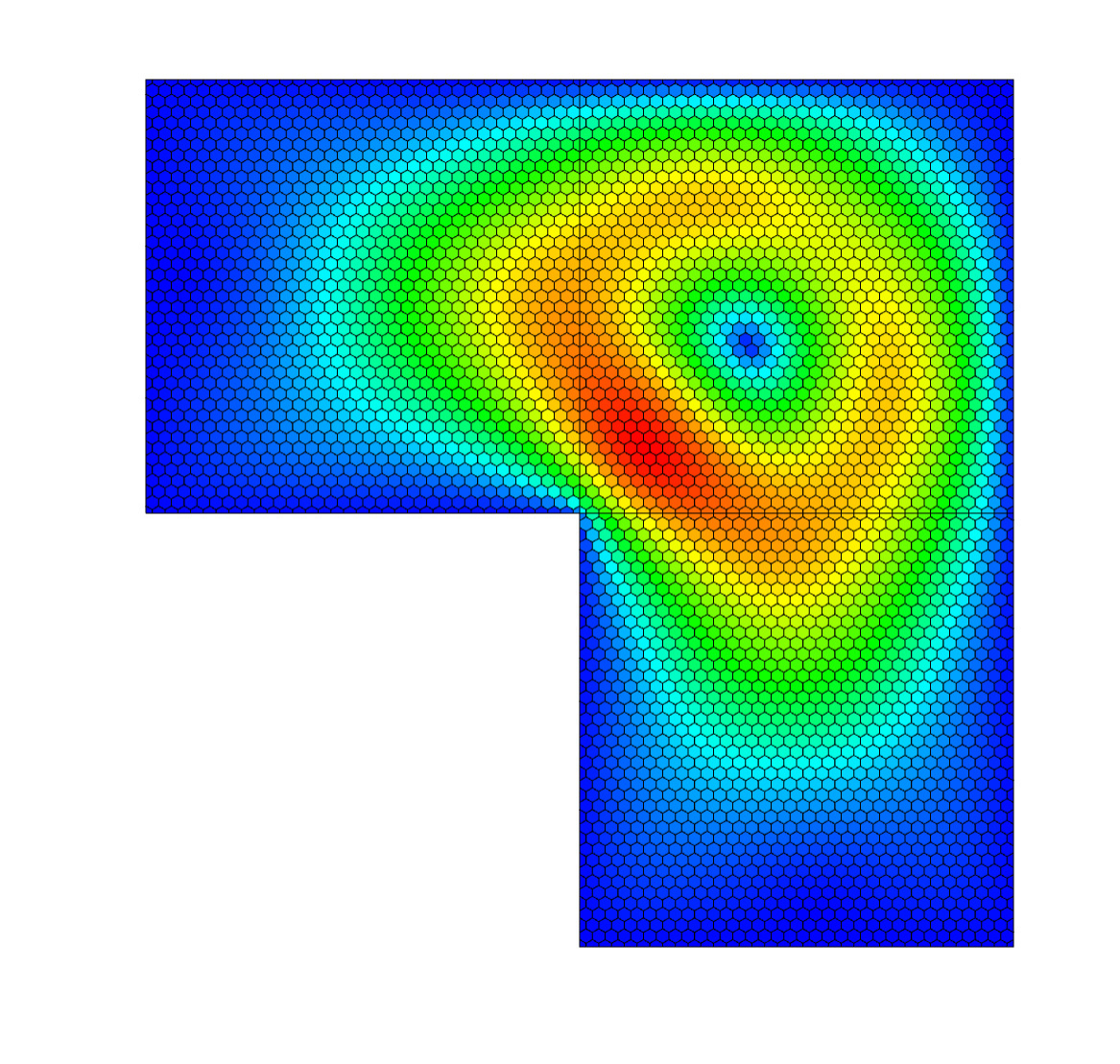}
			\centering\includegraphics[height=5cm, width=5cm]{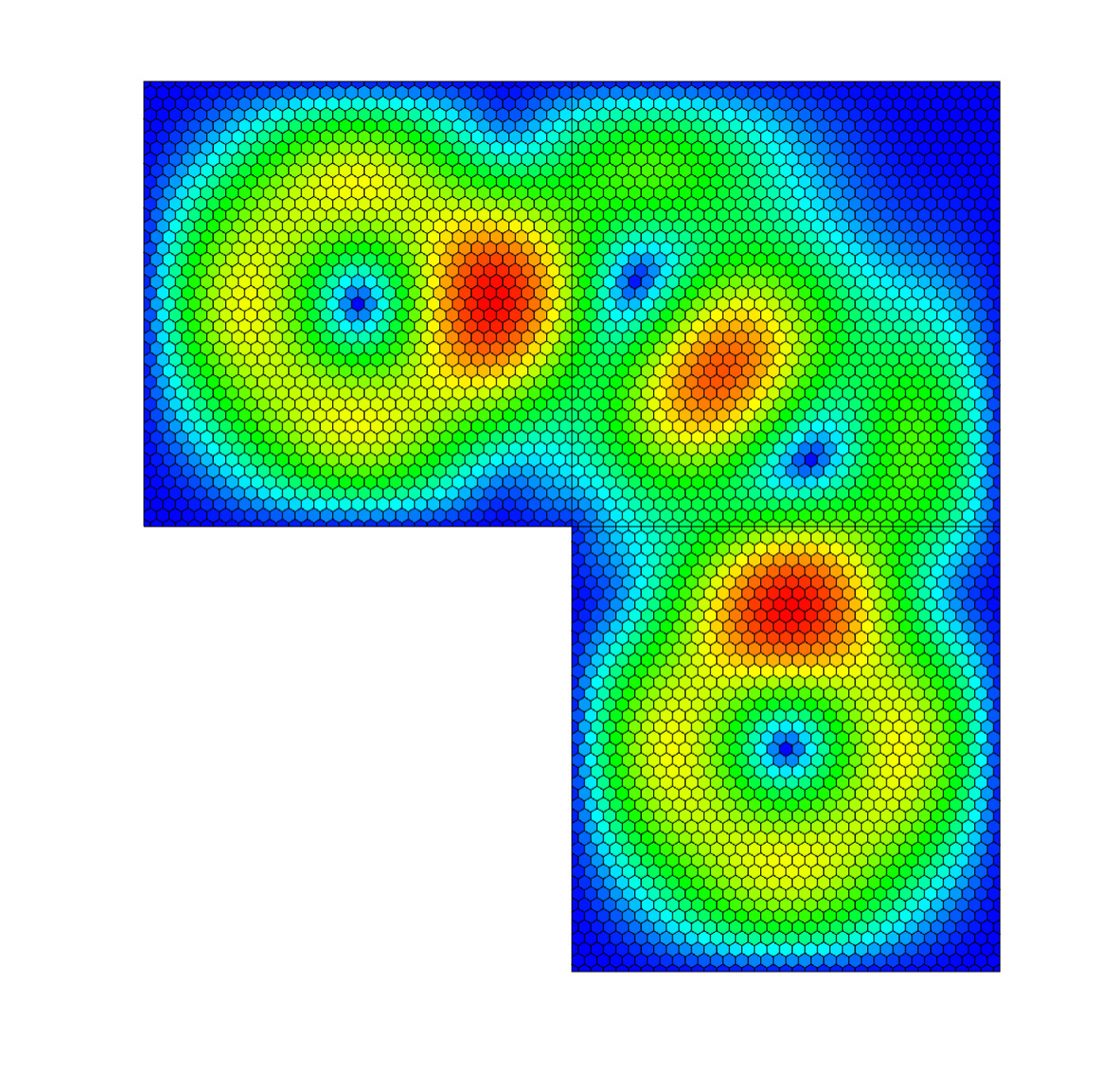}
         \end{minipage}
     				\caption{Test 2. Plots of the computed  first  eigenfunction for $\CT_{h}^{2}$(left), second eigenfunction (right) for $\CT_{h}^{4}$, third  eigenfunction for $\CT_{h}^{5}$(bottom left) and fourth eigenfunction for  $\CT_{h}^{5}$(bottom right). All these plots were obtained for $\nu=0.49$.}
		\label{fig:eigenfunctionL}
	\end{center}
\end{figure}

\subsection{Test 3: Circular domain}
In the next experiment we consider a situation that is not covered by the theoretical analysis, since the domain has a curved boundary. Such a geometry is not ideal for the virtual element method, because the boundary must be approximated by polygonal segments, which introduces a geometric inconsistency in the variational formulation. Nevertheless, the goal of this test is to illustrate that the proposed scheme is still able to reproduce the spectrum of the elasticity eigenvalue problem even under this geometric approximation.

The computational region is the unit disk,
$\Omega:=\{(x,y)\in\mathbb{R}^2:\; x^2+y^2<1\}$.
For this experiment we prescribe the homogeneous Dirichlet condition $\bu=\boldsymbol{0}$ on $\partial\Omega$.

Because the domain is convex, the eigenfunctions of the problem are expected to be sufficiently regular. As a consequence, the method should display a convergence rate close to quadratic order.
\begin{table}[H]
  \centering
  \caption{Test 3. Computed lowest eigenvalues $\omega_{hi}$, $1\leq i\leq 5$, on Voronoi meshes ($\mathcal{T}_h^6$).}
  \label{TABLA:9}
  \begin{tabular}{l *{7}{r}}
    \toprule
    $\omega_{hi}$ & $N=115$ & $N=243$ & $N=357$ & $N=457$ & $\alpha$ & Extrap. & Ref.~\cite{MR4570534} \\
    \midrule
    \multicolumn{8}{c}{ $\nu=0.35$} \\  \midrule
    \addlinespace[2pt]
    $\omega_{h1}$ & 2.3220 & 2.3301 & 2.3313 & 2.3316 & 2.10 & 2.3322 & 2.33234 \\
    $\omega_{h2}$ & 2.3222 & 2.3301 & 2.3313 & 2.3317 & 2.02 & 2.3323 & 2.33190 \\
    $\omega_{h3}$ & 2.3275 & 2.3309 & 2.3315 & 2.3317 & 1.87 & 2.3320 & 2.33234 \\
    $\omega_{h4}$ & 3.3033 & 3.3144 & 3.3162 & 3.3167 & 2.00 & 3.3176 & 3.31761 \\
    $\omega_{h5}$ & 3.3036 & 3.3145 & 3.3162 & 3.3167 & 1.99 & 3.3176 & 3.31761 \\
    \midrule
    \multicolumn{8}{c}{ $\nu=0.49$} \\  \midrule
    \addlinespace[2pt]
    $\omega_{h1}$ & 2.2154 & 2.2187 & 2.2192 & 2.2194 & 2.01 & 2.2197 & 2.21964 \\
    $\omega_{h2}$ & 2.9445 & 2.9550 & 2.9567 & 2.9572 & 1.98 & 2.9581 & 2.95809 \\
    $\omega_{h3}$ & 2.9447 & 2.9550 & 2.9567 & 2.9572 & 1.96 & 2.9581 & 2.95809 \\
    $\omega_{h4}$ & 3.6552 & 3.6767 & 3.6801 & 3.6812 & 2.00 & 3.6829 & 3.68292 \\
    $\omega_{h5}$ & 3.6563 & 3.6767 & 3.6801 & 3.6812 & 1.93 & 3.6831 & 3.68292 \\
    \midrule
    \multicolumn{8}{c}{ $\nu=0.50$} \\
      \midrule
    \addlinespace[2pt]
    $\omega_{h1}$ & 2.2080 & 2.2113 & 2.2118 & 2.2120 & 2.01 & 2.2122 & 2.21223 \\
    $\omega_{h2}$ & 2.9519 & 2.9621 & 2.9637 & 2.9642 & 1.98 & 2.9651 & 2.96505 \\
    $\omega_{h3}$ & 2.9521 & 2.9621 & 2.9637 & 2.9642 & 1.96 & 2.9651 & 2.96505 \\
    $\omega_{h4}$ & 3.6563 & 3.6775 & 3.6808 & 3.6819 & 1.99 & 3.6836 & 3.68358 \\
    $\omega_{h5}$ & 3.6573 & 3.6775 & 3.6808 & 3.6819 & 1.92 & 3.6838 & 3.68358 \\
    \bottomrule
  \end{tabular}
\end{table}

The results reported on Table \ref{TABLA:9} show that the VEM scheme is capable to capture the physical spectrum on a circular domain,  for mesh $\CT_h^6$ and for any value of the Poisson ratio. The results for other type of meshes are similar.   It is clear that  the orders of convergence convergence are the optimal for the selected values of $\nu$, even for the limit case. We present in Figure \ref{fig:eigenfunctioncirculo} plots of the magnitude of the displacement of the first and fifth eigenfunctions.
\begin{figure}[H]
	\begin{center}
		\begin{minipage}{13cm}
			\centering\includegraphics[height=5.0cm, width=5.0cm]{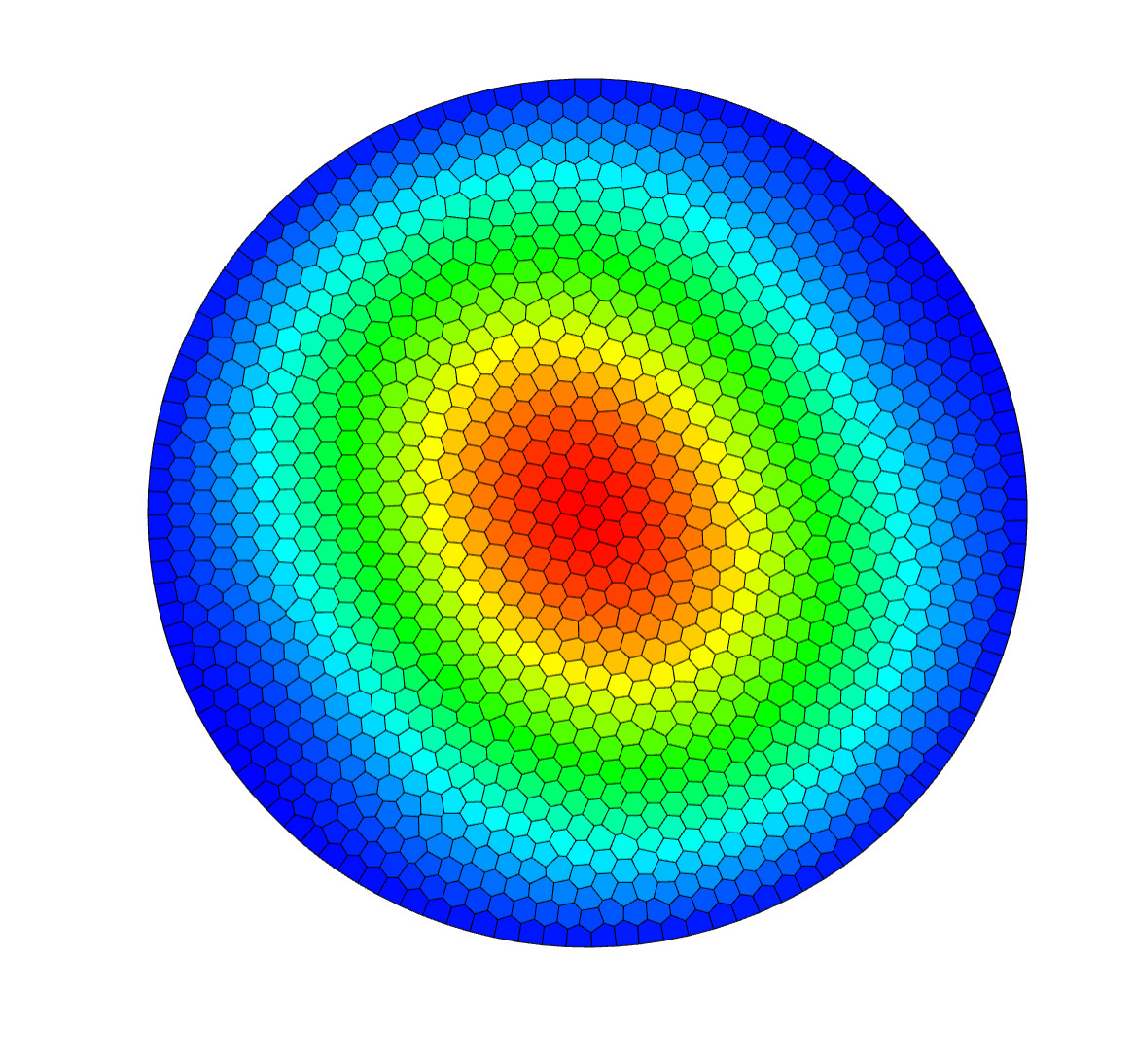}\hspace{1.4cm}
			\centering\includegraphics[height=5.0cm, width=5.0cm]{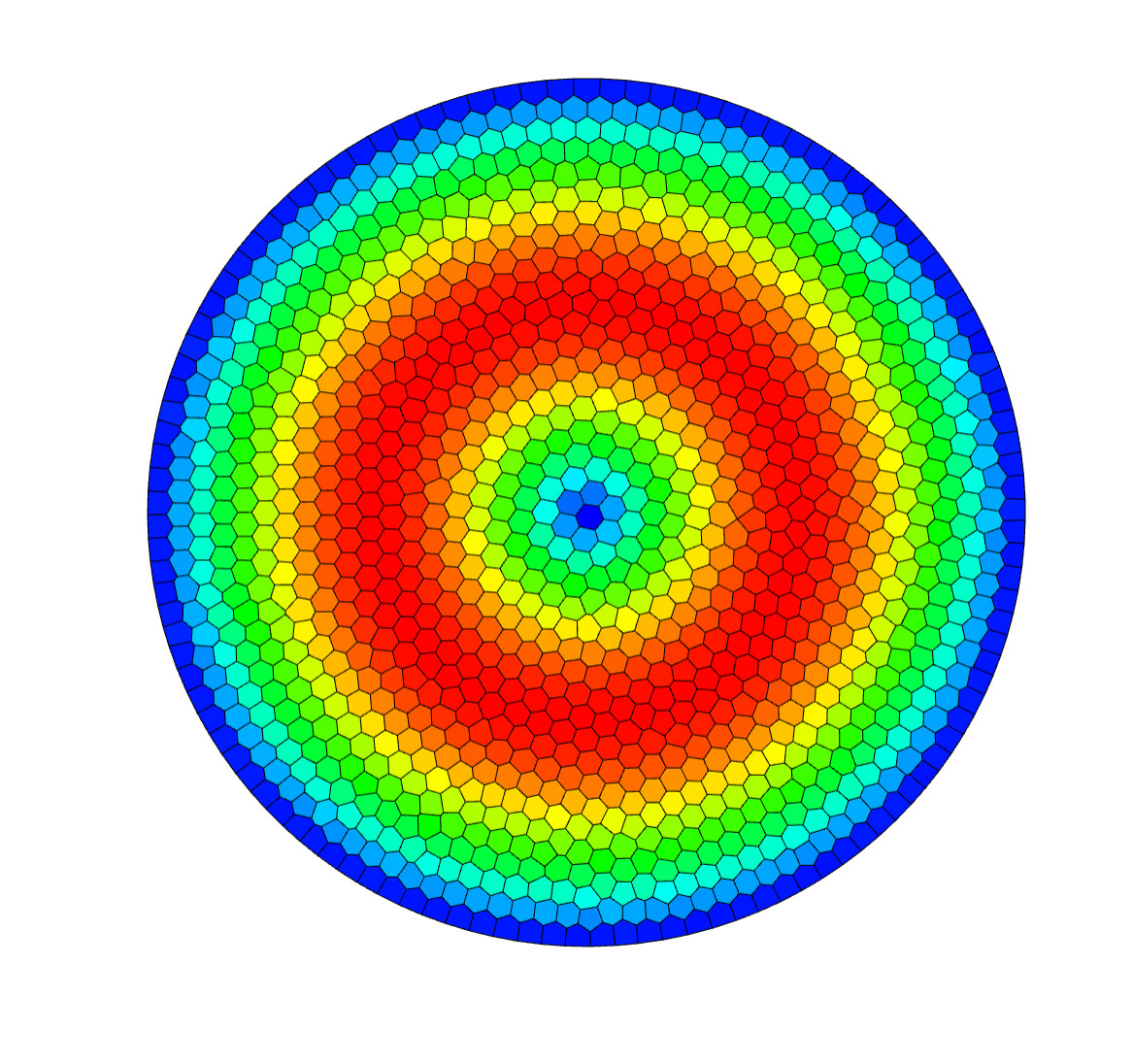}\\
		%	\centering\includegraphics[height=5cm, width=5cm]{Figuras/wh3L.pdf}
		%	centering\includegraphics[height=5cm, width=5cm]{Figuras/wh4L.pdf}
         \end{minipage}
     				\caption{Test 3. Plots of the computed  first   and the second eigenfunction for $\CT_{h}^{6}$ and $\nu=0.35$,}
		\label{fig:eigenfunctioncirculo}
	\end{center}
\end{figure}

\begin{remark}
Since we are considering the lowest order approximation for the method ($k=0$), it is not possible to observe the variational crime introduced by approximating curved boundaries with straight segments. If we  implement a polynomial degree of approximation $k\geq 1$, this phenomenon would become evident as shown in \cite{MR4570534} for the same domain.
\end{remark}
\subsection{Test 4: Effects of the stability constant $\gamma_E$}
%As a final test of this numerical section, we investigate the robustness of our method with respect to the stabilization parameter. \FL{This is important to analyze when the spectrum of operators are approximated.}
As a final test in this numerical section, we investigate the robustness of our method with respect to the stabilization parameter, in order to assess whether the quality of the computations and the accurate approximation of the operator spectrum may be affected by this constant. This type of behavior has already been observed in VEM solutions of different eigenvalue problems.
In particular, it was shown in \cite{zbMATH06443639} that certain VEM discretizations of the Steklov eigenvalue problem introduce spurious eigenvalues, which can be suitably separated from the physical spectrum by an appropriate choice of the stability constant $\gamma_E$.
Similar analyses, also carried out for instance in \cite{MR4658607, MR4177014, MR4229296, MR4050542} within the VEM framework, show analogous results.
In the present case, no spurious eigenvalues were detected for any choice of this constant.
This is due to the boundary condition $\bu=\boldsymbol{0}$ imposed on the whole boundary $\partial\O$.
However, for large values of $\gamma_E$, the eigenvalues computed on coarse meshes may exhibit very poor accuracy.

The scaled version of the stabilization term $S^{E}(\cdot,\cdot)$ is multiplied by a parameter
$\gamma_E$, which is selected as follows: $\gamma_E=\{0,2^{-6},2^{-3},1,2^{3},2^{6}\}$.
For this test, we have selected the mesh $\CT_h^2$, and the idea is to vary its refinement level in order to compute the eigenvalues using the fitting \eqref{eq:fitting}. For this purpose, the refinement is performed with $N=\{20,60,100,140,180\}$.
Since $\O=(0,1)^2$ and considering the physical parameters, we set $E=1$ and the Poisson ratio $\nu=0.49$.
The results of this test are reported in Table \ref{TABLA:10}.
\begin{table}[H]
  \centering
  \caption{Test 4. Computed lowest eigenvalues $\omega_{hi}$, $1\leq i\leq 4$, on square meshes ($\mathcal{T}_h^2$) for $\gamma_E=0$ and $\gamma_E=2^i$ with $i \in \{-6,-3,0,3,6\}$.}
  \label{TABLA:10}
  \begin{tabular}{l *{7}{r}}
    \toprule
    $\omega_{hi}$ & $N=20$ & $N=60$ & $N=100$ & $N=140$ & $N=180$ & $\alpha$ & Ref.~\cite{MR4570534} \\
    \midrule
    \multicolumn{8}{c}{ $\gamma_E = 0$} \\ \midrule
    \addlinespace[2pt]
    $\omega_{h1}$ & 4.1227 & 4.1809 & 4.1858 & 4.1872 & 4.1877 & 2.00 & 4.1886 \\
    $\omega_{h2}$ & 5.3155 & 5.4938 & 5.5090 & 5.5132 & 5.5149 & 2.00 & 5.5176 \\
    $\omega_{h3}$ & 5.3155 & 5.4938 & 5.5090 & 5.5132 & 5.5149 & 2.00 & 5.5176 \\
    $\omega_{h4}$ & 6.2813 & 6.5123 & 6.5321 & 6.5376 & 6.5399 & 1.99 & 6.5434 \\
    \midrule
    \multicolumn{8}{c}{ $\gamma_E = 2^{-6}$} \\ \midrule
    \addlinespace[2pt]
    $\omega_{h1}$ & 4.1227 & 4.1809 & 4.1858 & 4.1872 & 4.1877 & 2.00 & 4.1886 \\
    $\omega_{h2}$ & 5.3155 & 5.4938 & 5.5090 & 5.5132 & 5.5149 & 2.00 & 5.5176 \\
    $\omega_{h3}$ & 5.3155 & 5.4938 & 5.5090 & 5.5132 & 5.5149 & 2.00 & 5.5176 \\
    $\omega_{h4}$ & 6.2813 & 6.5123 & 6.5321 & 6.5376 & 6.5399 & 1.99 & 6.5434 \\
    \midrule
    \multicolumn{8}{c}{ $\gamma_E = 2^{-3}$} \\ \midrule
    \addlinespace[2pt]
    $\omega_{h1}$ & 4.1227 & 4.1809 & 4.1858 & 4.1872 & 4.1877 & 2.00 & 4.1886 \\
    $\omega_{h2}$ & 5.3155 & 5.4938 & 5.5090 & 5.5132 & 5.5149 & 2.00 & 5.5176 \\
    $\omega_{h3}$ & 5.3155 & 5.4938 & 5.5090 & 5.5132 & 5.5149 & 2.00 & 5.5176 \\
    $\omega_{h4}$ & 6.2813 & 6.5123 & 6.5321 & 6.5376 & 6.5399 & 1.99 & 6.5434 \\
    \midrule
    \multicolumn{8}{c}{ $\gamma_E = 2^0$} \\ \midrule
    \addlinespace[2pt]
    $\omega_{h1}$ & 4.1227 & 4.1809 & 4.1858 & 4.1872 & 4.1877 & 2.00 & 4.1886 \\
    $\omega_{h2}$ & 5.3155 & 5.4938 & 5.5090 & 5.5132 & 5.5149 & 2.00 & 5.5176 \\
    $\omega_{h3}$ & 5.3155 & 5.4938 & 5.5090 & 5.5132 & 5.5149 & 2.00 & 5.5176 \\
    $\omega_{h4}$ & 6.2813 & 6.5123 & 6.5321 & 6.5376 & 6.5399 & 1.99 & 6.5434 \\
    \midrule
    \multicolumn{8}{c}{ $\gamma_E = 2^3$} \\ \midrule
    \addlinespace[2pt]
    $\omega_{h1}$ & 3.6366 & 4.1103 & 4.1596 & 4.1736 & 4.1795 & 2.00 & 4.1886 \\
    $\omega_{h2}$ & 4.1729 & 5.2869 & 5.4303 & 5.4724 & 5.4901 & 1.89 & 5.5176 \\
    $\omega_{h3}$ & 4.1729 & 5.2869 & 5.4303 & 5.4724 & 5.4901 & 1.89 & 5.5176 \\
    $\omega_{h4}$ & 4.7196 & 6.2472 & 6.4311 & 6.4852 & 6.5079 & 1.91 & 6.5434 \\
    \midrule
    \multicolumn{8}{c}{ $\gamma_E = 2^6$} \\ \midrule
    \addlinespace[2pt]
    $\omega_{h1}$ & 2.2721 & 3.6723 & 3.9722 & 4.0723 & 4.1165 & 2.00 & 4.1886 \\
    $\omega_{h2}$ & 2.3031 & 4.2499 & 4.9189 & 5.1816 & 5.3053 & 1.65 & 5.5176 \\
    $\omega_{h3}$ & 2.3031 & 4.2499 & 4.9189 & 5.1816 & 5.3053 & 1.65 & 5.5176 \\
    $\omega_{h4}$ & 2.3978 & 4.8192 & 5.7718 & 6.1121 & 6.2710 & 1.67 & 6.5434 \\
    \bottomrule
  \end{tabular}
\end{table}
We observe from Table \ref{TABLA:10} that for small values of $\gamma_E$, the method is capable to attain optimal order of convergence as the theory predicts. Even if $\gamma_E=1$ or $\gamma_E=8$, the spectrum is correctly approximated. We have contrasted our results with the extrapolated values obtained in \cite{MR4570534} for a FEM scheme. On the other hand, if the scaling parameter is larger, the order of convergence begins to deteriorate. Indeed, in the last row of Table \ref{TABLA:10}, the first eigenvalue es approximated with optimal order, however the rest of the spectrum that we report, is not approximated with the expected order. Moreover, the eigenvalues computed with highly refined meshes are not similar to the reference values in \cite{MR4570534}. These results are  consistent with those reported in \cite[Table 4]{MR3660301} where a similar phenomenon is observed. 
\section{Conclusions}
We have proposed, implemented, and analyzed a mixed virtual element method to approximate the physical spectrum
of the two-dimensional elasticity eigenvalue problem. The formulation leads to a locking-free method that allows the consideration of perfectly incompressible materials. Indeed, the mathematical analysis and the computational results are consistent with the available literature. We have demonstrated that, for ordinary materials (i.e., with positive Poisson ratio), the method is capable of accurately capturing the spectrum of the operator. Moreover, we have studied the effects of the stabilization, showing that for clamped structures this pollution effect can be easily avoided. All these results motivate the study of the three-dimensional case with more general boundary conditions, such as mixed conditions.

\section*{Declarations}
The authors have no competing interests to declare that are relevant to the content of this article.  Data sharing not applicable to this article as no datasets were generated or analysed during the current study.
\section*{Availability of Data and Materials} 
The datasets generated and/or analyzed during the current study are
available from the corresponding author on reasonable request.

\end{document}